\documentclass[11pt]{article}

\usepackage{array}
\usepackage{amsfonts}
\usepackage{amscd}
\usepackage{amssymb}
\usepackage{amsthm}
\usepackage{amsmath}
\usepackage{stmaryrd}
\usepackage{graphicx}
\usepackage{color}
\usepackage{verbatim}
\usepackage{mathrsfs}
\usepackage{pstricks,tikz}
\usetikzlibrary{calc}
\usepackage{caption}
\usepackage{float}
\usepackage[neveradjust]{paralist}
\usepackage{ytableau}
\usepackage{hyperref}
\usepackage{subcaption}

\let\OLDthebibliography\thebibliography
\renewcommand\thebibliography[1]{
  \OLDthebibliography{#1}
  \setlength{\parskip}{0pt}
  \setlength{\itemsep}{0pt plus 0.3ex}
}

\newlength\scratchlength
\newcommand\s[2]{
  \settoheight\scratchlength{\mathstrut}%
  \scratchlength=\number\numexpr\number#1-1\relax\scratchlength
  \lower.5\scratchlength\hbox{\scalebox{1}[#1]{$#2$}}%
}

\definecolor{green1}{RGB}{153,216,201}
\definecolor{green2}{RGB}{44,190,95}

\definecolor{blue1}{RGB}{158,202,225}
\definecolor{blue2}{RGB}{49,130,189}

\definecolor{RedOrange}{cmyk}{0,0.77,0.87,0} 
\definecolor{Mahogany}{cmyk}{0,0.85,0.87,0.35} 
\definecolor{Maroon}{cmyk}{0,0.87,0.68,0.32} 
\definecolor{BrickRed}{cmyk}{0,0.89,0.94,0.28} 
\definecolor{Red}{cmyk}{0,1.,1.,0} 
\definecolor{OrangeRed}{cmyk}{0,1.,0.50,0} 

\definecolor{purple}{rgb}{0.8,0.12,0.8}
\definecolor{orange}{rgb}{1.0,0.7,0.0}
\definecolor{pink}{rgb}{1,0.5,0.8}
\definecolor{blackg}{rgb}{0.1,0.25,0.1}
\definecolor{ForestGreen}{cmyk}{0.91,0,0.88,0.42}
\definecolor{Turquoise}{cmyk}{0.85,0,0.20,0}

\newcommand{\qu}[1]{\quad\text{#1}\quad}

 \theoremstyle{plain}

\newtheorem{thm}{Theorem}[section]
\newtheorem{lemma}[thm]{Lemma}
\newtheorem{prop}[thm]{Proposition}
\newtheorem{cor}[thm]{Corollary}
\newtheorem{conjecture}[thm]{Conjecture}

\theoremstyle{definition}
\newtheorem{defn}[thm]{Definition}
 
\newtheorem{remark}[thm]{Remark}
\newtheorem{example}[thm]{Example}

\numberwithin{equation}{section}

\def\Wext{\widetilde{W}}
\def\Waff{W_{\mathrm{aff}}}
\def\Wfin{W}

\def\Sfin{S_0}
\def\Saff{S}

\def\Hext{\widetilde{H}}

\def\Hfin{H_0}

\def\peq{\preccurlyeq}

\newcommand{\wt}{\mathrm{wt}}

\def\lengths{\mathbf{L}}
\def\hh{h_{\la}}

\def\sp{\mathrm{sp}_0}

\def\sA{\mathsf{A}}
\def\sB{\mathsf{B}}
\def\sC{\mathsf{C}}

\def\sG{\mathsf{G}}

\def\sw{\mathsf{w}}
\def\su{\mathsf{u}}
\def\sm{\mathsf{m}}
\def\bt{\mathbf{t}}

\def\st{\mathsf{t}}

\def\cA{\mathcal{A}}

\def\cD{\mathcal{D}}

\def\cJ{\mathcal{J}}

\def\cL{\mathcal{L}}

\def \cP{\mathcal{P}}
\def\cQ{\mathcal{Q}}
\def\cR{\mathcal{R}}

\newcommand{\ba}{\mathbf{a}}

\newcommand{\sq}{\mathsf{q}}
\newcommand{\sy}{\mathsf{y}}
\newcommand{\sv}{\mathsf{v}}

\newcommand{\sT}{\mathsf{T}}

\def\sR{\mathsf{R}}
\def\fs{\mathfrak{s}}
\def\fc{\mathfrak{c}}
\def\fC{\mathfrak{C}}

\def\fe{\mathfrak{e}}

\def\clap{C_{\mathsf{w}_{\lambda'}}}

\def\CC{\mathbb{C}}

\def\ZZ{\mathbb{Z}}

\def\Ga{\Gamma}
\def\ga{\gamma}
\def\al{\alpha}
\def\De{\Delta}

\DeclareMathOperator\tr{\mathsf{tr}}
\DeclareMathOperator\Tr{\mathsf{Tr}}

\def\muu{\boldsymbol{\mu}}

\def\la{\lambda}

\def\<{\langle}
\def\>{\rangle}

\makeatletter
\renewcommand{\@makefnmark}{\mbox{\textsuperscript{}}}
\makeatother

\makeatletter
\renewcommand*{\@fnsymbol}[1]{%
  \ensuremath{%
    \ifcase#1%
    \or \dagger 
    \or * 
    \or \mathsection 
    \or \mathparagraph 
    \or \| 
    \or ** 
    \or \dagger\dagger 
    \or \ddagger\ddagger 
    \else\@ctrerr
    \fi
  }%
}
\makeatother

\title{The asymptotic Plancherel formula and Lusztig's asymptotic algebra for $\tilde{\sA}_n$}
\author{Nathan Chapelier-Laget, J\'er\'emie Guilhot\footnote{J\'er\'emie Guilhot passed away on 27th July 2025 while this paper was in the refereeing process. The remaining authors dedicate the paper to his memory. Tributes to J\'er\'emie can be found on the following webpages: \texttt{https://www.idpoisson.fr/hommage-guilhot/}\newline  \texttt{https://mathematical-research-institute.sydney.edu.au/vale-jeremie-guilhot/}}
, 
Eloise Little, James Parkinson\footnote{
All authors are supported by the Australian Research Council Discovery Project~DP200100712. The second author is supported by the Agence Nationale de la Recherche fundings ANR CORTIPOM 21-CE40-001 and ANR JCJC Project ANR-18-CE40-0001.} }
\date{\today}

\begin{document}

\maketitle

\begin{abstract} 
The aim of this paper is to give a new explicit construction of
Lusztig's asymptotic algebra in affine type~$\mathsf{A}$.
To do so, we construct a balanced system of cell modules, prove an
asymptotic version of the Plancherel Theorem and develop a relative
version of the Satake Isomorphism for each two-sided Kazhdan-Lusztig cell.
\end{abstract}

\tableofcontents

\section*{Introduction}

Kazhdan-Lusztig theory plays a fundamental role in the representation theory of Coxeter groups, Hecke algebras, groups of Lie type, and Lie algebras. One of the most fascinating objects in the theory is the \textit{asymptotic algebra} introduced by Lusztig in~\cite{Lus:87}. This algebra is ``simpler'' than the associated Hecke algebra, yet still encapsulates essential features of the representation theory. This apparent simplicity is contrasted by the considerable difficulty one faces in explicitly realising the asymptotic algebra for a given Coxeter group, because on face value it requires a detailed understanding of the entire Kazhdan-Lusztig basis, and the structure constants with respect to this basis.

The asymptotic algebra $\cJ$ is a $\ZZ$-algebra with basis $(\st_w)_{w\in W}$ indexed by the associated Coxeter group $W$ and multiplication defined using coefficients of the structure constants of the Kazhdan-Lusztig basis in the Hecke algebra. The structure of this algebra is intimately related to the notion of Kazhdan-Lusztig cells: for instance, if $\De$ is a two-sided cell of $W$ then $\cJ_\De=\mathrm{span}_{\ZZ}\{\st_w\mid w\in \Delta\}$ is a two-sided ideal of $\cJ$ and in turn, $\cJ$ is a direct sum of $\cJ_\De$ where~$\De$ runs through the set of two-sided cells of~$W$. 

For general affine (equal parameter) type, Lusztig gave a conjectural description of the asymptotic algebra in terms of the Langlands dual group~\cite{Lus:89b}.  In the case of affine type $\sA$, this conjecture is equivalent to showing that, for all two-sided cells $\De_\la$, the subalgebra $\cJ_{\la}=\cJ_{\Delta_{\la}}$ is isomorphic to a full matrix algebra over the representation ring of a certain connected reductive group $F_{\la}$ with Weyl group $G_{\la}$. This was first proved by Xi in a remarkable paper~\cite{Xi:02} in 2002 where he gave the first explicit construction of the asymptotic algebra in affine type~$\sA$ using the notion of chains and antichains, following work of Shi~\cite{Shi:86}. Around the same time Bezrukavnikov and Ostrik~\cite{BO:04} verified Lusztig's conjecture in a modified form for general affine type (see also \cite{BDD:23,Daw:23b}). A further approach to constructing the asymptotic algebra in affine type~$\sA$ was given recently by Kim and Pylyavskyy~\cite{KP:23} using the affine matrix ball construction.

The aim of this paper is to present a new construction of the asymptotic algebra in affine type $\sA$ as a matrix algebra with coefficients in the ring of $G_{\la}$-symmetric functions which we believe is interesting for three main reasons: (1) our construction is based on the combinatorial notion of alcove paths and the representation ring of $F_\la$ appears very naturally via the $G_{\la}$-Schur symmetric functions, (2) an asymptotic version of the Plancherel Theorem is proved, which endows the asymptotic algebra with an inner product, and (3) there are indications that our methods are adaptable to other affine types, and moreover to the unequal parameter case. For instance, similar methods have been used by the second and fourth authors to construct the asymptotic algebra associated to the lowest two-sided cell in general affine type (for all parameters; see \cite[Section~6]{GP:19}) and to construct the full asymptotic algebra for all affine Hecke algebras associated to rank~$2$ root systems (for all parameters; see \cite{GP:19,GP:19b}).

Let $\Wext$ be the extended affine Weyl group of type $\tilde{\sA}_n$ and let $\Hext$ be the associated extended Hecke algebra defined over $\sR = \ZZ[\sq,\sq^{-1}]$. We will denote by $(T_{w})$ the standard basis of $\Hext$ (which reflects the Coxeter structure of $\Hext$) and by $(C_w)$ the Kazhdan-Lusztig basis of $\Hext$ which is at the heart of the definition of Kazhdan-Lusztig (left, right, and two-sided) cells. The description of  cells in $\Wext$ is known by the work of Lusztig and Shi~\cite{Lus:85a,Shi:86}. For a partition $\la$ of $n+1$ we denote by $W_\la$ the standard Young subgroup of $\Wext$ associated to $\la$ and by $\sw_{\la}$ the longest element of~$W_\la$. Let $\Delta_{\la}$ be the two-sided cell containing $\sw_{\la'}$ (where $\la'$ is the transposed partition of $\la$). Then $(\Delta_\la)_{\la\vdash n+1}$ describes the full set of two-sided cells and we further have $\Delta_\la\leq_{LR} \Delta_\mu$ in the two-sided order if and only if $\la\leq \mu$ in the dominance order for partitions. 

Since affine type~$\sA$ is necessarily ``equal parameters'', deep geometric interpretations of Kazhdan-Lusztig theory imply positivity properties in the Hecke algebra, such as the positivity of the coefficients of the Kazhdan-Lusztig polynomials, and positivity of the coefficients of the structure constants $h_{x,y,z}$ with respect to Kazhdan-Lusztig basis. This positivity implies a collection of useful properties in the theory, which have been collected by Lusztig~\cite{Lus:03} in a series of statements now known as P1--P15. These statements capture the essential properties of cells, Lusztig's $\ba$-function (defined by $\ba(z)= \max\{\deg(h_{x,y,z})\mid x,y\in \Wext\}$), and of the asymptotic algebra (note that P1--P15 remain conjectural in the general case of Coxeter groups with unequal parameters).  In particular, these properties imply that $\ba$ is constant on two-sided cells, and it follows that $\ba(\De_\la)=\ba(\sw_\la') = \ell(\sw_{\la'})$ where $\ell$ is the usual length function on $\Wext$. The structure constants  $\ga_{x,y,z}$ with respect to the basis $(\st_{w})$ of the asymptotic algebra $\cJ$ are the coefficients of the term of degree $\ba(z)$ in $h_{x,y,z^{-1}}$. 

Our main tool to describe the algebra $\cJ$ is a family of matrix representations $(\pi_\la)_{\la\vdash n+1}$ of~$\Hext$ defined over a ring $\sR[\zeta_{\la}]$ that turn out to be deeply connected to Kazhdan-Lusztig theory and that admit a  combinatorial description in terms of $\la$-folded alcove paths (with respect to a distinguished basis). In the course of this paper we will show that:
 \begin{compactenum}
\item[\raisebox{0.35ex}{\tiny$\bullet$}] the representations $\pi_\la$ are {\it bounded} by $\ba(\De_\la)$. That is, the maximal degree in $\sq$ of the entries of the matrices $\pi_{\la}(T_w)$ for $w\in \Wext$ is bounded by $\ba(\De_\la)$ (see Theorem~\ref{thm:fullbounded}). 
\item[\raisebox{0.35ex}{\tiny$\bullet$}] the representations $\pi_\la$ {\it recognise} $\De_\la$. That is, for $w\in \Wext$ we have $w\in \De_\la$ if and only if the matrix $\pi_{\la}(T_w)$ has an entry of degree $\ba(\De_\la)$. 
\item[\raisebox{0.35ex}{\tiny$\bullet$}] The asymptotic algebra $\cJ_{\la}=\cJ_{\De_\la}$ associated to $\De_\la$ is isomorphic to the matrix algebra $\fC_{\la}$ with $\ZZ$-basis given by the \textit{leading matrices} $\fc_\la(w)$ with $w\in \De_\la$. Here $\fc_\la(w)$ is the matrix obtained by evaluating the matrix $\sq^{-\ba(\De_\la)}\pi_\la(T_w)$ at $\sq^{-1}=0$ (see Theorem~\ref{thm:recognise2b}).
\end{compactenum}
These properties show that the family $(\pi_\la)_{\la\vdash n+1}$  forms a \textit{balanced system of cell representations} as defined in \cite{GP:19} (see Corollary~\ref{cor:balancedsystem}). 

In the process of proving the above statements we prove the following results, which we believe are interesting in their own right. 
 \begin{compactenum}
 \item[\raisebox{0.35ex}{\tiny$\bullet$}] We construct a canonical trace functional on the asymptotic algebra for a weighted Hecke algebra of arbitrary type, following Lusztig \cite[\S20.1(b)]{Lus:03} (see Section~\ref{thm:innerproductonJ}), and we prove an asymptotic version of Opdam's Plancherel Theorem~\cite{Opd:04} giving a spectral decomposition of this trace functional in affine type~$\sA$. In particular, we show that the terms in the Plancherel Theorem (in affine type~$\sA$) are in bijection with the set of two-sided cells, and that the coefficients that appear are linked to Lusztig's $\ba$-function (see Section~\ref{subsec:apt}).
\item[\raisebox{0.35ex}{\tiny$\bullet$}] We prove a $\la$-relative version of the Satake Isomorphism, giving an isomorphism from $\pi_\la({\bf 1}_{\la}\Hext {\bf 1}_{\la})$ (where ${\bf 1}_{\la}$ is a renormalisation of $C_{\sw_{\la'}}$) to the ring of $G_{\la}$-symmetric functions (see Theorem~\ref{thm:satake}).
\item[\raisebox{0.35ex}{\tiny$\bullet$}] We describe the set $\Ga_\la\cap \Ga_\la^{-1}$ (with $\Ga_{\la}$ the right cell containing~$\sw_{\la'}$) in a very natural way in terms of the geometry of the \textit{fundamental $\la$-alcove} (see Theorem~\ref{thm:leading0}). 
\end{compactenum}
\medskip

We now describe in more details the content of the paper. To understand the general philosophy it is helpful to understand the situation for the lowest two-sided cell $\De_0$ (the minimal cell with respect to the two-sided order $\leq_{LR}$; in this case our method applies to all affine Weyl groups and all choices of parameters). This cell corresponds to the partition $(1^{n+1})$, and contains the longest element $\sw_0$ of the finite Weyl group~$\Wfin$. Let $\Ga_0$ be the right cell that contains~$\sw_0$.  Then $\Ga_0\cap \Ga_0^{-1}$ is the set of words $t_{\ga}\sw_0$ where $\ga\in P_+$ is a dominant weight and $t_\ga\in \Wext$ is the translation by~$\ga$. 
Let $\pi_0$ be the principal series representation of $\Hext$ as defined in \cite{GP:19} for example. Then $\pi_0$ has a combinatorial description in terms of positively folded alcoves paths, is bounded by $\ell(\sw_0)$ and recognises~$\De_0$ (see~\cite[Section~6]{GP:19}). Furthermore, the leading matrix $\fc_0(t_{\ga}\sw_0)$ for $\ga\in P_+$ has a unique non-zero coefficient, equal to $\fs_{\ga}(\zeta)$ (the Schur function of type~$\Wfin$). In this case, the Satake Isomorphism provides an isomorphism between ${\bf 1}_{0} \Hext {\bf 1}_{0}$ and the ring of $\Wfin$-symmetric functions (see \cite{Lus:83,NR:03}). In this paper we construct analogues of the above results for all two-sided cells.

Given a root system of type $\sA_n$ we can construct the set of alcoves in the usual way. Then to any partition $\la$ of $n+1$, we associate the fundamental $\la$-alcove $\cA_\la$ defined as the set of alcoves that lie between the hyperplanes $H_{\al,0}$ and $H_{\al,1}$ where $\al$ runs over the positive roots~$\Phi_\la^+$ associated to the Young subgroup $W_\la$. Our representations $\pi_\la$ can be expressed in terms of $\la$-folded alcove paths in $\cA_\la$ (a generalisation of Ram's positively folded alcove paths~\cite{Ram:06}; see \cite{GLP:23}). The symmetry group of $\cA_\la$ plays a very important role in our work: it is the semi-direct product of $G_\la$ (the subgroup of $\Wfin$ stabilising $\cA_\la$) and a set of \textit{pseudo-translations}~$\sT_\la$. Remarkably, the $\la$-dominant elements of $\sT_{\la}$ turns out to be closely related to $\Ga_{\la}\cap\Ga_{\la}^{-1}$, and the ring $\cJ_{\Ga_{\la}\cap\Ga_{\la}^{-1}}$ spanned by $\st_w$ with $w\in\Ga_{\la}\cap\Ga_{\la}^{-1}$ turns out to be isomorphic to the ring $\ZZ[\zeta_{\la}]^{G_{\la}}$ of $G_{\la}$-symmetric functions (see Theorem~\ref{thm:leading0}). In the case of the lowest two-sided cell, $\cA_\la$ is the set of all alcoves, $G_\la$ is $\Wfin$, and the set of pseudo-translations $\sT_\la$ is the set of all translations.

Before turning our attention to the asymptotic Plancherel Theorem, let us first recall the situation for a finite dimensional Hecke algebra $\Hfin$ of type~$\sA_n$. In this case the canonical trace on the Hecke algebra decomposes as a linear combination of irreducible characters (indexed by partitions) and the coefficients that appear are normalisations of the \textit{generic degrees} of~$\Hfin$. It turns out that the coefficient associated to $\la$ has valuation $2\ba(\sw_{\la'})$ (see \cite{Geck:16}). There is an analogue of this decomposition for affine Hecke algebras in the form of the Plancherel Theorem~\cite{Opd:04} which expresses the canonical trace as a sum of integrals over families of representations. In this paper we show that, in affine  type $\sA$, a similar phenomenon as in the finite case happens: the terms in the Plancherel formula are in bijection with the two-sided cells, and the coefficients have valuation equal to $2\ba(\De_\la)$ (this behaviour was first observed in affine Hecke algebras of types $\tilde{\sG}_2$ and $\tilde{\sC}_2$ by the second and fourth authors in~\cite{GP:19,GP:19b}). This leads to an asymptotic version of the Plancherel Theorem, which in turn gives rise to an inner product on~the ring~$\fC_{\la}$ of leading matrices.

Our $\la$-relative Satake theory and asymptotic Plancherel Theorem come together to prove that the $\ZZ$-algebra $\fC_\la$ spanned by the leading matrices $\fc_{\la}(w)$ with $w\in\Delta_{\la}$  is isomorphic to the asymptotic algebra~$\cJ_\la$ (see Theorem~\ref{thm:recognise2b}). We then carefully analyse our leading matrices $\fc_{\la}(w)$ for $w\in\Ga_{\la}\cap\Ga_{\la}^{-1}$ to prove that $\cJ_{\Ga_{\la}\cap\Ga_{\la}^{-1}}\cong \ZZ[\zeta_{\la}]^{G_{\la}}$, from which it follows that $\cJ_{\la}$ is a full matrix algebra over the ring $\ZZ[\zeta_{\la}]^{G_{\la}}$ (see Theorem~\ref{thm:leading0} and Corollary~\ref{cor:aalg}).

The structure of the paper is as follows. In Section~\ref{sec:prelim} we recall background material on the symmetric group, partitions and tableaux, affine Hecke algebras, and Kazhdan-Lusztig theory. In Section~\ref{sec:2} we explicitly describe the fundamental $\la$-alcove and its symmetries in affine type~$\sA$. Moreover we introduce a $\la$-dominance order on weights of the fundamental $\la$-alcove, and recall the theory of $\la$-folded alcove paths from~\cite{GLP:23,GP:19,GP:19b}. We construct a ring $\ZZ[\zeta_{\la}]^{G_{\la}}$ of $G_{\la}$-symmetric functions, and recall the definition of the induced representations $\pi_{\la}$, and their combinatorial description, from~\cite{GLP:23}. 

In Section~\ref{sec:satake} we develop our $\la$-relative Satake theory, and in particular we prove in Theorem~\ref{thm:symmetry} that the function $f_{\la}(h)=\chi_{\la}(h\clap)$ is $G_{\la}$-symmetric. The $\la$-relative Satake Isomorphism is then given in Theorem~\ref{thm:satake}. Section~\ref{sec:killbound} proves two important properties of our representations: the killing property (see Theorem~\ref{thm:fullkilling}) and the boundedness property (see Theorem~\ref{thm:fullbounded}). This boundedness property allows us to define the leading matrices $\fc_{\la}(w)$ in Section~\ref{sec:leading}.

In Section~\ref{sec:asymptoticplancherel} we construct a canonical trace on Lusztig's asymptotic algebra~$\cJ$ (for arbitrary Coxeter type). We then recall the Plancherel Theorem for type $\tilde{\sA}_n$, following~\cite{AP:05}, and develop our asymptotic Plancherel Theorem in Theorem~\ref{thm:AsymptoticPlancherel}. This allows us to prove that $\pi_{\la}$ recognises~$\Delta_{\la}$ (Theorem~\ref{thm:recognise2a}) and that $\cJ_{\la}\cong \fC_{\la}$ (Theorem~\ref{thm:recognise2b}). 

In Section~\ref{sec:maximallengthelts} we introduce important elements $\sm_{\ga}$ as certain maximal length $W_{\la'}$-double coset representatives. These elements will ultimately describe $\Ga_{\la}\cap\Ga_{\la}^{-1}$, however we first need to compute their lengths and prove a monotonicity of length with respect to the $\la$-dominance order. These properties are stated in Theorem~\ref{thm:mgamma}, however since the proofs turn out to be technical, we present them in an appendix (see Appendix~\ref{app:proofs}). Finally, in Section~\ref{sec:asymptotic} we give our explicit construction of Lusztig's asymptotic algebra~$\cJ_{\la}$, with the main results being Theorem~\ref{thm:leading0} and Corollary~\ref{cor:aalg}).

\section{Background and preliminary results}\label{sec:prelim}

In this section we provide background on the symmetric group, tableau, the extended affine Weyl group, the affine Hecke algebra, and Kazhdan-Lusztig theory.

\subsection{The symmetric group and type $\sA_n$ root system}

Let $
V=\{v\in\mathbb{R}^{n+1}\mid v\cdot\mathbf{1}=0\},
$
where $\mathbf{1}=(1,1,\ldots,1)$, and let $e_i=(0,0,\ldots,0,1,0,\ldots,0)-\frac{1}{n+1}\mathbf{1}$ for $1\leq i\leq n+1$
(with the $1$ in the $i$th place). Thus $e_i\in V$, and we have $e_1+\cdots+e_{n+1}=0$. Let $\langle \cdot,\cdot\rangle$ be the restriction of the standard inner product on $\mathbb{R}^{n+1}$ to $V$.

Let $\Phi^+=\{e_i-e_j\mid 1\leq i<j\leq n+1\}$ and $\Phi=\Phi^+\cup(-\Phi^+)$ be a root system of type $\sA_n$. The simple roots are $\alpha_i=e_i-e_{i+1}$ for $1\leq i\leq n$. The \textit{height} of a root $\alpha=\sum_{i=1}^na_i\alpha_i$ is $\mathrm{ht}(\alpha)=\sum_{i=1}^na_i$, and the unique highest root is $\varphi=\alpha_1+\cdots+\alpha_n$. The Weyl group $\Wfin$ of $\Phi$ is the symmetric group $\mathfrak{S}_{n+1}$, acting by $we_i=e_{w(i)}$. Let $s_1,\ldots,s_n$ be the simple reflections (elementary transpositions), and write $S_0=\{s_i\mid 1\leq i\leq n\}$. Let $\ell:\Wfin\to \mathbb{N}$ be the standard length function on the Coxeter system $(\Wfin,\Sfin)$. The \textit{inversion set} of $w\in \Wfin$ is $\Phi(w)=\{\alpha\in\Phi^+\mid w^{-1}\alpha\in-\Phi^+\}$.

The \textit{one line expression} of $w\in \Wfin$ is the sequence $[w(1),w(2),\cdots,w(n+1)]$. For $1\leq i\leq n$ we have $\ell(ws_i)=\ell(w)+1$ if and only if $w(i+1)>w(i)$, and $\ell(s_iw)=\ell(w)+1$ if and only if $i$ and $i+1$ appear in ascending order in the $1$-line notation of~$w$. 

If $J\subseteq \{1,\ldots,n\}$ let $W_J$ be the parabolic subgroup of $\Wfin$ generated by $\{s_j\mid j\in J\}$. Let~${^J}W$ denote the transversal of minimal length elements of cosets in $W_J\backslash \Wfin$. Each $w\in\Wfin$ has a unique expression as $w=uv$ with $u\in W_J$ and $v\in {^J}W$, and $\ell(w)=\ell(u)+\ell(v)$. Let $\sw_J$ denote the longest element of $W_J$. We write $\sw_0=\sw_{\{1,\ldots,n\}}$ for the longest element of~$\Wfin$.

The \textit{support} of a root $\alpha=\sum_{i=1}^na_i\alpha_i\in\Phi$ is $\mathrm{supp}(\alpha)=\{i\mid a_i\neq 0\}$.  For $J\subseteq \{1,\ldots,n\}$ let 
$
\Phi_J=\{\alpha\in\Phi\mid\mathrm{supp}(\alpha)\subseteq J\},
$
and for $w\in \Wfin$ write $\Phi_J(w)=\Phi(w)\cap\Phi_J$. If $w=uv$ with $u\in W_J$ and $v\in {^J}W$ then $\Phi(u)=\Phi(w)\cap \Phi_J$. In particular we have
$
{^J}W=\{v\in W\mid \Phi_J(v)=\emptyset\}
$ (see \cite[Lemma~2.2]{GLP:23}).

\subsection{Partitions and tableaux}\label{sec:tableaudefns}

Many objects in this paper are indexed by the set $\cP(n+1)$ of all partitions of $n+1$. For $\la\in\cP(n+1)$ write $\la\vdash n+1$. We represent partitions as Young diagrams (using English notation conventions). For $\la\vdash n+1$ we make the following definitions. 
\begin{compactenum}
\item[--] $r(\la)$ denotes the number of parts of $\la$, and so $\la=(\la_1,\la_2,\ldots,\la_{r(\la)})$. 
\item[--] $\la(0)=0$ and $\la(i)=\la_1+\cdots+\la_i$ for $1\leq i\leq r(\la)$. 
\item[--] $\bt_r(\la)$ is the standard tableau of shape $\la$ filled by row. 
\item[--] $\bt_c(\la)$ is the standard tableau of shape $\la$ filled by column. 
\item[--] $\la[i,j]$ is the element in row $i$ and column $j$ of $\bt_r(\la)$ (with $1\leq i\leq r(\la)$ and $1\leq j\leq \la_i$). 
\item[--] $J_{\la}=\{1,2,\ldots,n+1\}\backslash\{\la(1),\la(2),\ldots,\la(r(\la))\}$. 
\item[--] $\la'$ denotes the transposed partition of $\la$ (obtained by reflecting $\la$ in the main diagonal). 
\end{compactenum}

\begin{example}\label{ex:running1}
Let $\la=(5,3,3,2,1)\vdash 14$. Then $r(\la)=5$ (the number of rows of the Young diagram), and we have
$$
\bt_r(\la)=\begin{ytableau}
1&2&3&4&5\\
6&7&8\\
9&10&11\\
12&13\\
14
\end{ytableau}\quad\text{and}\quad \bt_c(\la)=\begin{ytableau}
1&6&10&13&14\\
2&7&11\\
3&8&12\\
4&9\\
5
\end{ytableau}
$$
Then $\{\la(i)\mid 1\leq i\leq r(\la)\}=\{5,8,11,13,14\}$ (note that these are the final entries in each row of $\bt_r(\la)$) and so $J_{\la}=\{1,2,3,4,6,7,9,10,12\}$. We have $\la'=(5,4,3,1,1)$. 
\end{example}

Let $\leq $ be the dominance partial order on $\cP(n+1)$ given by $\mu\leq \la$ if and only if $\mu(i)\leq \la(i)$ for all $i\geq 1$ (where we set $\la_i=0$ if $i>r(\la)$). Recall that $\mu\leq\lambda$ if and only if $\lambda'\leq \mu'$.

Write $W_{\la}$, ${^\la}W$, $\Phi_{\la}$ and $\sw_{\la}$ in place of $W_{J_{\la}}$, ${^{J_{\la}}}W$, $\Phi_{J_{\la}}$ and $\sw_{J_{\la}}$ (we will employ similar notational simplifications throughout the paper without further comment). In particular $W_{\la}$ is the standard \textit{Young subgroup} associated to~$\la$.

\begin{lemma}\label{lem:reducedcharacterisation}
We have $w\in{^\la}W$ if and only if for each row of $\bt_r(\la)$ the elements of the row appear in ascending order in the $1$-line notation of~$w$.
\end{lemma}

\begin{proof}
Note that if $i$ and $i+1$ lie in the same row of $\bt_r(\la)$ then $i\in J_{\la}$, and if $i+1$ occurs before $i$ in the $1$-line notation of $w$ then $\ell(s_iw)=\ell(w)-1$. The result follows. 
\end{proof}

Let
$$
\ba_{\la}=\sum_{i\geq 1}(i-1)\la_i=\frac{1}{2}\sum_{i\geq 1}\la_i'(\la_i'-1)=\ell(\sw_{\la'}). 
$$

\begin{lemma}\label{lem:basic}
Let $\lambda,\mu\in\cP(n+1)$. If $\mu\leq \lambda$ then $\ba_{\mu}\geq \ba_{\la}$ with equality if and only if $\mu=\lambda$. 
\end{lemma}

\begin{proof}
This follows from the identity
$\ba_{\la}=\sum_{i\geq 1}(n+1-\la(i))$. 
\end{proof}

Let 
\begin{align}\label{eq:rholambda}
\rho_{\la}=\frac{1}{2}\sum_{\alpha\in\Phi_{\la}^+}\alpha.
\end{align}
Then for $j\in J_{\la}$ we have $\langle\alpha_j,\rho_{\la}\rangle=1$ (c.f. \cite[Lemma~2.4]{GLP:23}).

\subsection{The element $\su_{\la}$}

The element $\su_{\la}$ introduced below plays an important role in this work.

\begin{defn}
Let $\su_{\la}\in\Wfin$ be the element given in $1$-line notation by reading in the columns of $\bt_{r}(\la)$, from top to bottom, left to right. 
\end{defn}

\begin{defn} 
Let $\la\vdash n+1$ and $w\in \Wfin$. Insert dividers into the $1$-line notation 
$
w=[w(1),w(2),\ldots,w(n+1)]
$ 
forming ``blocks'' according to the following rules. For $1\leq i\leq n$ insert a divider between $w(i)$ and $w(i+1)$ if either (1) $w(i+1)<w(i)$ or, (2) $w(i+1)>w(i)$ and the numbers $w(i)$ and $w(i+1)$ lie in a common row of $\bt_r(\la)$.  The resulting expression, with the $1$-line notation split into blocks, is called the \textit{$\lambda$-expression of~$w$}. Define a partition $\muu(w, \la)\in\cP(n+1)$ by rearranging the sequence of lengths of the blocks of the $\lambda$-expression of~$w$ into decreasing order.
\end{defn}

\begin{example}
If $\la=(5,3,3,2,1)$  then $\su_{\la}=[1,6,9,12,14,2,7,10,13,3,8,11,4,5]$ and  $w=[1,3,7,2,6,11,12,5,4,14,9,10,13,8]$ has $\la$-expression 
$
[1\,|\, 3,7\,|\, 2,6,11,12\,|\, 5\,|\, 4,14\,|\, 9\,|\, 10,13\,|\, 8]
$. Thus $\muu(w,\la)=(4,2,2,2,1,1,1,1)$. 
\end{example}

\begin{defn}
The \textit{right $\la$-ascent set} of $u\in {^\la}W$ is 
$$
A_{\la}(u)=\{s\in S_0\mid \ell(us)=\ell(u)+1\text{ and }us\in {^\la}W\}.
$$
\end{defn}

\begin{lemma}\label{lem:lambdaexp}
For $u\in {^\la}W$ we have 
$$
A_{\la}(u)=\{s_i\mid \text{$u(i)$ and $u(i+1)$ lie in a common block of the $\lambda$-expression of $u$}\}. 
$$
\end{lemma}

\begin{proof}
Suppose that $u(i)$ and $u(i+1)$ are in different blocks of the $\la$-expression of $u$. Then either $u(i+1)<u(i)$ or $u(i+1)>u(i)$ and the numbers $u(i)$ and $u(i+1)$ are in a common row of $\bt_r(\la)$. In the former case we have $\ell(us_i)=\ell(u)-1$, hence $s_i\notin A_{\la}(u)$. In the latter case the numbers $us_i(i+1)=u(i)$ and $us_i(i)=u(i+1)$ lie in the common row of $\bt_r(\la)$ yet are not in ascending order in the $1$-line notation of $us_i$, giving $us_i\notin {^\la}W$ (by Lemma~\ref{lem:reducedcharacterisation}), and so $s_i\notin A_{\la}(u)$.

Conversely, suppose that $u(i)$ and $u(i+1)$ are in a common block of the $\la$-expression of~$u$. Then $u(i+1)>u(i)$ and hence $\ell(us_i)=\ell(u)+1$. Since $u\in{^\la}W$, for each row of $\bt_r(\la)$ the elements of the row appear in ascending order in the $1$-line notation of~$u$ (by Lemma~\ref{lem:reducedcharacterisation}). Since $u(i)$ and $u(i+1)$ lie in a common block of the $\la$-expression of $u$ they lie on different rows of $\bt_r(\la)$, and hence the elements of each row of $\bt_r(\la)$ appear in ascending order in the $1$-line notation of~$us_i$, and so $us_i\in {^\la}W$. Thus $s_i\in A_{\la}(u)$. 
\end{proof}

\begin{cor}\label{cor:lengthaddula}
We have $\muu(\su_{\la},\la)=\la'$ and $A_{\la}(\su_{\la})=J_{\la'}$. Moreover $\su_\la w\in {^{\la}W}$ for all $w\in W_{\la'}$.
\end{cor}

\begin{proof}
By construction, the blocks of the $\la$-expression of $\su_{\la}$ are the columns of $\bt_r(\la)$, and the first statement follows from Lemma~\ref{lem:lambdaexp}. For the second statement, note that multiplying on the right by $\sw_{\la'}$ will only permute elements in blocks of the $\la$-expression (see Lemma~\ref{lem:reducedcharacterisation}). 
\end{proof}

\begin{lemma}\label{lem:dominance}
For all $w\in \Wfin$ we have $\muu(w,\la)\leq \la'$.
\end{lemma}

\begin{proof}
Let $\mu=\muu(w,\la)$. Let $\gamma_1,\gamma_2,\ldots,\gamma_k$ be the blocks of the $\la$-expression for $w$, arranged so that they are in decreasing length. Thus $\mu=(|\gamma_1|,\ldots,|\gamma_k|)$. Let $\bt$ be a (not necessarily column strict) tableau of shape $\mu$ given by entering the elements of $\gamma_i$ into the $i$th row of $\bt$ in any order. By construction of the $\la$-expression, the elements of row $i$ of $\bt$ lie in different columns of $\bt_c(\la')$, and so by the Dominance Lemma for partitions (see \cite[Lemma~2.2.4]{Sag:01}) we have $\mu\leq\la'$ as required. 
\end{proof}

The following characterisation of $\su_{\la}$ is crucial at a few occasions later. 

\begin{thm}\label{thm:etale}
Let $u\in{^\la}W$. We have
\begin{compactenum}[$(1)$]
\item  $\ell(\sw_{A_{\la}(u)})\leq \ell(\sw_{\la'})$ with equality if and only if $\muu(u,\la)=\la'$, and
\item  $A_{\la}(u)=J_{\la'}$ if and only if $u=\su_{\la}$. 
\end{compactenum}
\end{thm}

\begin{proof}
Let $u\in {^\la}W$ and write $\mu=\muu(u,\la)$. By Lemma~\ref{lem:lambdaexp} the type of the Young subgroup generated by $A_{\la}(u)$ is 
$
\mathfrak{S}_{\mu_1}\times\cdots\times \mathfrak{S}_{\mu_m}
$, and hence $\ell(\sw_{A_{\la}(u)})=\ell(\sw_{\mu})=\ba_{\mu'}$. Lemma~\ref{lem:dominance} gives $\mu\leq \lambda'$, which is equivalent to $\mu'\geq\lambda$. Thus Lemma~\ref{lem:basic} gives
$
\ell(\sw_{A_{\la}(u)})=\ba_{\mu'}\leq \ba_{\lambda}
$
with equality if and only if $\muu(u,\la)=\la'$, hence (1).

To prove (2), suppose that $u\in {^\la}W$ with $A_{\la}(u)=J_{\la'}$. Thus $\muu(u,\la)=\la'$. If $u\neq \su_{\la}$ then the $\la$-expression for $u$ is either a rearrangement of the blocks of the $\la$-expression of $\su_{\la}$, or there is a block of the $\la$-expression of $u$ that is not a column of $\bt_r(\la)$. In either case, there exist two elements of a row of $\bt_{r}(\la)$ that are not in ascending order in the $\la$-expression of $u$, and hence $u\notin {^\la}W$, a contradiction.
\end{proof}

\subsection{The extended affine Weyl group}

The fundamental weights of $\Phi$ are the vectors $\omega_i=e_1+\cdots+e_i$ for $1\leq i\leq n$ (then $\langle\omega_i,\alpha_j\rangle=\delta_{i,j}$). By convention we set $\omega_0=0$ and $\omega_{n+1}=e_1+\cdots+e_{n+1}=0$. Let $P$ be the $\mathbb{Z}$-span of $\{\omega_1,\ldots,\omega_n\}$ and let $P_+$ be the $\ZZ_{\geq0}$-span of $\{\omega_1,\ldots,\omega_n\}$.  Then $P$ is the $\ZZ$-span of $\{e_i\mid1\leq i\leq n+1\}$ (recall that $e_1+\cdots+e_{n+1}=0$) and $a_1e_1+\cdots+a_{n+1}e_{n+1}\in P_+$ if and only if $a_1\geq a_2\geq \cdots\geq a_{n+1}$.

Let $Q$ be the $\mathbb{Z}$-span of $\Phi$ and let $Q_+$ be the $\ZZ_{\geq 0}$-span of $\{\alpha_1,\ldots,\alpha_n\}$. Note that $Q\subseteq P$, and $a_1e_1+\cdots+a_{n+1}e_{n+1}\in Q$ if and only if $a_1+\cdots+a_{n+1}=0\mod n+1$. Moreover, if $\gamma=a_1e_1+\cdots+a_{n+1}e_{n+1}\in Q$ with $a_1+\cdots+a_{n+1}=k(n+1)$ then $\gamma=a_1'e_1+\cdots+a_{n+1}'e_{n+1}$ where $a_i'=a_i-k$, and $\gamma\in Q_+$ if and only if $a_1'+\cdots+a_i'\geq 0$ for all $1\leq i\leq n+1$.

For $\gamma\in P$ let $t_{\gamma}:V\to V$ be the translation $t_{\gamma}(v)=v+\gamma$. The affine Weyl group $\Waff$ and the extended affine Weyl group $\Wext$ are
$$
\Waff=Q\rtimes \Wfin\quad\text{and}\quad \Wext=P\rtimes \Wfin,
$$
where $wt_{\gamma}=t_{w\gamma}w$ for $w\in\Wfin$ and $\gamma\in P$.

For $\alpha\in\Phi$ and $k\in\ZZ$ let 
$
H_{\alpha,k}=\{x\in V\mid \langle x,\alpha\rangle=k\},
$
and let $s_{\alpha,k}(v)=v-(\langle v,\alpha\rangle-k)\alpha$ be the orthogonal reflection in $H_{\alpha,k}$, and so $s_{\alpha,k}=t_{k\alpha}s_{\alpha,0}$. The group $\Waff$ is a Coxeter group with generators $S=\{s_0\}\cup S_0$, where $s_0=s_{\varphi,1}$, with $\varphi$ the highest root of $\Phi$. The group $\Wext$ is not a Coxeter group, however we have 
$$\Wext=\Waff\rtimes \Sigma\quad\text{where}\quad \Sigma=P/Q\cong \ZZ/(n+1)\ZZ.$$ 
We extend the length function $\ell:\Waff\to \mathbb{N}$ to $\Wext$ by setting $\ell(w\pi)=\ell(w)$ for all $w\in\Waff$ and $\pi\in\Sigma$. Thus $\Sigma=\{w\in \Wext\mid \ell(w)=0\}$. Each $\pi\in\Sigma$ induces a permutation of the nodes of the extended Dynkin diagram by $\pi s_i \pi^{-1}=s_{\pi(i)}$ for $0\leq i\leq n$. Let $\sigma\in\Sigma$ be the element with $\sigma(i)=i+1$ (with indices read cyclically). 
By a \textit{reduced expression} for $w\in \Wext$ we shall mean a decomposition $w=s_{i_1}\cdots s_{i_{\ell}}\pi$ with $\ell=\ell(w)$ and $\pi\in\Sigma$.

If $w\in\Wext$ we define the \textit{linear part} $\theta(w)\in \Wfin$ and the \textit{translation weight} $\wt(w)\in P$ by
\begin{align*}
w=t_{\wt(w)}\theta(w).
\end{align*}
Moreover, for $\la\vdash n+1$ let 
$$
\theta(w)=\theta_{\la}(w)\theta^{\la}(w),
$$
where $\theta_{\la}(w)\in W_{\la}$ and $\theta^{\la}(w)\in {^\la}W$. 

The closures of the open connected components of $V\backslash\left(\bigcup_{\alpha,k}H_{\alpha,k}\right)$ are \textit{alcoves}. Let $\mathbb{A}$ denote the set of all alcoves. The \textit{fundamental alcove} is given by
$$
\cA_0=\{x\in V\mid 0\leq\langle x,\alpha\rangle\leq 1\text{ for all $\alpha\in\Phi^+$}\}.
$$
The hyperplanes bounding $\cA_0$ are called the \textit{walls} of $\cA_0$. Explicitly these walls are $H_{\alpha_i,0}$ with $1\leq i\leq n$ and $H_{\varphi,1}$. We say that a \textit{panel} of $\cA_0$ (that is, a codimension $1$ facet) has \textit{type} $i$ for $1\leq i\leq n$ if it lies on the wall $H_{\alpha_i,0}$, and type $0$ if it lies on the wall $H_{\varphi,1}$.

The (non-extended) affine Weyl group $\Waff$ acts simply transitively on $\mathbb{A}$. We use the action of $\Waff$ to transfer the notions of walls, panels, and types of panels to arbitrary alcoves. Alcoves~$A$ and $A'$ are called \textit{$i$-adjacent} (written $A\sim_{i}A'$) if $A\neq A'$ and $A$ and $A'$ share a common type~$i$ panel (with $0\leq i\leq n$). Thus the alcoves $w\cA_0$ and $ws_i\cA_0$ are $i$-adjacent for all $w\in W$ and $0\leq i\leq n$. The extended affine Weyl group $\Wext$ acts transitively on $\mathbb{A}$, and the stabiliser of $\cA_0$ is~$\Sigma$.

Each affine hyperplane $H_{\alpha,k}$ with $\alpha\in\Phi^+$ and $k\in\ZZ$ divides $V$ into two half-spaces, denoted
$$
H_{\alpha,k}^+=\{\alpha\in V\mid \langle x,\alpha\rangle\geq k\}\quad\text{and}\quad H_{\alpha,k}^-=\{x\in V\mid \langle x,\alpha\rangle\leq k\}.
$$
This ``orientation'' of the hyperplanes is called the \textit{periodic orientation}. Write $w\cA_0\,{^-}|^+\, ws_i\cA_0$ to indicate that $ws_i\cA_0$ is on the positive side of the hyperplane separating $w\cA_0$ and $ws_i\cA_0$. In this situation we will also say that $w\to ws_i$ is a \textit{positive crossing} (or simply is \textit{positive}), and similarly if $w\cA_0\,{^+}|^-\, ws_i\cA_0$ we say that $w\to ws_i$ is a negative crossing.

\subsection{The extended affine Hecke algebra of type $\tilde{\sA}_n$}

Let $\sq$ be an indeterminate, and let $\sR=\ZZ[\sq,\sq^{-1}]$. The Hecke algebra of type $\Wext$ is the algebra~$\Hext$ over $\sR$ with basis $\{T_w\mid w\in\Wext\}$ and multiplication given by (for $v,w\in\Wext$ and $s\in \Saff$)
\begin{align}\label{eq:relations}
\begin{aligned}
T_wT_v&=T_{wv}&&\text{if $\ell(wv)=\ell(w)+\ell(v)$}\\
T_wT_{s}&=T_{ws}+(\sq-\sq^{-1})T_w&&\text{if $\ell(ws)=\ell(w)-1$}.
\end{aligned}
\end{align}
We often write $T_i$ in place of $T_{s_i}$ for $0\leq i\leq n$.

Let $w\in \Wext$ and choose any expression $w=s_{i_1}\cdots s_{i_{\ell}}\pi$ (not necessarily reduced). Let $v_0=e$ and $v_k=s_{i_1}\cdots s_{i_k}$ for $1\leq k\leq \ell$. Let $A_k=v_k\cA_0$, and let $\epsilon_1,\ldots,\epsilon_{\ell}\in\{-1,1\}$ be the sequence defined by
\begin{align*}
\epsilon_k=\begin{cases}
+1&\text{if $A_{k-1}\,{^-}|^+\,A_k$}\\
-1&\text{if $A_{k-1}\,{^+}|^-\,A_k$}.
\end{cases}
\end{align*}
Then the element 
$$
X_w=T_{i_1}^{\epsilon_1}\cdots T_{i_k}^{\epsilon_k}T_{\pi}\in\Hext
$$
does not depend on the particular expression $w=s_{i_1}\cdots s_{i_k}\pi$ chosen (see \cite{Goe:07}). From the defining relations~(\ref{eq:relations}) it follows that $X_w-T_w$ is a linear combination of terms $T_v$ with $v<w$ (in extended Bruhat order), and hence $\{X_w\mid w\in\Wext\}$ is a basis of $\Hext$.

If $\gamma\in P$ we write 
$$
X^{\gamma}=X_{t_{\gamma}}.
$$
Then $X^{\gamma_1}X^{\gamma_2}=X^{\gamma_1+\gamma_2}=X^{\gamma_2}X^{\gamma_1}$ for all $\gamma_1,\gamma_2\in P$, and for $w\in\Wext$ we have
\begin{align}\label{eq:splitting}
X_w=X_{t_{\gamma}u}=X^{\gamma}X_{u}=X^{\gamma}T_{u^{-1}}^{-1}
\end{align}
where $\gamma=\wt(w)$ and $u=\theta(w)$. Thus the set $\{X^{\gamma}T_{u^{-1}}^{-1}\mid \gamma\in P,\,u\in \Wfin\}$ is a basis of~$\Hext$ (called the \textit{Bernstein-Lusztig basis}). Since $s_0=t_{\varphi^{\vee}}s_{\varphi}$, Equation~(\ref{eq:splitting}) gives $T_0=X^{\varphi^{\vee}}T_{s_{\varphi}}^{-1}$.

The \textit{Bernstein relation} (see \cite[Proposition~3.6]{Lus:89}) is
$$
T_iX^{\gamma}=X^{s_i\gamma}T_i+(\sq-\sq^{-1})\frac{X^{\gamma}-X^{s_i\gamma}}{1-X^{-\alpha_i}}.
$$
Let $X_i=X^{e_i}$ for $1\leq i\leq n+1$ (recall that $e_i\in P$, and since $e_1+\cdots +e_{n+1}=0$ we have $X_1\cdots X_{n+1}=1$). The Bernstein relation gives $T_jX_i=X_iT_j$ if $|i-j|>1$ and
$$
T_i^{-1}X_iT_i^{-1}=X_{i+1}\quad\text{for $1\leq i\leq n$}.
$$

Let
$$
\sR[X]=\mathrm{span}_{\sR}\{X^{\gamma}\mid \gamma\in P\}.
$$
The Weyl group $\Wfin$ acts on $\sR[X]$ by linearly extending $w\cdot X^{\gamma}=X^{w\gamma}$. Let
$$
\sR[X]^{\Wfin}=\{f(X)\in \sR[X]\mid w\cdot f(X)=f(X)\text{ for all $w\in \Wfin$}\}
$$
(the \textit{ring of symmetric polynomials}).

\subsection{Kazhdan-Lusztig theory}\label{sec:KLsec}

The \textit{bar involution} on $\sR$ with $\overline{\sq}=\sq^{-1}$ extends to an involution on $\Hext$ by setting
$$
\overline{\sum_{w\in \Wext}a_wT_w}=\sum_{w\in \Wext}\overline{a_w}\,T_{w^{-1}}^{-1}.
$$
In \cite{KL:79} Kazhdan and Lusztig proved that there is a unique basis $\{C_w\mid w\in \Wext\}$ of $\Hext$ such that, for all $w\in W$, 
$$
\overline{C_w}=C_w\quad\text{and}\quad C_w=T_w+\sum_{v<w}P_{v,w}T_v\quad\text{with}\quad P_{v,w}\in \sq^{-1}\ZZ[\sq^{-1}].
$$
This basis is called the \textit{Kazhdan-Lusztig basis} of $\Hext$, and the polynomials $P_{v,w}$ are the \textit{Kazhdan-Lusztig polynomials} (to complete the definition, set $P_{w,w}=1$ and $P_{v,w}=0$ for $v\not\leq w$). For $x,y\in \Wext$ we write 
$$
C_xC_y=\sum_{z\in \Wext}h_{x,y,z}C_z
$$
where $h_{x,y,z}\in\sR$ (and necessarily $\overline{h_{x,y,z}}=h_{x,y,z}$).

Let $\deg:\sR\to\mathbb{Z}\cup\{-\infty\}$ denote the degree in $\sq$. For example $\deg(\sq^{-3})=-3$ and $\deg(0)=-\infty$. Since $h_{x,y,z}$ is bar-invariant we have $\deg(h_{x,y,z})\geq 0$ (provided $h_{x,y,z}\neq 0$), and by \cite[\S 7.2]{Lus:85} that $\deg(h_{x,y,z})\leq \ell(\sw_0)$ for all $x,y,z\in\Wext$. 

\begin{defn}[{\cite{Lus:85}}]
\textit{Lusztig's $\ba$-function} is the function $\ba:\Wext\to \{0,1,\ldots,\ell(\sw_0)\}$ defined by
$$
\ba(z)=\max\{\deg(h_{x,y,z})\mid x,y\in \Wext\}.
$$
\end{defn}

Let $\leq_L$, $\leq_R$, and $\leq_{LR}$ be the standard left, right, and two-sided Kazhdan-Lusztig preorders on $\Wext$ defined in~\cite{KL:79} (for example, $\leq_L$ is the transitive closure of the relation $x\leq_L y$ if there exists $z\in \Wext$ such that $h_{z,y,x}\neq 0$). Let $\sim_*$ (for $*\in\{L,R,LR\}$) denote the associated equivalence relations, with $x\sim_* y$ if and only if $x\leq_*y$ and $y\leq_* x$. The equivalence classes of $\sim_L$, $\sim_R$, and $\sim_{LR}$ are called \textit{left cells}, \textit{right cells}, and \textit{two-sided cells}, respectively. We shall typically use the letter $\Delta$ to denote a two sided cell, and $\Gamma$ to denote a right cell. Note that if $\Gamma$ is a right cell then $\Gamma^{-1}$ is a left cell. 

Note that right (respectively left) cells are invariant under right (respectively left) multiplication by~$\Sigma$ (for example, $C_xC_{\sigma}=C_{x\sigma}$ and $C_{x\sigma}C_{\sigma^{-1}}=C_x$ show that $x\sigma\sim_R x$). Moreover, if $x\sim_R y$ (respectively $x\sim_L y$) then $\sigma x\sim_R\sigma y$ (respectively $x\sigma\sim_L y\sigma$), and so the set of right (respectively left) cells is permuted by left (respectively right) multiplication by~$\Sigma$.

The structure of cells in affine type~$\sA$ has been determined by Lusztig and Shi~\cite{Lus:85a,Shi:86}. In particular, the two sided cells of $\Wext$ are indexed by partitions $\la\vdash n+1$. We do not require the full details of this indexation. We make the following definitions. 
\begin{compactenum}
\item[--] Let $\Ga_{\la}=\{w\in \Wext\mid w\sim_R\sw_{\la'}\}$ (the right cell containing~$\sw_{\la'}$), 
\item[--] Let $\Delta_{\la}=\{w\in\Wext\mid w\sim_{LR}\sw_{\la'}\}$ (the two-sided cell containing $\sw_{\la'}$). 
\end{compactenum}
The two-sided order $\leq_{LR}$ on $\Wext$ induces a partial order on two-sided cells and we have $\Delta_{\la}\leq_{LR} \Delta_{\mu}$ if and only of $\la\leq \mu$ (see \cite[\S2.1]{Shi:96}).

We shall make use of the following result of Tanisaki and Xi. 

\begin{thm}[{\cite[Theorem~4.3]{TX:06}}]\label{thm:TanisakiXi}
Let $\la\vdash n+1$. The two-sided ideal 
$$\textnormal{span}_\sR\{C_{w}\mid w\leq_{LR} \sw_{\la'}\} \slash \textnormal{span}_{\sR}\{C_{w}\mid w<_{LR} \sw_{\la'}\}$$ is generated by the image of $C_{\sw_{\la'}}$ in the quotient. 
\end{thm}

Let $J\subseteq \{1,\ldots,n\}$. It is well known that the Kazhdan-Lusztig basis element $C_{\sw_J}$ is 
\begin{align}\label{eq:CJ}
C_{\sw_J}=\sq^{-\ell(w_J)}\sum_{w\in W_J}\sq^{\ell(w)}T_w=\sq^{\ell(w_J)}\sum_{w\in W_J}\sq^{-\ell(w)}X_w.
\end{align}
In particular, note that for $w\in W_J$ we have 
$
T_wC_{\sw_J}=C_{\sw_J}T_w=\sq^{\ell(w)}C_{\sw_J},
$
and hence
$
C_{\sw_J}^2=\sq^{\ell(w_J)}W_J(\sq^{-2})C_{\sw_J}=\sq^{-\ell(\sw_J)}W_J(\sq^2)C_{\sw_J}$, where
$
W_J(\sq^2)=\sum_{w\in W_J}\sq^{2\ell(w)}
$.

\begin{lemma}\label{lem:CJformula}
If $\ell(\sw_Jw)=\ell(w)-\ell(\sw_J)$ then 
$$
C_{\sw_J}C_w=\sq^{-\ell(\sw_J)}W_J(\sq^2)C_w.
$$
\end{lemma}

\begin{proof}
Since $\ell(sw)<\ell(w)$ for all $s\in J$ we have $C_sC_w=(\sq+\sq^{-1})C_w$ (see \cite[Theorem 6.6]{Lus:03}), and hence $T_sC_w=\sq C_w$. Thus $T_vC_w=\sq^{\ell(v)}C_w$ for all $v\in W_J$. The result follows by expanding $C_{\sw_J}$ in the $T_w$ basis using~(\ref{eq:CJ}). 
\end{proof}

\begin{defn}\label{def:gamma}
Following Lusztig~\cite{Lus:03} we define $\gamma_{x,y,z^{-1}}\in\ZZ$  by 
$$h_{x,y,z}=\gamma_{x,y,z^{-1}}\sq^{\ba(z)}+(\text{terms of degree $<\ba(z)$}).$$ 
\end{defn}
Let $\cD=\{w\in W\mid \ba(w)=-\deg(P_{e,w})\}$ (the set of \textit{distinguished involutions}). For $d\in \cD$ define $n_d\in\ZZ$ by
$$P_{e,d}=n_d\sq^{-\ba(d)}+\text{strictly smaller powers of $\sq$}.$$ 

In \cite{Lus:03}, Lusztig gave a series of statements  (now known as P1--P15) that capture the essential properties of cells, Lusztig's $\ba$-function, the $\ga$-coefficents and the distinguished involutions. In the case of type $\sA$ properties P1--P15 can be proved using the positivity properties of the Kazhdan-Lusztig basis (see~\cite[\S15]{Lus:03}). We record some of these properties for later use: \smallskip
\begin{compactenum}
\item[P2.] If $d \in \cD$ and $x,y\in \Wext$ satisfy $\gamma_{x,y,d}\neq 0$,
then $y=x^{-1}$.
\item[P3.] If $x\in\Wext$ then there exists a unique $d\in\cD$ such that $\gamma_{x,x^{-1},d}\neq 0$. 
\item[P4.]  If $z'\leq_{LR} z$ then $\ba(z')\geq \ba(z)$. In particular the $\ba$-function is constant on two-sided cells.
\item[P5.] If $d\in\cD$, $x\in\Wext$ with $\gamma_{x,x^{-1},d}\neq 0$ then $\gamma_{x,x^{-1},d}=n_d=\pm 1$ (by positivity, $n_d=1$ in~$\tilde{\sA}_n$). 
\item[P6.] If $d\in \cD$ then $d^{-1}=d$. 
\item[P7.] For any $x,y,z\in \Wext$, we have $\gamma_{x,y,z}=\gamma_{y,z,x}$.
\item[P8.] Let $x,y,z\in \Wext$ be such that $\gamma_{x,y,z}\neq 0$. Then
$y\sim_{R} x^{-1}$, $z \sim_{R} y^{-1}$, and $x\sim_{R} z^{-1}$.
\item[P11.] If $z'\leq_{LR}z$ and $\ba(z')=\ba(z)$ then $z'\sim_{LR}z$. 
\item[P13.] Each right cell $\Ga$ of $\Wext$ contains a unique element
$d\in \cD$, and $\ga_{x,x^{-1},d}\neq 0$ for all $x\in\Ga$. 
\end{compactenum}
\smallskip

In \cite{Lus:87} Lusztig defined an \textit{asymptotic ring} (a $\ZZ$-algebra) $\cJ$ as follows. The ring $\cJ$ has basis $\st_w$ with $w\in\Wext$, and 
$$
\st_x\st_y=\sum_{z\in\Wext }\gamma_{x,y,z^{-1}}\st_z,
$$
and, thanks to P1--P15, $\cJ$ turns out to be an associative $\ZZ$-algebra (called \textit{Lusztig's asymptotic algebra}). For each left or right cell $\Gamma$ the submodule $\cJ_{\Gamma}$ spanned by $\{\st_{w}\mid w\in\Gamma\}$ is a subring of~$\cJ$. Let $\cJ_{\la}$ be the subring spanned by $\{\st_{w}\mid w\in\Delta_{\la}\}$.

\section{The fundamental $\la$-alcove and the group $G_{\la}$}\label{sec:2}

In this section we explicitly describe structures introduced in~\cite{GLP:23} in the affine type $\sA$ case. In particular, we describe the fundamental $\la$-alcove~$\cA_{\la}$ and its symmetry group $\sT_{\la}\rtimes G_{\la}$, the set~$P^{(\la)}$ of~$\la$-weights, and the quotient $P/Q_{\la}$. We also introduce a delicate $\la$-dominance order on~$P^{(\la)}$. Next we study the rings $\sR[\zeta_{\la}]^{G_{\la}}$ and $\ZZ[\zeta_{\la}]^{G_{\la}}$ of $G_{\la}$-invariant polynomials, and recall the construction of generic induced representations $\pi_{\la}$ of $\Hext$ from~\cite{GLP:23}. We discuss the combinatorial model of $\la$-folded alcove paths, and recall a combinatorial formula for the matrix entries of $\pi_{\la}$.

\subsection{The $\la$-weights and the fundamental $\la$-alcove}\label{sec:fundalcove}

In this section we describe the abelian group $P/Q_{\la}$ of $\la$-weights, and the fundamental $\la$-alcove. 

\begin{defn} Let
$
Q_{\la}=\mathrm{span}_{\mathbb{Z}}\{\alpha_j\mid j \in J_{\la}\}. 
$ The elements of $P/Q_{\la}$ are called \textit{$\la$-weights}. 
\end{defn}

We now explicitly describe the $\la$-weights. If $j,j'$ are on the same row of $\bt_r(\la)$ then $e_j+Q_{\la}=e_{j'}+Q_{\la}$ (because $e_j-e_{j'}\in\Phi_{\la}$). Thus for $1\leq i\leq r(\la)$ define $\tilde{e}_i\in P/Q_{\la}$ by
$$
\tilde{e}_i=e_j+Q_{\la}\quad\text{for any $\la(i-1)+1\leq j\leq \la(i)$ (that is, any $j$ in the $i$th row of $\bt_r(\la)$)}.
$$
Thus $P/Q_{\la}=\mathrm{span}_{\mathbb{Z}}\{\tilde{e}_1,\ldots,\tilde{e}_{r(\la)}\}$, and the natural map $P\to P/Q_{\la}$ is given by 
\begin{align}\label{eq:naturalmap}
d_1e_1+\cdots+d_{n+1}e_{n+1}\mapsto a_1\tilde{e}_1+\cdots+a_{r(\la)}\tilde{e}_{r(\la)},\quad\text{where}\quad a_k=\sum_{i=\la(k-1)+1}^{\la(k)}d_i.
\end{align}
The equation $e_1+\cdots+e_{n+1}=0$ implies that 
\begin{align}\label{eq:zero}
\la_1\tilde{e}_1+\la_2\tilde{e}_2+\cdots+\la_{r(\la)}\tilde{e}_{r(\la)}=0.
\end{align}

\begin{example}\label{ex:triv}
Let $\la=(n+1)$. Then $Q_{\la}=Q$. By~(\ref{eq:zero}) we have $(n+1)\tilde{e}_1=0$, however note that $\tilde{e}_1\neq 0$ (in particular, note that $P/Q_{\la}$ is not torsion free). Thus $P/Q_{\la}=\ZZ\tilde{e}_1=\{0,\tilde{e}_1,\ldots,n\tilde{e}_1\}$, where $\tilde{e}_1=e_i+Q$ for any $1\leq i\leq n+1$. 
\end{example}

\begin{example}\label{ex:A3}
If $\la=(2,2)$ then 
$$
\bt_r(\la)=\begin{ytableau}
1&2\\
3&4
\end{ytableau}\ .
$$
We have $P/Q_{\la}=\mathrm{span}_{\ZZ}\{\tilde{e}_1,\tilde{e}_2\}$ where $\tilde{e}_1=e_1+Q_{\la}=e_2+Q_{\la}$ and $\tilde{e}_2=e_3+Q_{\la}=e_4+Q_{\la}$. By~(\ref{eq:zero}) we have $2\tilde{e}_1+2\tilde{e}_2=0$. However, note that $\tilde{e}_1+\tilde{e}_2\neq 0$ (because $e_1+e_3\notin Q_{\la}$). Thus the element $\tilde{e}_1+\tilde{e}_2$ has order~$2$.
\end{example}

For $\la\in\cP(n+1)$ the \textit{fundamental $\la$-alcove} is 
\begin{align}\label{eq:falc}
\cA_{\la}=\{v\in V\mid 0\leq\langle v,\alpha\rangle\leq 1\text{ for all $\alpha\in\Phi_{\la}^+$}\}.
\end{align}
Note that if $\la=(n+1)$ then $J_{\la}=\{1,\ldots,n\}$, and $\cA_{\la}=\cA_0$ (the fundamental alcove).

By \cite[Lemma~3.4]{GLP:23}, for each $\gamma\in P$ there exists a unique element $\gamma^{(\la)}\in \cA_{\la}\cap P$ with $\gamma-\gamma^{(\la)}\in Q_{\la}$ (we call $\gamma^{(\la)}$ the \textit{projection of $\gamma$ onto the fundamental $\la$-alcove}). The set
$$
P^{(\la)}=\cA_{\la}\cap P=\{\gamma^{(\la)}\mid \gamma\in P\}
$$
is closely related to the set $P/Q_{\la}$ of $\la$-weights (see Proposition~\ref{prop:bijectionP} below). By \cite[Lemma~3.6]{GLP:23} an element $\gamma=\sum_{i=1}^na_i\omega_i\in P$ lies in $P^{(\la)}$ if and only if whenever $j\in J_{\la}$ we have $a_j\in \{0,1\}$, and for each connected component $K$ of $J_{\la}$ there is at most one $k\in K$ with $a_k=1$ (the \textit{connected components} of $J_{\la}$ are the sets $\{\la(i-1)+1,\ldots,\la(i)-1\}$ for $1\leq i\leq r(\la)$).

For $\gamma\in P^{(\la)}$ let 
$
J_{\la}(\gamma)=\{j\in J_{\la}\mid \langle\gamma,\alpha_j\rangle=1\},
$
and following~\cite{GLP:23} let 
\begin{align}\label{eq:defny}
\sy_{\gamma}=\sw_{J_{\la}\backslash J_{\la}(\gamma)}\sw_{J_{\la}}\quad\text{and}\quad \tau_{\gamma}=t_{\gamma}\sy_{\gamma}.
\end{align}
The elements $\tau_{\ga}$ are called \textit{pseudo-translations}, and by \cite[Lemma~3.14]{GLP:23} we have
$$
\tau_{\gamma}\tau_{\gamma'}=\tau_{(\gamma+\gamma')^{(\la)}}\quad\text{ for all $\gamma,\gamma'\in P^{(\la)}$}.
$$ 
By \cite[Lemma~3.10]{GLP:23} we have $\Phi(\sy_{\gamma})=\{\alpha\in\Phi_{\la}^+\mid \langle\gamma,\alpha\rangle=1\}$. Note that $\tau_{\gamma}$ and $\sy_{\gamma}$ depend on the partition $\la$, however this dependence is suppressed without risk of confusion.

\begin{example}
If $\Phi$ is of type $\sA_2$ and $\la=(2,1)$ then $\cA_{\la}$ (shaded green) and $P^{(\la)}$ (represented as bullets) are as illustrated in Figure~\ref{fig:A2}. If $\gamma\in P^{(\la)}$ lies on the hyperplane $H_{\alpha_1,1}$ then $\sy_{\ga}=s_1$.
\begin{figure}[H]
\centering
\begin{tikzpicture}[scale=1.4]
\clip ({-5*0.5},{-2.5*0.866}) rectangle + ({5},{5*0.866});
    \path [fill=green2] (-2.5,{-3*0.866})--(-1.5,{-3*0.866})--(1.5,{3*0.866})--(0.5,{3*0.866})--(-2.5,{-3*0.866});
     \path [fill=gray!90] (0,0) -- (-0.5,0.866) -- (0.5,0.866) -- (0,0);
    \draw (2.5, {-3*0.866})--( 3.5, {-1*0.866} );
    \draw (1.5, {-3*0.866})--( 3.5, {1*0.866} );
    \draw  (0.5, {-3*0.866})--( 3.5, {3*0.866} );
    \draw  (-0.5, {-3*0.866})--( 3, {4*0.866} );
    \draw  [line width=2pt](-1.5, {-3*0.866})--( 2, {4*0.866} );
  \draw  (-2.5, {-3*0.866})--( 1, {4*0.866} );
    \draw  (-3.5, {-3*0.866})--(0, {4*0.866} );
    \draw  (-3.5, {-1*0.866})--(-1, {4*0.866} );
    \draw  (-3.5, {1*0.866})--(-2, {4*0.866} );
    \draw  (-3.5, {3*0.866})--(-3, {4*0.866} );
   \draw (-2.5, {-3*0.866})--( -3.5, {-1*0.866} );
    \draw (-1.5, {-3*0.866})--( -3.5, {1*0.866} );
    \draw  (-0.5, {-3*0.866})--( -3.5, {3*0.866} );
    \draw  (0.5, {-3*0.866})--( -3, {4*0.866} );
  \draw (1.5, {-3*0.866})--( -2, {4*0.866} );
    \draw  (2.5, {-3*0.866})--( -1, {4*0.866} );
    \draw  (3.5, {-3*0.866})--(0, {4*0.866} );
    \draw  (3.5, {-1*0.866})--(1, {4*0.866} );
    \draw  (3.5, {1*0.866})--(2, {4*0.866} );
    \draw  (3.5, {3*0.866})--(3, {4*0.866} );
    \draw (-3.5, -2.598)--( 3.5, -2.598);
    \draw (-3.5, -1.732)--( 3.5, -1.732);
    \draw (-3.5, -0.866)--( 3.5, -0.866);
    \draw (-3.5, 0)--( 3.5, 0);
    \draw (-3.5, 3.464)--( 3.5, 3.464 );
    \draw (-3.5, 2.598)--( 3.5, 2.598);
    \draw (-3.5, 1.732)--( 3.5, 1.732);
        \draw (-3.5, 0.866)--( 3.5, 0.866);
    \draw [latex-latex, line width=1pt] (-1.5,-0.866)--(1.5,0.866);
    \draw [latex-latex, line width=1pt] (1.5,-0.866)--(-1.5,0.866);
    \draw [latex-latex, line width=1pt] (0,-1.732)--(0,1.732);
    \node at (-1.8,1) {\small{$\alpha_1$}};
    \node at (1.8,1) {\small{$\alpha_2$}};
    \node at (-0.8,1.05) {\small{$\omega_1$}};
    \node at (0.9,1.05) {\small{$\omega_2$}};
    \node at ({0*0.5},{0*0.866}) {\Large{$\bullet$}};
     \node at ({1*0.5},{1*0.866}) {\Large{$\bullet$}};
      \node at ({2*0.5},{2*0.866}) {\Large{$\bullet$}};
     \node at ({-1*0.5},{-1*0.866}) {\Large{$\bullet$}};
      \node at ({-2*0.5},{-2*0.866}) {\Large{$\bullet$}};
        \node at ({-0.5},{1*0.866}) {\Large{$\bullet$}};
         \node at ({0},{2*0.866}) {\Large{$\bullet$}};
          \node at ({-1},{0*0.866}) {\Large{$\bullet$}};
           \node at ({-1.5},{-1*0.866}) {\Large{$\bullet$}};
            \node at ({-2},{-2*0.866}) {\Large{$\bullet$}};
\end{tikzpicture}
\caption{The set $P^{(\la)}$ for type $\sA_2$ with $\la=(2,1)$}\label{fig:A2}
\end{figure}
\end{example}

\begin{prop}[{\cite[Corollary~3.15]{GLP:23}}]\label{prop:latticeiso} The set $\mathsf{T}_{\la}=\{\tau_{\gamma}\mid \gamma\in P^{(\la)}\}$ of pseudo-translations is an abelian group, and $\mathsf{T}_{\la}\cong P/Q_{\la}$ with $\tau_{\gamma}\mapsto \gamma+Q_{\la}$. 
\end{prop}

Thus $P^{(\la)}$ inherits the structure of an abelian group with operation $\ga\oplus\ga'=(\ga+\ga')^{(\la)}$, and $P^{(\la)}\cong P/Q_{\la}$ via the map $\ga\mapsto \ga+Q_{\la}$. The following description of the inverse isomorphism $P/Q_{\la}\to P^{(\la)}$ is useful (combined with~(\ref{eq:naturalmap}) it gives an explicit expression for the projection map $P\to P^{(\la)}$). Recall that $\la[i,j]$ denotes the element in row $i$ and column $j$ of~$\bt_r(\la)$. 

\begin{prop}\label{prop:bijectionP}
The isomorphism $P/Q_{\la}\xrightarrow{\sim} P^{(\la)}$ is given by 
$$
a_1\tilde{e}_1+\cdots+a_{r(\la)}\tilde{e}_{r(\la)}\mapsto \sum_{k=1}^{r(\la)}\left((b_k+1)(e_{\la[k,1]}+\cdots+e_{\la[k,c_k]})+b_k(e_{\la[k,c_k+1]}+\cdots+e_{\la[k,\la_k]})\right)
$$ 
where $b_k$ and $c_k$ are defined by $a_k=\la_kb_k+c_k$ with $0\leq c_k<\la_k$. 
\end{prop}

\begin{proof}
Let $\ga$ denote the right hand side of the map. Since $e_j+Q_{\la}=\tilde{e}_k$ for all $j$ in the $k$th row of $\bt_r(\la)$ we have $\gamma+Q_{\la}=a_1\tilde{e}_1+\cdots+a_{r(\la)}\tilde{e}_{r(\la)}$. Moreover, if $\alpha=e_i-e_j\in \Phi_{\la}^+$ then $i,j$ lie in a common row of $\bt_r(\la)$, say $i=\la[k,i']$ and $j=\la[k,j']$ with $1\leq i'<j'\leq \la_k$. Then $\langle\ga,\alpha\rangle\in\{0,1\}$, and so $\ga\in P^{(\la)}$. Hence the result. 
\end{proof}

\begin{remark}\label{rem:bijectionP}
Identify elements $\ga=d_1e_1+\cdots+d_{n+1}e_{n+1}$ with tableau of shape $\la$ filled with $d_1,d_2,\ldots,d_{n+1}$ along rows. Since $e_1+\cdots+e_{n+1}=0$, two filled tableaux are considered equal if they differ by a multiple of the constant tableau with every entry~$1$. Proposition~\ref{prop:bijectionP} shows that~$P^{(\la)}$ consists precisely of the elements whose associated tableau has $k$th row (for $1\leq k\leq r(\la)$) of the form
$$
\ytableausetup{boxsize=3em}
\scalebox{1}[.5]{
\begin{ytableau}
\s{2}{b_k+1}&\s{2}{b_k+1}&\s{2}{\cdots}&\s{2}{b_k+1}&\s{2}{b_k}&\s{2}{b_k}&\s{2}{\cdots}&\s{2}{b_k}
\end{ytableau}}
\ytableausetup{boxsize=normal}
$$
where $b_k+1$ occurs $c_k$ times, with $0\leq c_k<\la_k$. Moreover, the isomorphism $P/Q_{\la}\xrightarrow{\sim}P^{(\la)}$ is given by mapping $a_1\tilde{e}_1+\cdots+a_{r(\la)}\tilde{e}_{r(\la)}$ to the above tableau (where $a_k=\la_kb_k+c_k$). For example, if $\la=(7,3)$ then 
$$
39\tilde{e}_1+11\tilde{e}_2\mapsto\begin{ytableau}
6&6&6&6&5&5&5\\
4&4&3
\end{ytableau}\ .
$$
\end{remark}

\subsection{The group $G_{\la}$}

Let 
$$
G_{\la}=\{g\in \Wfin \mid g\cA_{\la}=\cA_{\la}\}
$$
be the subgroup of $\Wfin$ stabilising the fundamental $\la$-alcove. By \cite[Theorem~3.20]{GLP:23} we have
\begin{align}\label{eq:stabiliserho}
G_{\la}=\{g\in \Wfin \mid g\Phi_{\la}^+=\Phi_{\la}^+\},\quad\text{and so}\quad \text{$g\rho_{\la}=\rho_{\la}$ for all $g\in G_{\la}$}
\end{align}
with $\rho_{\lambda}$ as in~(\ref{eq:rholambda}). Moreover $g\in G_{\la}$ maps the simple roots of $\Phi_{\la}$ to the simple roots of $\Phi_{\la}$, and hence induces a permutation of~$J_{\la}$. This permutation maps connected components of $J_{\la}$ to connected components of~$J_{\la}$. Furthermore, by \cite[Lemma~3.21]{GLP:23} for $\gamma\in P^{(\la)}$ and $g\in G_{\la}$ we have
$$
g\sy_{\gamma}g^{-1}=\sy_{g\gamma},\quad g\tau_{\gamma}g^{-1}=\tau_{g\gamma},\quad\text{and}\quad \ell(\tau_{g\gamma})=\ell(\tau_{\gamma}). 
$$
The subgroup of $\Wext$ stabilising $\cA_{\la}$ is (see~\cite[Corollary~3.22]{GLP:23})
\begin{align}\label{eq:lambdaaffine}
\{w\in \Wext\mid w\cA_{\la}=\cA_{\la}\}=\mathsf{T}_{\la}\rtimes G_{\la}.
\end{align}
The group $G_{\la}$ has the following explicit description in terms of tableaux. Let 
$$\lengths(\la)=\{\la_i\mid 1\leq i\leq r(\la)\}$$ be the set of lengths of the rows of $\la$, and for $l\in \lengths(\la)$ let
$$
\kappa_{\la}(l)=\#\{1\leq i\leq r(\la)\mid \la_i=l\}
$$
be the number of rows of length~$l$. For example, if $\la=(5,5,3,3,3,2)$ then $\lengths(\la)=\{2,3,5\}$, and $\kappa_{\la}(2)=1$, $\kappa_{\la}(3)=3$, and $\kappa_{\la}(5)=2$. 

\begin{prop}\label{prop:Glatype}
The group $G_{\la}$ is the subgroup of $\mathfrak{S}_{n+1}$ stabilising each column of $\bt_r(\la)$ and permuting the set of rows of $\bt_r(\la)$. Thus $G_{\la}$ is a Coxeter group of type $\prod_{l\in\lengths(\la)}\sA_{\kappa_{\la}(l)-1}$.
\end{prop}

\begin{proof}
Since $g\in G_{\la}$ maps connected components of $J_{\la}$ to connected components of $J_{\la}$ it maps rows of $\bt_r(\la)$ to rows of $\bt_r(\la)$. Since $g$ also maps positive $\Phi_{\la}$ roots to positive $\Phi_{\la}$ roots it follows that $g$ preserves columns of $\bt_r(\la)$. Conversely, any $g\in \Wfin=\mathfrak{S}_{n+1}$ that stabilises columns and acts on the set of rows of $\bt_r(\la)$ satisfies $g\Phi_{\la}^+=\Phi_{\la}^+$, and hence $g\in G_{\la}$. Thus $G_{\la}\cong \prod_{l\in\lengths(\la)}\mathfrak{S}_{\kappa_{\la}(l)}$. 
\end{proof}

\begin{example}\label{ex:running2} Let $\la=(6,6,4,4,4,2,1,1)$, and so
$$
\bt_r(\la)=\begin{ytableau}
1&2&3&4&5&6\\
7&8&9&10&11&12\\
13&14&15&16\\
17&18&19&20\\
21&22&23&24\\
25&26\\
27\\
28
\end{ytableau}
$$
Then $G_{\la}\cong \mathfrak{S}_2\times \mathfrak{S}_3\times \mathfrak{S}_2$ is generated by the involutions $(1,7)(2,8)(3,9)(4,10)(5,11)(6,12)$, $(13,17)(14,18)(15,19)(16,20)$, $(17,21)(18,22)(19,23)(20,24)$, and $(27,28)$. 
\end{example}

\begin{prop}\label{prop:Gla1}
We have 
 $G_{\la}=\{g\in \su_{\la}W_{\la'}\su_{\la}^{-1}\mid g\Phi_{\la}=\Phi_{\la}\}.$ 
 In particular, $\su_{\la}^{-1}G_{\la}\su_{\la}\leq W_{\la'}$. 
\end{prop}

\begin{proof}
It follows from the one-line description of $\su_{\la}$ that $\su_{\la}^{-1}G_{\la}\su_{\la}$ is the subgroup of $\mathfrak{S}_{n+1}$ stabilising each row of $\bt_r(\la')$ and acting on the set of columns of $\bt_r(\la')$, and so in particular $\su_{\la}^{-1}G_{\la}\su_{\la}\leq W_{\la'}$. Thus if $g\in G_{\la}$ then $g\in \su_{\la}W_{\la'}\su_{\la}^{-1}$ and $g\Phi_{\la}=\Phi_{\la}$ (as $g\Phi_{\la}^+=\Phi_{\la}^+$). On the other hand, suppose that $g=\su_{\la}w\su_{\la}^{-1}$ with $w\in W_{\la'}$ and that $g\Phi_{\la}=\Phi_{\la}$. If there is $\alpha\in\Phi_{\la}^+$ such that $g\alpha=-\beta\in-\Phi^+$ then we have $w\su_{\la}^{-1}\alpha=-\su_{\la}^{-1}\beta$. But $w\su_{\la}^{-1}\alpha>0$ (as $\su_{\la}w^{-1}\in {^\la}W$) and so $\su_{\la}^{-1}\beta<0$. However since $\beta\in\Phi_{\la}^+$ we have $\su_{\la}^{-1}\beta>0$ (as $\su_{\la}\in {^\la}W$), a contradiction. 
\end{proof}

\begin{remark}\label{rem:columnstabiliser}
It follows from the definition of $\su_{\la}$ that the group $\su_{\la}W_{\la'}\su_{\la}^{-1}$ is the subgroup of the symmetric group stabilising each column of $\bt_r(\la)$. 
\end{remark}

\begin{defn}\label{defn:Glaroots}
The \textit{$G_{\la}$-root system} is the subset $\Phi_{G_{\la}}=\Phi_{G_{\la}}^+\cup(-\Phi_{G_{\la}}^+)$ of $P/Q_{\la}$ with
\begin{align*}
\Phi_{G_{\la}}^+=\{\tilde{e_i}-\tilde{e}_j\mid 1\leq i<j\leq r(\la)\text{ with }\la_i=\la_j\}.
\end{align*}
Note that $\Phi_{G_{\la}}$ is not a true root system because the group $P/Q_{\la}$ can have torsion (see Example~\ref{ex:A3}), however it plays an analogous role to a root system in the theory below. 
\end{defn}

\subsection{Dominant $\la$-weights and the $\la$-dominance order}\label{sec:lambdadominance}

Since $G_{\la}$ preserves $Q_{\la}$ (by (\ref{eq:stabiliserho})) the equation $g(\gamma+Q_{\la})=g\gamma+Q_{\la}$ defines an action of $G_{\la}$ on the set~$P/Q_{\la}$ of $\la$-weights. More explicitly, this action is given by permuting the vectors $\tilde{e}_1,\ldots,\tilde{e}_{r(\la)}$ subject to the constraint that if $g\tilde{e}_i=\tilde{e}_j$ then $\la_i=\la_j$.

A fundamental domain for the action of $G_{\la}$ on $P/Q_{\la}$ is given by 
$$
(P/Q_{\la})_+=\{a_1\tilde{e}_1+\cdots+a_{r(\la)}\tilde{e}_{r(\la)}\in P/Q_{\la}\mid a_i\geq a_j\text{ if $i<j$ with $\la_i=\la_j$}\}.
$$
We call the elements of $(P/Q_{\la})_+$ the \textit{dominant $\la$-weights}.

By definition the group $G_{\la}$ also acts on $P^{(\la)}$. Let $P^{(\la)}_+$ be the fundamental domain for this action corresponding to the fundamental domain $(P/Q_{\la})_+$ under the isomorphism $P/Q_{\la}\cong P^{(\la)}$ (see Proposition~\ref{prop:bijectionP}). That is,
\begin{align}\label{eq:dominantweights}
P^{(\la)}_+=\{\gamma\in P^{(\la)}\mid \gamma+Q_{\la}\in (P/Q_{\la})_+\}.
\end{align}

\begin{example}\label{ex:A32}
Let $\la=(2,2)$ as in Example~\ref{ex:A3}. Then $G_{\la}$ is of type $\sA_1$, generated by $\tilde{s}=s_2s_1s_3s_2$, and $\Phi_{G_{\la}}^+=\{\tilde{e}_1-\tilde{e}_2\}$. We have $
(P/Q_{\la})_+=\{a_1\tilde{e}_1+a_2\tilde{e}_2\mid a_1\geq a_2\}$ and 
$$
P^{(\la)}_+=(\ZZ_{\geq 0}\omega_2)\cup(\omega_1+\ZZ_{\geq 0}\omega_2)\cup(\omega_3+\ZZ_{\geq 0}\omega_2)\cup(\omega_1+\omega_3+\ZZ_{\geq -1}\omega_2).
$$
\end{example}

\begin{example}
Let $\la=(6,6,4,4,4,2,1,1)$, as in Example~\ref{ex:running2}. Then 
\begin{align*}
\Phi_{G_{\la}}^+&=\{\tilde{e}_1-\tilde{e}_2,\tilde{e}_3-\tilde{e}_4,\tilde{e}_3-\tilde{e}_5,\tilde{e}_4-\tilde{e}_5,\tilde{e}_7-\tilde{e}_8\}\\
(P/Q_{\la})_+&=\{a_1\tilde{e}_1+\cdots+a_8\tilde{e}_8\mid a_1\geq a_2,\,a_3\geq a_4\geq a_5,\,\text{and}\,a_7\geq a_8\}.
\end{align*}
\end{example}

Let 
$$Q^{\la}=\mathrm{span}_{\ZZ}(\Phi_{G_{\la}})\quad\text{and}\quad Q^{\la}_+=\mathrm{span}_{\ZZ_{\geq 0}}(\Phi_{G_{\la}}^+).
$$
Define the \textit{$\la$-dominance order} $\peq_{\la}$ on $P/Q_{\la}$ by
$$
\gamma+Q_{\la}\peq_{\la}\gamma'+Q_{\la}\quad\text{if and only if}\quad \gamma'-\gamma+Q_{\la}\in Q^{\la}_+.
$$
The $\la$-dominance order can also be considered as a partial order on $P^{(\la)}$ via Proposition~\ref{prop:bijectionP}. 

The following lemmas give conditions for membership of $Q_{\la}$ and $Q^{\la}_+$, and hence give a more concrete understanding of the partial order $\peq_{\la}$. The straightforward proofs are omitted.

\begin{lemma}\label{lem:membershipQ_la}
Let $\gamma=\sum_{i=1}^{n+1}a_ie_i\in Q$, with the expression chosen so that $a_1+\cdots+a_{n+1}=0$. Then $\gamma\in Q_{\la}$ if and only if $\sum_{i=\la(k-1)+1}^{\la(k)}a_i=0$ for each $1\leq k\leq r(\la)$.
\end{lemma}

\begin{lemma}\label{lem:membershipQ^la}
We have $\gamma+Q_{\la}\in Q^{\la}$ if and only if there is an expression $\gamma+Q_{\la}=\sum_{i=1}^{r(\la)}a_i\tilde{e}_i$ with 
$$
\sum_{1\leq i\leq r(\la),\,\la_i=l}a_i=0\quad\text{for all $l\in\lengths(\la)$},
$$
and moreover $\gamma+Q_{\la}\in Q^{\la}_+$ if and only if $a_1+\cdots+a_i\geq 0$ for all $1\leq i\leq r(\la)$. 
\end{lemma}

In particular, note that if $\gamma+Q_{\la}\in Q^{\la}$ then necessarily $\gamma\in Q$. If $\la=(1^{n+1})$ then $Q_{\la}=\{0\}$ and so $P/Q_{\la}=P$. In this case $\peq_{\la}$ is the usual dominance order $\peq$ on $P$ given by $\gamma\peq \gamma'$ if and only if $\gamma'-\gamma\in Q_+$, where $Q_+=\ZZ_{\geq 0}\alpha_1+\cdots+\ZZ_{\geq 0}\alpha_n$ (this notion is related to, but distinct from, the dominance order $\leq$ on partitions).

\begin{example}
Let $\la=(6,6,4,4,4,2,1,1)$ as in Example~\ref{ex:running2}. Then $Q^{\la}_+$ consists of the elements $a_1\tilde{e}_1+\cdots+a_8\tilde{e}_8$ with $a_i\in\ZZ$ satisfying $a_1+a_2=a_3+a_4+a_5=a_6=a_7+a_8=0$, $a_1\geq 0$, $a_3\geq 0$, $a_3+a_4\geq 0$, and $a_7\geq 0$.
\end{example}

\begin{example}
Let $\la=(n+1)$ (see Example~\ref{ex:triv}). Then $Q_{\la}=Q$ and $P/Q_{\la}=\{0,\tilde{e}_1,\ldots,n\tilde{e}_1\}$. We have $Q^{\la}=\{0\}$, and so for $\gamma,\gamma'\in P/Q_{\la}$ we have $\gamma\peq_{\la}\gamma'$ if and only if $\gamma=\gamma'$. 
\end{example}

\subsection{$\la$-folded alcove paths}

In \cite{GLP:23,GP:19} we introduced the combinatorial model of $\la$-folded alcove paths. This theory will play an important role in the present paper via the formula given in Theorem~\ref{thm:pathformula} (providing a combinatorial formula for the matrix entries of certain generic representations of $\Hext$ induced from Levi subalgebras).

\begin{defn}\label{defn:Jfold}
Let $\vec w=s_{i_1}s_{i_2}\cdots s_{i_{\ell}}\pi$ be an expression for $w\in \Wext$ (not necessarily reduced) with $\pi\in \Sigma$. A \textit{$\la$-folded alcove path of type~$\vec w$ starting at $v$} is a sequence $p=(v_0,v_1,\ldots,v_{\ell},v_{\ell}\pi)$ with $v_0,\ldots,v_{\ell}\in \Wext$ such that, for $0\leq k\leq \ell$,
\begin{compactenum}[$(1)$]
\item $v_k\cA_0\subseteq \cA_{\la}$, 
\item $v_k\in\{v_{k-1},v_{k-1}s_{i_k}\}$, and if $v_{k-1}=v_k$ then either:
\begin{compactenum}[$(a)$]
\item $v_{k-1}s_{i_k}\cA_0\not\subseteq \cA_{\la}$, or
\item $v_{k-1}s_{i_k}\cA_0\subseteq \cA_{\la}$ and the alcove $v_{k-1}\cA_0$ is on the positive side of the hyperplane separating the alcoves $v_{k-1}\cA_0$ and $v_{k-1}s_{i_k}\cA_0$. 
\end{compactenum}
\end{compactenum}
If  $p=(v_0,\ldots,v_{\ell},v_{\ell}\pi)$ is a $\la$-folded alcove path we define:
\begin{compactenum}[--]
\item the \textit{start} of $p$ is $\mathrm{start}(p)=v_0$ and the \textit{end} of $p$ is $\mathrm{end}(p)=v_{\ell}\pi$. 
\item the \textit{length} of $p$ is $\ell$ (note that the final step $(v_{\ell},v_{\ell}\pi)$ does not count towards length). 
\item the \textit{weight} of $p$ is $\wt(\mathrm{end}(p))$ (necessarily $\wt(p)\in P^{(\la)}$ as $v_{\ell}\pi\cA_{0}\subseteq\cA_{\la}$). 
\item the \textit{final direction} of $p$ is $\theta^{\la}(p)$, where $\theta(p)=\theta(\mathrm{end}(p))$ and $\theta(p)=\theta_{\la}(p)\theta^{\la}(p)$ with $\theta_{\la}(p)\in W_{\la}$ and $\theta^{\la}(p)\in {^\la}W$. 
\end{compactenum}
\end{defn}

Note that a $\la$-folded alcove path is, by definition, a sequence $p=(v_0,v_1,\ldots,v_{\ell},v_{\ell}\pi)$ of elements of $\Wext$. There is an associated sequence $v_0\cA_0,v_1\cA_0,\ldots,v_{\ell}\cA_0,v_{\ell}\pi\cA_0$ of alcoves, with each alcove either adjacent to or equal to the preceding one. See \cite{GLP:23,GP:19} for examples.

Note that if $\la=(1^{n+1})$ then $\cA_{\la}=V$ and so part (1) of the definition is trivially satisfied, and part (2)(a) is vacuous. Thus $(1^{n+1})$-folded alcove paths are the same as classical ``positively folded alcove paths'' (in the sense of Ram~\cite{Ram:06}).

Let $\vec w=s_{i_1}s_{i_2}\cdots s_{i_{\ell}}\pi$ and let $p=(v_0,\ldots,v_{\ell},v_{\ell}\pi)$ be a $\lambda$-folded alcove path of type $\vec{w}$. The index~$k\in\{1,2,\ldots,\ell\}$ is:
\begin{compactenum}[$(1)$]
\item a \textit{positive (respectively, negative) $i_k$-crossing} if $v_k=v_{k-1}s_{i_k}$ and $v_{k}\cA_0$ is on the positive (respectively, negative) side of the hyperplane separating the alcoves $v_{k-1}\cA_0$ and $v_k\cA_0$;
\item a \textit{(positive) $i_k$-fold} if $v_k=v_{k-1}$ and $v_{k-1}s_{i_k}\cA_0\subseteq\cA_{\la}$ (in which case $v_{k-1}\cA_0$ is necessarily on the positive side of the hyperplane separating $v_{k-1}\cA_0$ and $v_{k-1}s_{i_k}\cA_0$);
\item a \textit{bounce} if $v_k=v_{k-1}$ with $v_{k-1}\cA_0\subseteq \cA_{\la}$ and $v_{k-1}s_{i_k}\cA_0\not\subseteq \cA_{\la}$. 
\end{compactenum}

Less formally, these steps are denoted as follows (where $x=v_{k-1}$ and $s=s_{i_k}$):
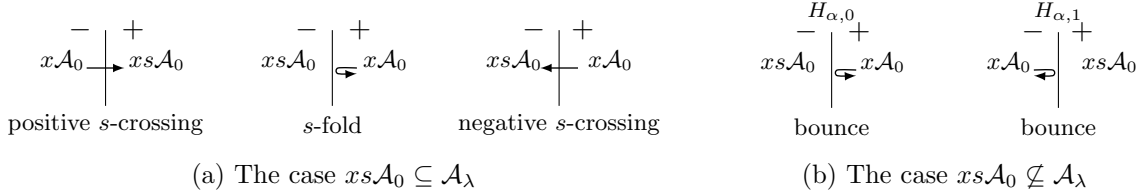
\begin{figure}[H]
\begin{subfigure}{.6\textwidth}
\begin{center}
\begin{tikzpicture}[xscale=0.5, yscale=0.5]
\draw (-6,-1)--(-6,1);
\draw[line width=0.5pt,-latex](-6.5,0)--(-5.5,0);
\node at (-6.8,1) {{ $-$}};
\node at (-7.25,.2) {\footnotesize{ $x\cA_0$}};
\node at (-4.75,.2) {\footnotesize{ $xs\cA_0$}};
\node at (-5.4,1) {{ $+$}};
\node at (-6,-1.5) {\footnotesize{positive $s$-crossing}};
\draw (-0,-1)--(-0,1);
\draw[line width=0.5pt,-latex] plot[smooth] coordinates {(.65,0)(0.15,0)(0.15,-.15)(.65,-.15)};
\node at (-.8,1){{ $-$}};
\node at (-1.25,.2){\footnotesize{ $xs\cA_0$}};
\node at (1.3,.2){\footnotesize{ $x\cA_0$}};
\node at (.6,1){{ $+$}};
\node at (0,-1.5){\footnotesize{$s$-fold}};
\draw (6,-1)--(6,1);
\draw[line width=0.5pt,-latex] (6.5,0)--(5.5,0);
\node at (6.5,1){{ $+$}};
\node at (7.25,.2){\footnotesize{ $x\cA_0$}};
\node at (4.7,.2){\footnotesize{ $xs\cA_0$}};
\node at (5.2,1){{ $-$}};
\node at (-3,1.5){\scriptsize{\vphantom{$H_{\alpha,0}$}}};
\node at (6,-1.5){\footnotesize{negative $s$-crossing}};
\end{tikzpicture}
\caption{The case $xs\cA_0\subseteq\cA_{\la}$}
\end{center}
\end{subfigure}
\begin{subfigure}{0.4\textwidth}
\begin{center}
\begin{tikzpicture}[xscale=0.5, yscale=0.5]
\draw (3,-1)--(3,1);
\draw[line width=0.5pt,-latex] plot[smooth] coordinates {(2.35,0)(2.85,0)(2.85,-.15)(2.35,-.15)};
\node at (3.5,1){{ $+$}};
\node at (4.25,.2){\footnotesize{ $xs\cA_0$}};
\node at (1.7,.2){\footnotesize{$x\cA_0$}};
\node at (2.2,1){{ $-$}};
\node at (3,1.5){\scriptsize{$H_{\alpha,1}$}};
\node at (3,-1.5){\footnotesize{bounce}};
\draw(-3,-1)--(-3,1);
\draw[line width=0.5pt,-latex] plot[smooth] coordinates {(-2.35,0)(-2.85,0)(-2.85,-.15)(-2.35,-.15)};
\node at (-3.8,1){{ $-$}};
\node at (-4.25,.2){\footnotesize{ $xs\cA_0$}};
\node at (-1.75,.2){\footnotesize{ $x\cA_0$}};
\node at (-2.4,1){{ $+$}};
\node at (-3,1.5){\scriptsize{$H_{\alpha,0}$}};
\node at (-3,-1.5){\footnotesize{bounce}};
\end{tikzpicture}
\caption{The case $xs\cA_0\not\subseteq\cA_{\la}$}
\end{center}
\end{subfigure}
\end{figure}
Bounces play a different role in the theory to folds, and so we emphasise the distinction between these two concepts. Put briefly, all of the interactions a path makes with the walls of $\cA_{\la}$ are bounces, and the folds can only occur in the ``interior'' of $\cA_{\la}$.

Let $p$ be a $\la$-folded alcove path. Let
\begin{align*}
f(p)=\#(\text{folds in $p$})\quad\text{and}\quad 
b(p)=\#(\text{bounces in $p$}).
\end{align*}
For $u\in\Wext$ with $u\cA_0\subseteq \cA_{\la}$, and $v\in {^\la}W$, let
\begin{align*}
\mathcal{P}_{\la}(\vec{w},u)&=\{\text{all $\la$-folded alcove paths of type $\vec{w}$ starting at $u$}\}\\
\cP_{\la}(\vec{w},u)_v&=\{p\in\cP_{\la}(\vec{w},u)\mid \theta^{\la}(p)=v\}.
\end{align*}

\subsection{The ring $\ZZ[\zeta_{\la}]^{G_{\la}}$ and $G_{\la}$-Schur functions}\label{sec:rings}

 Let $\zeta_1,\ldots,\zeta_{n+1}$ be commuting invertible indeterminates with $\zeta_1\cdots\zeta_{n+1}=1$, and for $\gamma\in P$ let
$\zeta^{\gamma}=\zeta_1^{a_1}\cdots\zeta_{n+1}^{a_{n+1}}$ if $\gamma=\sum_{i=1}^{n+1}a_ie_i$. For $\la\vdash n+1$ let $\mathcal{I}_{\la}$ denote the ideal of the Laurent polynomial ring $\sR[\zeta_1^{\pm1},\ldots,\zeta_{n+1}^{\pm 1}]$ generated by the elements $\zeta^{\alpha_j}-1$ for $j\in J_{\la}$. Let
$$
\sR[\zeta_{\la}]=\sR[\zeta_1^{\pm1},\ldots,\zeta_{n+1}^{\pm 1}]/\mathcal{I}_{\la},
$$
and similarly define $\ZZ[\zeta_{\la}]$. Write $\zeta_{\la}^{\gamma}=\zeta^{\gamma}+\mathcal{I}_{\la}$. Thus $\zeta_{\la}^{\gamma}=1$ for all $\gamma\in Q_{\la}$, and so $\zeta_{\la}^{\gamma}$ depends only on the coset $\gamma+Q_{\la}$. Indeed, $\sR[\zeta_{\la}]\cong \sR[P/Q_{\la}]$ is the group ring of $P/Q_{\la}$ over $\sR$. In particular, we have $\zeta_{\la}^{\gamma}=\zeta_{\la}^{\gamma^{(\la)}}$ for all $\gamma\in P$, and we define $\zeta_{\la}^{\gamma+Q_{\la}}=\zeta_{\la}^{\gamma}$ for $\gamma+Q_{\la}\in P/Q_{\la}$. 

A more explicit description of the ring $\sR[\zeta_{\la}]$ is as follows. Since $\zeta_{\la}^{\alpha}=1$ for all $\alpha\in \Phi_{\la}$ we have $\zeta_j+\mathcal{I}_{\la}=\zeta_{j'}+\mathcal{I}_{\la}$ whenever $j$ and $j'$ lie in the same row of $\bt_r(\la)$. Define $z_1,\ldots,z_{r(\la)}\in \sR[\zeta_{\la}]$ by $z_i=\zeta_j+\mathcal{I}_{\la}$ for any $j$ in the $i$th row of $\bt_r(\la)$ (that is, $z_i=\zeta_{\la}^{e_j}$). Then the relation $e_1+\cdots+e_{n+1}=0$ gives
$
z_1^{\la_1}z_2^{\la_2}\cdots z_{r(\la)}^{\la_{r(\la)}}=1,
$
and $\sR[\zeta_{\la}]$ may be regarded as the Laurent polynomial ring in the indeterminates $z_1^{\pm 1},\ldots,z_{r(\la)}^{\pm 1}$ subject to the above equation.

The group $G_{\la}$ acts on $\sR[\zeta_{\la}]$ via the equation $g\cdot(\zeta_{\la}^{\gamma})=\zeta_{\la}^{g\gamma}$ for $g\in G_{\la}$ and $\gamma\in P$ (to check that the definition is well defined, note that if $\zeta_{\la}^{\gamma}=\zeta_{\la}^{\gamma'}$ then $\gamma-\gamma'\in Q_{\la}$, and hence $g\gamma-g\gamma'\in Q_{\la}$ by~(\ref{eq:stabiliserho}), and so $\zeta_{\la}^{g\gamma}=\zeta_{\la}^{g\gamma'}$). Explicitly, the action of $G_{\la}$ on $\sR[\zeta_{\la}]$ is given by permuting the variables $z_1,\ldots,z_{r(\la)}$ subject to the constraint that if $gz_i=z_j$ then $\la_i=\la_j$. 

\begin{example} If $\la=(6,6,4,4,4,2,1,1)$ (see Example~\ref{ex:running2}) then $G_{\la}$ permutes the variables $z_1,\ldots,z_8$ preserving the partition $\{z_1,z_2\}\cup\{z_3,z_4,z_5\}\cup\{z_6\}\cup\{z_7,z_8\}$. 
\end{example}

\begin{defn}
Let 
$
\sR[\zeta_{\la}]^{G_{\la}}=\{p(\zeta_{\la})\in\sR[\zeta_{\la}]\mid g\cdot p(\zeta_{\la})=p(\zeta_{\la})\text{ for all $g\in G_{\la}$}\}
$
and similarly $\ZZ[\zeta_{\la}]^{G_{\la}}=\{p(\zeta_{\la})\in\ZZ[\zeta_{\la}]\mid g\cdot p(\zeta_{\la})=p(\zeta_{\la})\text{ for all $g\in G_{\la}$}\}$. 
\end{defn}

Since $P^{(\la)}_+$ is a fundamental domain for the action of $G_{\la}$ on $P^{(\la)}$, it follows that $\sR[\zeta_{\la}]^{G_{\la}}$ (respectively $\ZZ[\zeta_{\la}]^{G_{\la}}$) has basis as a free $\sR$-module (respectively $\ZZ$-module) given by the monomials
\begin{align}\label{eq:monomial}
\fe_{\gamma}(\zeta_{\la})=\sum_{\gamma'\in G_{\la}\cdot\gamma}\zeta_{\la}^{\gamma'},\quad\text{with $\gamma\in P^{(\la)}_+$}.
\end{align}

\begin{defn}\label{defn:schur}
For $\gamma\in P^{(\la)}$ (or $\gamma\in P/Q_{\la}$) let $\fs_{\ga}(\zeta_{\la})$ be the \textit{$G_{\la}$-Schur function}
$$
\mathfrak{s}_{\gamma}(\zeta_{\la})=\sum_{g\in G_{\la}}\zeta_{\la}^{g\gamma}\prod_{\alpha\in \Phi_{G_{\la}}^+}\frac{1}{1-\zeta_{\la}^{-g\alpha}}.
$$
\end{defn}

\begin{prop}
The elements $\fs_{\ga}(\zeta_{\la})$ are in $\ZZ[\zeta_{\la}]^{G_{\la}}$, and $\{\fs_{\ga}(\zeta_{\la})\mid \ga\in P^{(\la)}_+\}$ is a basis of $\sR[\zeta_{\la}]^{G_{\la}}$ (respectively $\ZZ[\zeta_{\la}]^{G_{\la}}$) as a free $\sR$-module (respectively free $\ZZ$-module). 
\end{prop}

\begin{proof}
This is classical, see for example, \cite[(2.14)]{NR:03}. 
\end{proof}

\begin{example}
In the case $\la=(n+1)$ we have $P^{(\la)}=\{0,\omega_1,\ldots,\omega_n\}$, and $G_{\la}=\{1\}$ and $\Phi_{G_{\la}}=\emptyset$. Hence $\fs_{\omega_i}(\zeta_{\la})=\zeta_{\la}^{\omega_i}=z_1^i$ where $z_1=\zeta_{\la}^{e_1}$. 
\end{example}

\begin{example}\label{ex:A33}
If $\la=(2,2)$ as in Examples~\ref{ex:A3} and~\ref{ex:A32} we have that $G_{\la}=\langle\tilde{s}\rangle$ is of type $\sA_1$ (explicitly, $\tilde{s}=s_2s_1s_3s_2$), and $\Phi_{G_{\la}}^+=\{\tilde{e}_1-\tilde{e}_2\}$. We have $(P/Q_{\la})_+=\{a_1\tilde{e}_1+a_2\tilde{e}_2\mid a_1\geq a_2\}$, and $\tilde{s}\tilde{e}_1=\tilde{e}_2$. Thus, for $\gamma=a\tilde{e}_1+b\tilde{e}_2\in P/Q_{\la}$ we have
\begin{align*}
\fs_{\gamma}(\zeta_{\la})&=\frac{z_1^az_2^b}{1-z_1^{-1}z_2}+\frac{z_1^bz_2^a}{1-z_1z_2^{-1}}=\frac{z_1^{a+1}z_2^b-z_1^bz_2^{a+1}}{z_1-z_2}.
\end{align*}
Recall that $z_1^2z_2^2=1$. In particular $\fs_{\tilde{e}_1+\tilde{e}_2}(\zeta_{\la})=z_1z_2$, and $\fs_{\tilde{e}_1+\tilde{e}_2}(\zeta_{\la})^2=z_1^2z_2^2=1$. 
\end{example}

We define a \textit{conjugation} operation on $\sR[\zeta_\la]$ (and $\ZZ[\zeta_{\la}]$) by linearly extending 
$$\mathrm{conj}(\zeta_{\la}^{\gamma})= \zeta_{\la}^{-\gamma},$$ 
and for $f(\zeta_{\la})\in\sR[\zeta_{\la}]$ (or $\ZZ[\zeta_{\la}]$) we write 
$$
[f(\zeta_{\la})]_{\mathrm{ct}}=a_0\quad\text{where}\quad f(\zeta_{\la})=\sum_{\gamma+Q_{\la}\in P/Q_{\la}}a_{\gamma+Q_{\la}}\zeta_{\la}^{\gamma}\quad\text{with}\quad a_{\gamma+Q_{\la}}\in\sR
$$
(the \textit{constant term} of $f(\zeta_{\la})$). For $f(\zeta_\la),g(\zeta_\la)\in \sR[\zeta_\la]$ we set 
$$
\langle f(\zeta_\la),g(\zeta_\la)\rangle_{\la}^{\infty}=\frac{1}{|G_{\la}|}\bigg[f(\zeta_\la)\cdot \mathrm{conj}(g(\zeta_\la))\cdot \prod_{\alpha\in\Phi_{G_{\la}}}(1-\zeta_{\la}^{\alpha})\bigg]_{\mathrm{ct}}.
$$

Recall the definition of the dominance order on $P/Q_{\la}$ from Section~\ref{sec:lambdadominance} (considered also as a partial order on $P^{(\la)}$ via Proposition~\ref{prop:bijectionP}). 

\goodbreak 
\begin{lemma}\label{lem:uniquebasis}
The Schur functions $\fs_{\gamma}(\zeta_{\la})$ with $\gamma\in (P/Q_{\la})_+$, are the unique elements of $\ZZ[\zeta_{\la}]^{G_{\la}}$ satisfying:
\begin{compactenum}[$(1)$]
\item $\fs_{\gamma}(\zeta_{\la})=\fe_{\gamma}(\zeta_{\la})+\sum_{\gamma'\prec_{\la}\gamma}a_{\gamma,\gamma'}\fe_{\gamma'}(\zeta_{\la})$ with $a_{\gamma,\gamma'}\in\ZZ$, and
\item $\langle \fs_{\gamma}(\zeta_{\la}),\fs_{\gamma'}(\zeta_{\la})\rangle_{\la}^{\infty}=\delta_{\gamma,\gamma'}$. 
\end{compactenum}
Moreover, we have $a_{\gamma,\gamma'}\geq 0$. 
\end{lemma}

\begin{proof}
The fact that the Schur functions satisfy (1) and (2) is classical, see \cite[Proposition~3.4]{NR:03}. To prove that these elements are unique, note that by (1) and (2) we have that $\{\fs_{\gamma'}(\zeta_{\la})\mid \gamma'\prec_{\la}\gamma\}$ is an orthonormal basis for the subspace of $G_{\la}$-invariant functions spanned by $\{\fe_{\gamma'}(\zeta_{\la})\mid \gamma'\prec_{\la}\gamma\}$. Thus 
$
\fs_{\gamma}(\zeta_{\la})=\fe_{\gamma}(\zeta_{\la})+\sum_{\gamma'\prec_{\la}\gamma}b_{\gamma,\gamma'}\fs_{\gamma'}(\zeta_{\la})
$
for some integers $b_{\gamma,\gamma'}$, and these integers are uniquely determined by $b_{\gamma,\gamma'}=-\langle \fe_{\gamma}(\zeta_{\la}),\fs_{\gamma'}(\zeta_{\la})\rangle_{\la}^{\infty}$, using (2). Hence the result.
\end{proof}

\subsection{Representations of $\Hext$ induced from Levi subalgebras}\label{sec:reps}

The \textit{$\la$-Levi subalgebra} of $\Hext$ is the subalgebra $\cL_{\la}$ generated by $T_j$, $j\in J_{\la}$, and $X^{\gamma}$, $\gamma\in P$. It is convenient to define $\sv=-\sq^{-1}$. The assignment
\begin{align*}
\psi_{\la}(T_j)=\sv\quad\text{and}\quad \psi_{\la}(X^{\gamma})=\sv^{\langle \gamma,2\rho_{\la}\rangle}\zeta_{\la}^{\gamma}
\end{align*}
for $j\in J_{\la}$ and $\gamma\in P$ extends to a multiplicative character $\psi_{\la}:\cL_{\la}\to \sR[\zeta_{\la}]$ (see \cite[\S5.2]{GLP:23}). 

If $1\leq i\leq n+1$ with $i=\la[r,c]$ (that is, $i=\la(r-1)+c$ with $1\leq c\leq \la_r$), then
$$\psi_\la(X^{e_i}) = (-\sq)^{2c-\la_r-1}z_r.$$
Indeed we have $\psi_{\la}(X^{e_i})=\sv^{\langle e_i,2\rho_{\la}\rangle}\zeta_{\la}^{e_i}=(-\sq)^{-\langle e_i,2\rho_{\la}\rangle}z_r$, and
$$
\langle e_i,2\rho_\la\rangle = \bigg\langle e_i,\sum_{\la(r-1)+1\leq
k<\ell\leq \la(r)} e_{k} - e_\ell\bigg\rangle
= \sum_{\ell = i+1}^{\la(r)} 1 - \sum_{k=\la(r-1)+1}^{i-1} 1 =
\la_r-2c+1.$$
Thus, in particular, we have
\begin{align*}
\psi_{\la}(X^{\alpha_i})=\begin{cases}
\sq^{-2}&\text{if $i\in J_{\la}$}\\
(-\sq)^{\la_k+\la_{k+1}-2}z_kz_{k+1}^{-1}&\text{if $i=\la(k)$ with $1\leq k<r(\la)$.}
\end{cases}
\end{align*}

Let $\sR[\zeta_{\la}]\xi_{\la}$ be a $1$-dimensional $\sR[\zeta_{\la}]$-module generated by $\xi_{\la}$. Then $\xi_{\la}\cdot h=\psi_{\la}(h)\xi_{\la}$ for $h\in\cL_{\la}$ defines a $1$-dimensional representation of $\cL_{\la}$ over $\sR[\zeta_{\la}]$. 

\begin{defn}
Let $(\pi_{\la},M_{\la})$ be the induced representation $\mathrm{Ind}_{\cL_{\la}}^{\Hext}(\psi_{\la})$ with character $\chi_\la$. Thus
$$
M_{\la}=(\sR[\zeta_{\la}]\xi_{\la})\otimes_{\cL_{\la}} \Hext.
$$
\end{defn}

By \cite[Proposition~5.19]{GLP:23} $M_{\la}$ is a free $\sR[\zeta_{\la}]$-module with basis 
$$
\sB_{\la}=\{\xi_{\la}\otimes X_u\mid u\in {^\la}W\},
$$
and it follows that $\dim(M_{\la})=N_{\la}$, where
$$
N_{\la}=\frac{(n+1)!}{\la_1!\la_2!\cdots\la_{r(\la)}!}. 
$$

For $h\in\Hext$ and $u,v\in {^\la}W$ we will write $\pi_{\la}(h;\sB_{\la})$ for the matrix of $\pi_{\la}(h)$ in the above basis (with any chosen order on $^{\la}W$), and 
$
[\pi_{\la}(h;\sB_{\la})]_{u,v}
$
for the matrix entries.

The bar involution extends from $\sR$ to $\sR[\zeta_{\la}]$ with $\overline{\sq}=\sq^{-1}$ and $\overline{\zeta_{\la}^{\gamma}}=\zeta_{\la}^{\gamma}$.

\begin{lemma}\label{lem:barinv}
We have $\chi_{\la}(\overline{h})=\overline{\chi_{\la}(h)}$ for all $h\in\Hext$. 
\end{lemma}

\begin{proof}
By linearity, it is sufficient to prove the result with $h=T_w$. Let $u\in {^\la}W$. We have
$$
(\xi_{\la}\otimes T_u)\cdot \overline{T_w}=\xi_{\la}\otimes \overline{X_uT_w}=\sum_{v\in {^\la}W}(\xi_{\la}\otimes T_v)\overline{[\pi_{\la}(T_w;\sB_{\la})]_{u,v}}.
$$
Thus, writing $\pi_{\la}(h;\sB_{\la}')$ for the matrix with respect to the basis $\sB_{\la}'=\{\xi_{\la}\otimes T_u\mid u\in{^\la}W\}$, we have $[\pi_{\la}(\overline{T_w};\sB_{\la}')]_{u,v}=\overline{[\pi_{\la}(T_w;\sB_{\la})]_{u,v}}$. Since trace is basis independent, the result follows. 
\end{proof}

By \cite[Theorem~5.12]{GLP:23} we have the following combinatorial formula.
\begin{thm}[{\cite[Theorem 5.12]{GLP:23}}]\label{thm:pathformula}
We have, for $u,v\in{^\la}W$,
$$
[\pi_{\la}(T_w;\sB_\la)]_{u,v}=\sum_{p\in\cP_{\la}(\vec{w},u)_v}\cQ_{\la}(p)\zeta_{\la}^{\wt(p)}\quad\text{where}\quad \cQ_{\la}(p)=(-\sq)^{-b(p)}(\sq-\sq^{-1})^{f(p)}.
$$
\end{thm}

\subsection{Intertwiners}\label{se:intertwiners}

For each $1\leq i\leq n$ define an \textit{intertwiner} $U_i$ by 
$$
U_i=T_i-\frac{\sq-\sq^{-1}}{1-X^{-\alpha_i}}.
$$

\begin{remark}\label{rem:extend}
Note that $U_i$ is not an element of the Hecke algebra, however it can be considered as an operator acting on each module $M_{\la}$ (the key observation is that $1-X^{-\alpha_i}$ does not act by~$0$ on any of our modules $M_{\la}$). More generally, we define $\pi_{\la}(U)$, $\chi_{\la}(U)$, and $\pi_{\la}(U;\sB_{\la})$ in the obvious way for any operator $U$ acting on $M_{\la}$ on the right. 
\end{remark}

The following proposition is well known.

\begin{prop}\label{prop:basic} We have the following.
\begin{compactenum}[$(1)$]
\item The elements $U_i$ satisfy the braid relations, and hence the element $U_w=U_{i_1}\cdots U_{i_k}$ is independent of the particular reduced expression $w=s_{i_1}\cdots s_{i_k}$ for $w\in\Wfin$ chosen.
\item $U_wX^{\gamma}=X^{w\gamma}U_w$ for all $\gamma\in P$ and $w\in \Wfin$. 
\item We have
$$
U_i^2=\sq^2\frac{(1-\sq^{-2}X^{-\alpha_i})(1-\sq^{-2}X^{\alpha_i})}{(1-X^{-\alpha_i})(1-X^{\alpha_i})}.
$$
\item If $u,v\in \Wfin$ then $U_uU_v=b(X)U_{uv}$ for a rational function $b(X)$. 
\end{compactenum}
\end{prop}

\begin{proof}
(1), (2), (3) follow from the Bernstein relation, and (4) follows by induction on $\ell(v)$.
\end{proof}

Triangularity between $T_w$ and $U_w$ implies that the module $M_{\la}$ has  ``basis'' $\{\xi_{\la}\otimes U_w\mid w\in {^\la}W\}$, where one must extend scalars to rational functions in $\zeta_{\la}$.

For $j\in J_{\la}$ we have $\xi_{\la}\cdot T_j=-\sq^{-1}\xi_{\la}$ and $\xi_{\la}\cdot X^{\alpha_j}=\sq^{-2}\xi_{\la}$, and hence
$
\xi_{\la}\cdot U_j=0.
$
It follows, using Proposition~\ref{prop:basic}, that with respect to the basis $\{\xi_{\la}\otimes U_u\mid u\in {^\la}W\}$,
\begin{compactenum}[$(1)$]
\item the matrix for $\pi_{\la}(X^{\gamma})$ is diagonal, for $\gamma\in P$.
\item the matrix for $\pi_{\la}(U_w)$, for $w\in \Wfin$, has at most one non-zero entry in each row and column. Indeed, if $u\in {^\la}W$ and $w\in \Wfin$ then the $u$th row of $\pi_{\la}(U_w)$ is zero if $uw\notin {^\la}W$, and has an entry in only the $uw$-position if $uw\in {^\la}W$. 
\end{compactenum}
In particular it follows that $\chi_{\la}(U_w)=0$ if $w\neq e$.

\goodbreak
The following well known formula for $C_{\sw_J}$, in terms of the intertwiners, will be useful.

\begin{thm}\label{thm:Cform}
For $J\subseteq \{1,\ldots,n\}$ we have
$$
C_{\sw_J}=\sq^{\ell(w_J)}\sum_{w\in W_J}\sq^{-\ell(w)}c_w^J(X)U_w\quad\text{where}\quad c_w^J(X)=\prod_{\beta\in\Phi_J^+\backslash\Phi(w)}\frac{1-\sq^{-2}X^{-\beta}}{1-X^{-\beta}}.
$$
\end{thm}

\begin{proof}
The triangularity between $U_w$ and $T_w$ implies that 
$
C_{\sw_J}=\sum_{w\in W_J}a_w(X)U_w
$
for some rational functions $a_w(X)$ with $a_{\sw_J}(X)=1$, and that this expression is unique. Let $j\in J$. Since $C_{\sw_J}T_j=\sq C_{\sw_J}$ we compute
\begin{align*}
C_{\sw_J}U_j&=C_{\sw_J}\left(\sq-\frac{\sq-\sq^{-1}}{1-X^{-\alpha_j}}\right)=\sum_{w\in W_J}\sq a_w(X)\frac{1-\sq^{-2}X^{w\alpha_j}}{1-X^{w\alpha_j}}U_w.
\end{align*}
On the other hand, using the formula for $U_j^2$ from Proposition~\ref{prop:basic} we compute
\begin{align*}
C_{\sw_J}U_j
&=\sum_{w\in W_J,\, ws_j>w}\sq^2a_{ws_j}(X)\frac{(1-\sq^{-2}X^{-w\alpha_j})(1-\sq^{-2}X^{w\alpha_j})}{(1-X^{-w\alpha_j})(1-X^{w\alpha_j})}U_w+\sum_{w\in W_J,\,ws_j<w}a_{ws_j}(X)U_w.
\end{align*}
Comparing coefficients of $U_w$ it follows that if $\ell(ws_j)=\ell(w)-1$ then 
$$
a_{ws_j}(X)=\sq a_w(X)\frac{1-\sq^{-2}X^{w\alpha_j}}{1-X^{w\alpha_j}}.
$$
Now let $w\in W_J$ be arbitrary, and write $\sw_J=ws_{j_1}\cdots s_{j_k}$ with $\ell(\sw_J)=\ell(w)+k$. Thus $w=\sw_Js_{j_k}\cdots s_{j_1}$, and repeated use of the above recursion gives
\begin{align*}
a_w(X)&=\sq^{\ell(\sw_J)-\ell(w)}a_{\sw_J}(X)\prod_{\alpha}\frac{1-\sq^{-2}X^{-w\alpha}}{1-X^{-w\alpha}}
\end{align*}
where the product is over $\alpha\in\{\alpha_{j_1},s_{j_1}\alpha_{j_2},\ldots,s_{j_1}\cdots s_{j_{k-1}}\alpha_{j_k}\}$, which is the inversion set of $s_{j_1}\cdots s_{j_k}=w^{-1}\sw_J$. The result follows, since
$
w\Phi(w^{-1}\sw_J)=\Phi_J^+\backslash\Phi(w)
$
and $a_{\sw_J}(X)=1$. 
\end{proof}

\section{$\la$-relative Satake theory}\label{sec:satake}

In this section we develop a $\la$-relative version of the Satake isomorphism, providing an analogue of the classical Satake isomorphism for each two-sided cell. We first recall the classical Satake isomorphism (which will correspond to the lowest two-sided cell). Let 
$$
\mathbf{1}_0=\frac{\sq^{\ell(\sw_0)}}{\Wfin(\sq^2)}C_{\sw_0},
$$
where we have extended scalars to allow the inverse of $\Wfin(\sq^2)=\sum_{w\in \Wfin}\sq^{2\ell(w)}$ in the base ring~$\sR$ (let $\sR'$ denote this extended ring). This normalisation of the Kazhdan-Lusztig basis element $C_{\sw_0}$ is chosen so that $\mathbf{1}_0^2=\mathbf{1}_0$, and hence $\mathbf{1}_0\Hext\mathbf{1}_0$ is a unital algebra, with identity $\mathbf{1}_0$. The classical Satake isomorphism is then
$$
\mathbf{1}_0\Hext\mathbf{1}_0\cong \sR'[X]^{\Wfin}.
$$
The basic theme for the $\la$-relative Satake isomorphism is to consider the matrix algebras $\pi_{\la}(\Hext)$, and in particular the subalgebras $\pi_{\la}(\clap \Hext\clap)$.

\subsection{The matrix $\pi_{\la}(\clap;\sB_{\la})$}\label{sec:31}

The following proposition shows that the matrix entries of $\pi_{\la}(\clap;\sB_{\la})$ are supported on the interval $[\su_{\la},\su_{\la}\sw_{\la'}]\subseteq {^\la}W$ (note that $[\su_{\la},\su_{\la}\sw_{\la'}]\subseteq {^\la}W$ by Corollary~\ref{cor:lengthaddula}). 

\begin{prop}\label{prop:formofmatrix}
We have $[\pi_{\la}(\clap;\sB_{\la})]_{u,v}=0$ unless $u,v\in [\su_{\la},\su_{\la}\sw_{\la'}]$. If $x,y\in W_{\la'}$ then 
$$
[\pi_{\la}(\clap;\sB_{\la})]_{\su_{\la}x,\su_{\la}y}=\sq^{\ba_{\la}-\ell(x)-\ell(y)}.
$$
\end{prop}

\begin{proof}
For $x\in W_{\la'}$ we have
\begin{align*}
(\xi_{\la}\otimes X_{\su_{\la}x})\cdot \clap&=(\xi_{\la}\otimes X_{\su_{\la}})\cdot T_{x^{-1}}^{-1}\clap=\sq^{-\ell(x)}(\xi_{\la}\otimes X_{\su_{\la}})\cdot \clap,
\end{align*}
and using the second formula in~(\ref{eq:CJ}) it follows that
\begin{align*}
(\xi_{\la}\otimes X_{\su_{\la}x})\cdot \clap&=\sum_{y\in W_{\la'}}\sq^{\ell(\sw_{\la'})-\ell(x)-\ell(y)}(\xi_{\la}\otimes X_{\su_{\la}y}),
\end{align*}
which proves that the $\su_{\la}x$-row (with $x\in W_{\la'}$) of $\pi_{\la}(\clap;\sB_{\la})$ is as claimed. 

It remains to show that all other rows of $\pi_{\la}(\clap;\sB_{\la})$ are zero. That is, if $w\in {^\la}W$ with $w\notin [\su_{\la},\su_{\la}\sw_{\la'}]$ then $(\xi_{\la}\otimes X_w)\cdot \clap=0$. Write $w=w_1w_2$ with $w_1$ being $J_{\la'}$-reduced on the right, and $w_2\in W_{\la'}$. Thus $J_{\la'}$ is a subset of the right ascent set of $w_1$, however since $w\notin [\su_{\la},\su_{\la}\sw_{\la'}]$ we have $w_1\neq \su_{\la}$ and hence $A_{\la}(w_1)\neq J_{\la'}$ (by Theorem~\ref{thm:etale}). It follows that there is $s'\in J_{\la'}$ such that $w_1s'\notin {^\la}W$, and hence $w_1s'=sw_1$ for some $s\in J_{\la}$ (see \cite[p.79]{AB:08}). Now, it is elementary that 
$$
\clap=\left(\sq^{-1}T_{s'}^{-1}+1\right)\sum_{y\in W_{\la'},\,ys'>y}\sq^{\ell(\sw_{\la'})-\ell(y)}X_y.
$$
Since $(\xi_{\la}\otimes X_w)\cdot \clap=\sq^{-\ell(w_2)}(\xi_{\la}\otimes X_{w_1})\cdot \clap$ and $X_{w_1}(\sq^{-1}T_{s'}^{-1}+1)=(\sq^{-1}T_s^{-1}+1)X_{w_1}$ it follows that $(\xi_{\la}\otimes X_w)\cdot \clap=0$ as required (recall that $\xi_{\la}\cdot T_s=-\sq^{-1}\xi_{\la}$ for all $s\in J_{\la}$). 
\end{proof}

\subsection{The subalgebra $\pi_{\la}(C_{\sw_{\la'}}\Hext C_{\sw_{\la'}})$}\label{sec:32}

In this section we show that the algebra~$\pi_{\la}(C_{\sw_{\la'}}\Hext C_{\sw_{\la'}})$ is commutative.

\begin{defn}
Let $f_{\la}:\Hext\to\sR[\zeta_{\la}]$ be the function 
$
f_{\la}(h)=\chi_{\la}(h\clap).
$
We extend the definition of $f_{\la}(\cdot)$ to linear operators $U$ acting on $M_{\la}$, as in Remark~\ref{rem:extend}. 
\end{defn}

\begin{thm}\label{thm:satake1}
We have $\pi_{\la}(\clap h\clap)=f_{\la}(h)\pi_{\la}(\clap)$ for all $h\in\Hext$. 
\end{thm}

\begin{proof}
We have
\begin{align*}
[\pi_{\la}(\clap h\clap;\sB_{\la})]_{u,v}&=\sum_{u_1,u_2\in{^\la}W}[\pi_{\la}(\clap;\sB_{\la})]_{u,u_1}[\pi_{\la}(h;\sB_{\la})]_{u_1,u_2}[\pi_{\la}(\clap;\sB_{\la})]_{u_2,v}.
\end{align*}
By Proposition~\ref{prop:formofmatrix} this is zero unless $u=\su_{\la}x$ and $v=\su_{\la}y$ for some $x,y\in W_{\la'}$, and moreover in the sum $u_1=\su_{\la}x'$ and $u_2=\su_{\la}y'$ with $x',y'\in W_{\la'}$, and hence
\begin{align*}
[\pi_{\la}(\clap h\clap;\sB_{\la})]_{\su_{\la}x,\su_{\la}y}&=\sq^{\ba_{\la}-\ell(x)-\ell(y)}\sum_{x',y'\in W_{\la'}}\sq^{\ba_{\la}-\ell(x')-\ell(y')}[\pi_{\la}(h;\sB_{\la})]_{\su_{\la}x',\su_{\la}y'},
\end{align*}
and so $\pi_{\la}(\clap h\clap)=f'_{\la}(h)\pi_{\la}(\clap)$ where
$$
f'_{\la}(h)=\sum_{x,y\in W_{\la'}}\sq^{\ba_{\la}-\ell(x)-\ell(y)}[\pi_{\la}(h;\sB_{\la})]_{\su_{\la}x,\su_{\la}y}.
$$
Using the formula $\pi_{\la}(\clap h\clap)=f'_{\la}(h)\pi_{\la}(\clap)$ and Proposition~\ref{prop:formofmatrix} we have
$$
\chi_{\la}(\clap h\clap)=f'_{\la}(h)\chi_{\la}(\clap)=\sq^{-\ell(\sw_{\la'})}W_{\la'}(\sq^{2})f'_{\la}(h).
$$
On the other hand, we have 
$$
\chi_{\la}(\clap h\clap)=\chi_{\la}(h\clap^2)=\sq^{-\ell(\sw_{\la'})}W_{\la'}(\sq^2)\chi_{\la}(h\clap),
$$
and thus $f'_{\la}(h)=\chi_{\la}(h\clap)$ as required. 
\end{proof}

\begin{cor}\label{cor:commutative}
The subalgebra $\pi_{\la}(\clap \Hext\clap)$ of $\pi_{\la}(\Hext)$ is commutative. 
\end{cor}

\begin{proof}
For $h_1,h_2\in\Hext$ we have 
$$
\pi_{\la}(\clap h_1\clap)\pi_{\la}(\clap h_2\clap)=f_{\la}(h_1)f_{\la}(h_2)\pi_{\la}(\clap^2)=\pi_{\la}(\clap h_2\clap)\pi_{\la}(\clap h_1\clap)
$$
hence the result.
\end{proof}

The following property will be useful later. 

\begin{cor}\label{cor:fbar}
We have $f_{\la}(\overline{h})=\overline{f_{\la}(h)}$ for all $h\in\Hext$. 
\end{cor}

\begin{proof}
This follows from Lemma~\ref{lem:barinv} and the definition of $f_{\la}$. 
\end{proof}

\subsection{$G_{\la}$ symmetry of $f_{\la}(h)$}\label{sec:33}

Recall the definition of $\psi_{\la}$ from Section~\ref{sec:reps}. In particular, $\psi_{\la}(X^{\gamma})=\sv^{\langle \gamma,2\rho_{\la}\rangle}\zeta_{\la}^{\gamma}$. We extend $\psi_{\la}$ to rational functions in $X$ whose denominators do not vanish on applying $\psi_{\la}$. 

Let
$$
c_{\la'}(X)=\prod_{\alpha\in\Phi_{\la'}^+}\frac{1-\sq^{-2}X^{-\alpha}}{1-X^{-\alpha}}. 
$$

\begin{lemma}\label{lem:czero}
If $u\in {^\la}W$ then $\psi_{\la}(u\cdot c_{\la'}(X))=0$ unless $u\in \su_{\la}W_{\la'}$. 
\end{lemma}

\begin{proof}
Since 
$$
\psi_{\la}(u\cdot c_{\la'}(X))=\prod_{\alpha\in\Phi_{\la'}^+}\frac{1-\sq^{-2}\sv^{-\langle u\alpha,2\rho_{\la}\rangle}\zeta_{\la}^{-u\alpha}}{1-\sv^{-\langle u\alpha,2\rho_{\la}\rangle}\zeta_{\la}^{-u\alpha}}
$$
one just needs to show that if $u\in{^\la}W$ with $u\notin \su_{\la}W_{\la'}$ then $u\alpha=\alpha_s$ for some $\alpha\in\Phi_{\la'}^+$ and $s\in J_{\la}$ (because then $\zeta_{\la}^{-u\alpha}=1$ and $\sv^{-\langle u\alpha,2\rho_{\la}\rangle}=(-\sq^{-1})^{\langle -\alpha_s,2\rho_{\la}\rangle}=\sq^2$, killing the term in the product). The argument is similar to Proposition~\ref{prop:formofmatrix}. Let $u\in {^\la}W$ with $u\notin \su_{\la}W_{\la'}$. Write $u=u_1u_2$ with $u_1$ being $J_{\la'}$-reduced on the right, and $u_2\in W_{\la'}$. Since $u\notin \su_{\la}W_{\la'}$ we have $u_1\neq \su_{\la}$. Since $J_{\la'}$ is contained in the right ascent set of $u_1$, and $A_{\la}(u_1)\neq J_{\la'}$  (by Theorem~\ref{thm:etale}) there is $s'\in J_{\la'}$ with $\ell(u_1s')=\ell(u_1)+1$ and $u_1s'=su_1$ (see \cite[p.79]{AB:08}). But then $u_1\alpha_{s'}=\alpha_s$. Let $\alpha=u_2^{-1}\alpha_{s'}$. Then $\alpha\in\Phi_{\la'}^+$ (because $u_2\in W_{\la'}$, and if $u_2^{-1}\alpha_{s'}<0$ then $u$ is not $J_{\la}$-reduced on the left as $u_2^{-1}u_1^{-1}\alpha_{s}=u_2^{-1}\alpha_{s'}=\alpha$ giving that $\alpha_s$ is in the left descent set of~$u$). Then $u\alpha=\alpha_s$ as required. 
\end{proof}

\begin{lemma}\label{lem:invariance}
We have $\psi_{\la}(\su_{\la}\cdot p(X))\in \sR[\zeta_{\la}]^{G_{\la}}$ for all $p(X)\in\sR[X]^{W_{\la'}}$.
\end{lemma}

\begin{proof}
It is sufficient to prove the result for the monomial symmetric functions  
$
p(X)=\sum_{\ga'\in W_{\la'}\cdot \ga}X^{\gamma'}
$
with $\gamma\in P$. If $g\in G_{\la}$ then 
\begin{align*}
g\cdot \psi_{\la}(\su_{\la}\cdot p(X))&=\sum_{\ga'\in W_{\la'}\cdot \ga}\sv^{\langle \su_{\la}\gamma',2\rho_{\la}\rangle}\zeta_{\la}^{g\su_{\la}\gamma'}. 
\end{align*}
Writing $g\su_{\la}\gamma'=\su_{\la}(\su_{\la}^{-1}g\su_{\la})\gamma'$ and noting that $(\su_{\la}^{-1}g\su_{\la})\ga'\in W_{\la'}\cdot\ga$ (see Proposition~\ref{prop:Gla1}), we change variable in the sum to $\ga''=\su_{\la}^{-1}g\su_{\la}\ga'$, giving
\begin{align*}
g\cdot \psi_{\la}(\su_{\la}\cdot p(X))&=\sum_{\ga''\in W_{\la'}\cdot\ga}\sv^{\langle g^{-1}\su_{\la}\gamma'',2\rho_{\la}\rangle}\zeta_{\la}^{\su_{\la}\gamma''}. 
\end{align*}
By~(\ref{eq:stabiliserho}) we have $\langle g^{-1}\su_{\la}\gamma'',2\rho_{\la}\rangle=\langle \su_{\la}\gamma'',2\rho_{\la}\rangle$, and hence $g\cdot\psi_{\la}(\su_{\la}\cdot p(X))=\psi_{\la}(\su_{\la}\cdot p(X))$ as required. 
\end{proof}

The following theorem connects the subalgebra $\pi_{\la}(\clap \Hext\clap)$ of $\pi_{\la}(\Hext)$ to the ring of $G_{\la}$-symmetric functions. 

\begin{thm}\label{thm:symmetry}
We have $f_{\la}(h)\in\sR[\zeta_{\la}]^{G_{\la}}$ for all $h\in\Hext$. 
\end{thm}

\begin{proof}
Each element of the Hecke algebra can be written as a linear combination of the elements $U_w$, $w\in \Wfin$, with ``coefficients'' being rational functions (in the variables $X^{\gamma}$) whose denominators do not vanish on applying $\psi_{\la}$. Therefore, by linearity of $f_{\la}$, it is sufficient to prove that 
$
f_{\la}(p(X)U_v)
$ is $G_{\la}$-symmetric, where $v\in \Wfin$ and where $p(X)$ is a rational function of the form described above. By definition we have
$
f_{\la}(p(X)U_v)=\chi_{\la}(p(X)U_v\clap),
$
and we shall compute this character below. 

By Theorem~\ref{thm:Cform} we have
$$
C_{\sw_{\la'}}=\sum_{w\in W_{\la'}}\sq^{\ell(w_{\la'})-\ell(w)}c_{\la',w}(X)U_w\quad\text{where}\quad c_{\la',w}(X)=\prod_{\beta\in\Phi_{\la'}^+\backslash\Phi(w)}\frac{1-\sq^{-2}X^{-\beta}}{1-X^{-\beta}}
$$
(note that $c_{\la',e}(X)=c_{\la'}(X)$). Thus
$$
\chi_{\la}(p(X)U_v\clap)=\sum_{w\in W_{\la'}}\sq^{\ell(\sw_{\la'})-\ell(w)}\chi_{\la}\big(p(X)(v\cdot c_{\la',w}(X))U_vU_w\big).
$$
Recall that $U_vU_w$ is a rational function multiple of $U_{vw}$, and that $\pi_{\la}(r(X))$ is diagonal (in the basis of intertwiners) for all rational functions $r(X)$ with non-vanishing denominator on the module, see Proposition~\ref{prop:basic}. Thus, since $\chi_{\la}(U_y)=0$ unless $y=e$, we have 
$$
\chi_{\la}(p(X)U_v\clap)=\begin{cases}
0&\text{if $v\notin W_{\la'}$}\\
\sq^{\ell(\sw_{\la'})-\ell(v)}\chi_{\la}(p(X)(v\cdot c_{\la',v^{-1}}(X))U_vU_{v^{-1}})&\text{if $v\in W_{\la'}$.}
\end{cases}
$$
We compute
\begin{align*}
v\cdot c_{\la',v^{-1}}(X)=\prod_{\alpha\in\Phi_{\la'}^+\backslash \Phi(v)}\frac{1-\sq^{-2}X^{-\alpha}}{1-X^{-\alpha}}
\end{align*}
and repeatedly applying the formula for $U_j^2$ from Proposition~\ref{prop:basic} we have
\begin{align*}
U_vU_{v^{-1}}=\sq^{2\ell(v)}\prod_{\beta\in\Phi(v)}\frac{(1-\sq^{-2}X^{-\beta})(1-\sq^{-2}X^{\beta})}{(1-X^{-\beta})(1-X^{\beta})}.
\end{align*}
Thus, if $v\in W_{\la'}$, we have
$$
(v\cdot c_{\la',v^{-1}}(X))U_{v}U_{v^{-1}}=\sq^{2\ell(v)}c_{\la'}(X)\prod_{\beta\in\Phi(v)}\frac{1-\sq^{-2}X^{\beta}}{1-X^{\beta}},
$$
and it follows (by computing the trace using the basis of intertwiners) that
$$
f_{\la}(p(X)U_v)=\sq^{\ell(\sw_{\la'})+\ell(v)}\sum_{u\in {^\la}W}\psi_{\la}\left(u\cdot\left(p(X)c_{\la'}(X)\prod_{\beta\in\Phi(v)}\frac{1-\sq^{-2}X^{\beta}}{1-X^{\beta}}\right)\right)
.$$
By Lemma~\ref{lem:czero} we have $\psi_{\la}(u\cdot c_{\la'}(X))=0$ unless $u\in \su_{\la}W_{\la'}$, and so the sum over ${^\la}W$ becomes a sum over $W_{\la'}$, giving
\begin{align}\label{eq:explicitform}
f_{\la}(p(X)U_v)=\sq^{\ell(\sw_{\la'})+\ell(v)}\psi_{\la}\left(\su_{\la}\cdot \left(\sum_{y\in W_{\la'}}p(yX)c_{\la'}(yX)\prod_{\beta\in\Phi(v)}\frac{1-\sq^{-2}X^{y\beta}}{1-X^{y\beta}}\right)\right).
\end{align}
The sum is $W_{\la'}$-invariant, and the result follows from Lemma~\ref{lem:invariance}.
\end{proof}

\begin{remark}
The proof of Theorem~\ref{thm:symmetry} shows that $f_{\la}(p(X)U_v)=0$ if $v\notin W_{\la'}$, and~(\ref{eq:explicitform}) computes $f_{\la}(h)$ explicitly when $h$ is written in the form $h=\sum_{v\in \Wfin}p_v(X)U_v$. 
\end{remark}

We note the following corollary. 

\begin{cor}\label{cor:MacFormula}
For $\gamma\in P$ we have
$$
f_{\la}(X^{\gamma})=\chi_{\la}(X^{\gamma}\clap)=\sq^{\ell(\sw_{\la'})}\psi_{\la}(\su_{\la}\cdot P_{\gamma}^{\la'}(X))
$$
where $P_{\gamma}^{\la'}(X)$ is the $\la'$-relative Macdonald spherical function
$$
P_{\gamma}^{\la'}(X)=\sum_{y\in W_{\la'}}X^{y\gamma}\prod_{\alpha\in\Phi_{\la'}^+}\frac{1-\sq^{-2}X^{-y\alpha}}{1-X^{-y\alpha}}=\sum_{y\in W_{\la'}}y\cdot(X^{\gamma}c_{\la'}(X)). 
$$
\end{cor}

\begin{proof}
This is a corollary of~(\ref{eq:explicitform}).
\end{proof}

\subsection{$\la$-relative Satake isomorphism}

Let $\sR'$ be the ring $\sR$ where we adjoined the inverse of $W_{\la'}(\sq^2)$. In this section we work with the Hecke algebra $\Hext_{\sR'}$ with scalars extended to~$\sR'$. Thus our representation $\pi_{\la}$ is over the ring~$\sR'[\zeta_{\la}]$. Then we can consider the element 
$$
\mathbf{1}_{\la'}=\frac{\sq^{\ell(\sw_{\la'})}}{W_{\la'}(\sq^2)}C_{\sw_{\la'}}.
$$
With this normalisation we have $\mathbf{1}_{\la'}^2=\mathbf{1}_{\la'}$, and the algebra
$
\pi_{\la}(\mathbf{1}_{\la'}\Hext_{\sR'}\mathbf{1}_{\la'})
$
is unital (with identity $\mathbf{1}_{\la'}$) and commutative (by Corollary~\ref{cor:commutative}). 

We normalise $f_{\la}$ by
$$
\tilde{f}_{\la}(h)= \frac{\sq^{\ell(\sw_{\la'})}}{W_{\la'}(\sq^2)}f_{\la}(h),
$$
and then by Theorem~\ref{thm:satake1} we have
$
\pi_{\la}(\mathbf{1}_{\la'}h\mathbf{1}_{\la'})=\tilde{f}_{\la}(h)\pi_{\la}(\mathbf{1}_{\la'}).
$

We postpone the proof of the following lemma until Section~\ref{sec:62}, without introducing any circularity in the arguments. 

\begin{lemma}\label{lem:span}
The map $f_{\la}:\Hext\to \sR[\zeta_{\la}]^{G_{\la}}$ is surjective.
\end{lemma}

\begin{proof}
See Section~\ref{sec:62} for the proof. 
\end{proof}

\begin{thm}\label{thm:satake}
We have $\pi_{\la}(\mathbf{1}_{\la'}\Hext_{\sR'}\mathbf{1}_{\la'})\cong \sR'[\zeta_{\la}]^{G_{\la}}$, with the isomorphism given by 
$$
\pi_{\la}(\mathbf{1}_{\la'}h\mathbf{1}_{\la'})\leftrightarrow \tilde{f}_{\la}(h).
$$
\end{thm}

\begin{proof}
Let $\Theta:\pi_{\la}(\mathbf{1}_{\la'}\Hext_{\sR'}\mathbf{1}_{\la'})\to \sR'[\zeta_{\la}]^{G_{\la}}$ be given by $\Theta(A)=\tilde{f}_{\la}(h)$ whenever $A=\pi_{\la}(\mathbf{1}_{\la'}h\mathbf{1}_{\la'})$. It is clear that if $\pi_{\la}(\mathbf{1}_{\la'}h_1\mathbf{1}_{\la'})=\pi_{\la}(\mathbf{1}_{\la'}h_2\mathbf{1}_{\la'})$ then $\tilde{f}_{\la}(h_1)=\tilde{f}_{\la}(h_2)$, and by Theorem~\ref{thm:symmetry} we have $\tilde{f}_{\la}(h)\in\sR'[\zeta_{\la}]^{G_{\la}}$, and so $\Theta$ is well defined. Surjectivity is Lemma~\ref{lem:span}. For injectivity, if $A_i=\pi_{\la}(\mathbf{1}_{\la'}h_i\mathbf{1}_{\la'})$, $i=1,2$, and $\Theta(A_1)=\Theta(A_2)$ then 
$$
A_1=\tilde{f}_{\la}(h_1)\pi_{\la}(\mathbf{1}_{\la'})=\Theta(A_1)\pi_{\la}(\mathbf{1}_{\la'})=\Theta(A_2)\pi_{\la}(\mathbf{1}_{\la'})=\tilde{f}_{\la}(h_2)\pi_{\la}(\mathbf{1}_{\la'})=A_2.
$$
Finally, to check that $\Theta$ is a homomorphism, if $A_i=\pi_{\la}(\mathbf{1}_{\la'}h_i\mathbf{1}_{\la'})$ (for $i=1,2$) then
\begin{align*}
\Theta(A_1A_2)&=\Theta(\pi_{\la}(\mathbf{1}_{\la'}h_1\mathbf{1}_{\la'})\pi_{\la}(\mathbf{1}_{\la'}h_2\mathbf{1}_{\la'}))=\tilde{f}_{\la}(h_1)\tilde{f}_{\la}(h_2)\Theta(\pi_{\la}(\mathbf{1}_{\la'})).
\end{align*}
Since $\Theta(\pi_{\la}(\mathbf{1}_{\la'}))=\tilde{f}_{\la}(1)=1$ the result follows. 
\end{proof}

\begin{remark}\label{rem:recoverclassical}
To recover the classical Satake isomorphism, take $\la=(1^{n+1})$ (that is, the lowest two-sided cell). Then $\pi_{\la}$ is the principal series representation, and it is known that this is a faithful representation of $\Hext$ (this can be easily proved using the intertwiners). Thus $\pi_{\la}(\Hext)\cong\Hext$, and the classical Satake isomorphism follows from Theorem~\ref{thm:satake}.
\end{remark}

\section{The killing property and boundedness}\label{sec:killbound}

In this section we prove two important properties:
\begin{compactenum}[$(1)$]
\item The \textit{killing property:} The representation $\pi_{\la}$ kills all Kazhdan-Lusztig elements $C_w$ from lower or incomparable cells than $\Delta_{\la}$ (see Theorem~\ref{thm:fullkilling}). 
\item \textit{Boundedness:} The degree in $\sq$ of the entries of the matrices $\pi_{\la}(T_w;\sB_{\la})$ for $w\in \Wext$ is bounded by $\ell(\sw_{\la'})$, and if the bound is attained then $w\in\Delta_{\la}$ (see Theorem~\ref{thm:fullbounded}). 
\end{compactenum}
The boundedness property allows us to define \textit{$\la$-leading matrices} $\fc_{\la}(w)$ in Section~\ref{sec:leading}. The ring~$\fC_{\la}$ of $\la$-leading matrices will ultimately be seen to be isomorphic to $\cJ_{\la}$ (see Theorem~\ref{thm:recognise2b}).

\subsection{The killing property}

\begin{lemma}\label{lem:lem2}
Let $\la,\mu\vdash n+1$. If $\mu\not\geq\la$ then $u\Phi_{\mu'}\cap\{\alpha_j\mid j\in J_{\la}\}\neq\emptyset$ for all $u\in {^\la}W$. 
\end{lemma}

\begin{proof}
We will show that if $u\in {^\la}W$ with $u\Phi_{\mu'}\cap \{\alpha_j\mid j\in J_{\la}\}=\emptyset$ then $\mu\geq \la$. Write $u=u_1u_2$ with $u_1$ being $J_{\mu'}$-reduced on the right and $u_2\in W_{J_{\mu'}}$. Then $u_1\in {^\la}W$ (being a prefix of $u\in {^\la}W$), and $u\Phi_{\mu'}=u_1\Phi_{\mu'}$. We have $J_{\mu'}\subseteq A(u_1)$ (with $A(u_1)$ the right ascent set of $u_1$), and it follows that $J_{\mu'}\subseteq A_{\la}(u_1)$ (for if $s'\in J_{\mu'}$ with $u_1s'\notin {^\la}W$ then $u_1s'=su_1$ for some $s\in J_{\la}$, giving $u_1\alpha_{s'}=\alpha_s$ contradicting $u_1\Phi_{\mu'}\cap \{\alpha_j\mid j\in J_{\la}\}=\emptyset$). In particular $\mu'\leq \boldsymbol{\mu}(u_1,\la)$, and by Lemma~\ref{lem:dominance} we have $\boldsymbol{\mu}(u_1,\la)\leq \la'$. Thus $\mu\geq \la$ as required. 
\end{proof}

\begin{lemma}\label{lem:restrictedkilling}
If $\mu\not\geq\la$ then $\pi_{\la}(C_{\sw_{\mu'}})=0$. 
\end{lemma}

\begin{proof}
Applying $\pi_{\la}$ to the equation in Theorem~\ref{thm:Cform} gives
$$
\pi_{\la}(C_{\sw_{\mu'}})=\sq^{\ell(\sw_{\mu'})}\sum_{w\in W_{\mu'}}\sq^{-\ell(w)}\pi_{\la}(c_w^{\mu'}(X)U_w)\quad\text{where}\quad c_w^{\mu'}(X)=\prod_{\beta\in\Phi_{\mu'}^+\backslash\Phi(w)}\frac{1-\sq^{-2}X^{-\beta}}{1-X^{-\beta}}.
$$
In the basis of intertwiners, the matrices $\pi_{\la}(U_w)$ (for $w\in \Wfin$) have distinct support (the places of the nonzero entries; because the $u$th row of $\pi_{\la}(U_w)$ is either $0$, or has an entry only in the~$uw$ position in the case $uw\in {^\la}W$, see Proposition~\ref{prop:basic}). Since $\pi_{\la}(c_w^{\mu'}(X))$ is diagonal (in the intertwiner basis) it follows that $\pi_{\la}(C_{\sw_{\mu'}})=0$ if and only if 
$
\pi_{\la}(c_w^{\mu'}(X)U_w)=0$ for all $w\in W_{\mu'}$. This, in turn, is equivalent to the statement
$$
(\xi_{\la}\otimes U_u)\cdot c_w^{\mu'}(X)U_w=0\quad\text{for all $u\in {^\la}W$ and $w\in W_{\mu'}$}.
$$
We have
\begin{align}
\label{eq:action1}
(\xi_{\la}\otimes U_u)\cdot c_w^{\mu'}(X)U_w=\bigg(\prod_{\alpha\in\Phi_{\mu'}^+\backslash\Phi(w)}\frac{1-\sq^{-2}\psi_{\la}(X^{-u\alpha})}{1-\psi_{\la}(X^{-u\alpha})}\bigg)(\xi_{\la}\otimes U_uU_w). 
\end{align}

Since $\mu\not\geq\la$ Lemma~\ref{lem:lem2} gives $u\Phi_{\mu'}\cap \{\alpha_j\mid j\in J_{\la}\}\neq\emptyset$ for all $u\in {^\la}W$. Since $U_uU_w$ is a rational multiple of $U_{uw}$ we have that if $uw\notin {^\la}W$ then $(\xi_{\la}\otimes U_u)\cdot c_w^{\mu'}(X)U_w=0$. So assume that $uw\in {^\la}W$.  By assumption there is $j\in J_{\la}$ such that $u^{-1}\alpha_j=\beta\in\Phi_{\mu'}$. Since $u\in {^\la}W$ we have $\beta>0$, and so $\beta\in\Phi_{\mu'}^+$. Moreover, $w^{-1}\beta=(uw)^{-1}\alpha_j>0$ because $uw\in {^\la}W$ by assumption. Thus $\beta\in\Phi_{\mu'}^+\backslash\Phi(w)$. Thus $\beta$ appears in the above product, and the corresponding factor is 
$$
\frac{1-\sq^{-2}\psi_{\la}(X^{-u\beta})}{1-\psi_{\la}(X^{-u\beta})}=\frac{1-\sq^{-2}\psi_{\la}(X^{-\alpha_j})}{1-\psi_{\la}(X^{-\alpha_j})}=0
$$
(as $j\in J_{\la}$, and so $\psi_{\la}(X^{\alpha_j})=\sq^{-2}$). Thus $\pi_{\la}(C_{\sw_{\mu'}})=0$. 
\end{proof}

Recall that $\Delta_{\la}$ denotes the two sided cell of $\Wext$ containing~$\sw_{\la'}$.

\begin{thm}\label{thm:fullkilling}
Let $\la,\mu\vdash n+1$. If $w\in \Delta_{\mu}$ with $\mu\not\geq\la$ then $\pi_{\la}(C_w)=0$. 
\end{thm}

\begin{proof}
Let $w\in \Delta_\mu$. By Theorem~\ref{thm:TanisakiXi}, there exist $h,h'\in \widetilde{H}$ such that 
$$hC_{\sw_{\mu'}}h' = C_{w} +\sum_{z\in \widetilde{W},z<_{LR} \sw_{\mu'}} a_zC_z\qu{with $a_z\in \sR$.}$$
Assume that $a_z\neq 0$. Then $z\in \Delta_{\nu}$ where $\nu<\mu$. It follows that $\nu\ngeq\la$ and
by a straightforward induction and Lemma \ref{lem:restrictedkilling}, we get that $\pi_{\la}(C_z) = 0$. By Lemma \ref{lem:restrictedkilling}, we have $\pi_\la(C_{\sw_{\mu'}}) = 0$ which implies that $\pi_{\la}(C_w)=0$ as required. 
\end{proof}

\subsection{Boundedness}
\label{sec:Boundedness}

Given an expression $\vec{w}=s_{i_1}\cdots s_{i_{\ell}}\pi$, we define the \textit{reversed} expression by 
$$
\mathrm{rev}(\vec{w})=s_{\pi^{-1}(i_{\ell})}\cdots s_{\pi^{-1}(i_1)}\pi^{-1}.
$$
If $\vec{w}$ is a reduced expression for $w$, then $\mathrm{rev}(\vec{w})$ is a reduced expression for $w^{-1}$.

We now define an involution $p\mapsto p^{-1}$ on the set of all $\la$-folded alcove paths. Roughly speaking, this involution is given by ``reading the path backwards''.

\begin{defn}
Let $p=(v_0,v_1,\ldots,v_{\ell},v_{\ell}\pi)$ be a $\la$-folded alcove path with $\wt(p)=\gamma$. The \textit{inverse alcove path} is  
$$
p^{-1}=(\tau_{\gamma}^{-1}v_{\ell}\pi,\tau_{\gamma}^{-1}v_{\ell-1}\pi,\ldots,\tau_{\gamma}^{-1}v_1\pi,\tau_{\gamma}^{-1}v_0\pi,\tau_{\gamma}^{-1}v_0).
$$
\end{defn}

\begin{lemma}\label{lem:reversepath}
Let $p$ be a $\la$-folded alcove path of type $\vec{w}$ starting at $u\in {^\la}W$ with $\theta^{\la}(p)=v$ and $\wt(p)=\gamma$. Then $p^{-1}$ is a $\la$-folded alcove path of type $\mathrm{rev}(\vec{w})$ starting at $v\in {^\la}W$ with $\theta^{\la}(p^{-1})=u$ and $\wt(p^{-1})=(-\gamma)^{(\la)}$. Moreover $\cQ_{\la}(p^{-1})=\cQ_{\la}(p)$, and the map $p\mapsto p^{-1}$ is a bijection from $\{p\in\cP_{\la}(\vec{w},u)\mid \theta^{\la}(p)=v\}\to \{p\in \cP_{\la}(\mathrm{rev}(\vec{w}),v)\mid \theta^{\la}(p)=u\}$. 
\end{lemma}

\begin{proof}
Let $\vec{w}=s_{i_1}\cdots s_{i_{\ell}}\pi$ and $p=(v_0,v_1,\ldots,v_{\ell},v_{\ell}\pi)$, and write $p^{-1}=(v_0',v_1',\ldots,v_{\ell-1}',v_{\ell}',v_{\ell}'\pi^{-1})$. Thus $v_k'=\tau_{\gamma}^{-1}v_{\ell-k}\pi$ for $1\leq k\leq \ell$. We have
$$
v_{k-1}'^{-1}v_k'=(\tau_{\gamma}^{-1}v_{\ell-k+1}\pi)^{-1}(\tau_{\gamma}^{-1}v_{\ell-k}\pi)=\pi^{-1}v_{\ell-k+1}^{-1}v_{\ell-k}\pi\in \{1,s_{\pi^{-1}(i_{\ell-k})}\}
$$
(because $v_{k-1}^{-1}v_k\in \{1,s_{i_k}\}$ by definition of alcove paths). Thus $p^{-1}$ is a path of type $\mathrm{rev}(\vec{w})$. We must show that this path is $\la$-folded. 

Since $v_k\cA_0\in \cA_{\la}$ (by definition of $\la$-folded alcove paths) we have $v_k'\cA_0\in \cA_{\la}$ for all $1\leq k\leq \ell$ (as $\tau_{\ga}^{-1}$ preserves~$\cA_\la$ by~(\ref{eq:lambdaaffine})), and so the path $p^{-1}$ stays in $\cA_{\la}$. Since $\tau_{\ga}^{-1}$ preserves the boundary of $\cA_{\la}$ it follows that if $(v_{k-1},v_{k})$ is a bounce (in $p$), then $(v_{\ell-k}',v_{\ell-k+1}')$ is a bounce (in $p^{-1}$). Thus to show that $p^{-1}$ is a $\la$-folded alcove path we must show that if $(v_{k+1},v_k)$ is a (necessarily positive) fold in $p$, then $(v_{\ell-k}',v_{\ell-k+1}')$ is a positive fold in $p^{-1}$. Thus suppose that $v_{k-1}=v_k$ with $v_{k-1}s_{i_k}\cA_0\in \cA_{\la}$ such that $v_{k-1}s_{i_k}\cA_0$ is on the positive side of the hyperplane separating $v_{k-1}\cA_0$ and $v_{k-1}s_{i_k}\cA_0$. Since $v_{k-1}s_{i_k}\cA_0\in \cA_{\la}$ this hyperplane is in a direction $\alpha\in\Phi^+$ with $\alpha\notin\Phi_{\la}$. Since $v_{\ell-k}'=\tau_{\gamma}^{-1}v_k\pi=\tau_{\gamma}^{-1}v_{k-1}\pi$ we have $v_{\ell-k}'\cA_0=\tau_{\gamma}^{-1}v_{k-1}\cA_0$, and since $\tau_{\gamma}^{-1}=t_{(-\gamma)^{(\la)}}\sy_{(-\gamma)^{(\la)}}$ with $\sy_{(-\gamma)^{(\la)}}\in W_{\la}$ it follows that $v_{\ell-k}'\cA_0$ is on the positive side of the hyperplane separating $v_{\ell-k}'\cA_0$ from $v_{\ell-k}'s_{\pi^{-1}(i_{\ell-k+1})}\cA_0$, and that this hyperplane is in a direction $\alpha'\in \Phi^+$ with $\alpha'\notin \Phi_{\la}$. Thus $(v_{\ell-k}',v_{\ell-k+1}')$ is a positively oriented fold, and so $p^{-1}$ is a $\la$-folded alcove path. 

Since folds and bounces are preserved under the map $p\mapsto p^{-1}$ we have $\cQ_{\la}(p^{-1})=\cQ_{\la}(p)$. Moreover, we have $v_{\ell}\pi=\tau_{\gamma}v$ and hence $v_0'=v$, and since $v_{\ell}'\pi^{-1}=\tau_{\gamma}^{-1}u=\tau_{(-\gamma)^{(\la)}}u$ we have $\wt(p^{-1})=(-\gamma)^{(\la)}$ and $\theta^{\la}(p^{-1})=u$. It is clear that $(p^{-1})^{-1}=p$ and so the map $p\mapsto p^{-1}$ is an involution, and hence is bijective.
\end{proof}

Recall that we define a conjugation on $\sR[\zeta_\la]$ by linearly extending $\mathrm{conj}(\zeta_{\la}^{\gamma})= \zeta_{\la}^{-\gamma}$. 
We then define an anti-involution $*$ on $\pi_{\la}(\Hext)$ by transposing the matrix $\pi_{\la}(h;\sB_{\la})$ and performing conjugation entry-wise.  The following lemma explains the relation with the usual $\ast$ anti-involution on $\Hext$ defined by 
\begin{align}
\label{eq:star}
\left(\sum a_wT_w\right)^*=\sum a_wT_{w^{-1}}.
\end{align}

\begin{lemma}\label{lem:conjugate}
For all $h\in \Hext$ we have
$
\pi_{\la}(h^*)=\pi_{\la}(h)^*.
$
\end{lemma}

\begin{proof}
It is equivalent to prove $\pi_{\la}(h)=\pi_{\la}(h^*)^*$, and it is sufficient to prove this for $h=T_w$ with $w\in\Wext$. By Theorem~\ref{thm:pathformula} and Lemma~\ref{lem:reversepath} we have
\begin{align*}
[\pi_{\la}(T_w;\sB_{\la})]_{u,v}&=\sum_{\{p\in\cP_{\la}(\vec{w},u)\mid \theta^{\la}(p)=v\}}\cQ_{\la}(p)\zeta_{\la}^{\wt(p)}=\sum_{\{p\in\cP_{\la}(\mathrm{rev}(\vec{w}),v)\mid \theta^{\la}(p)=u\}}\cQ_{\la}(p)\zeta_{\la}^{-\wt(p)},
\end{align*}
and the latter equals $[\pi_{\la}(T_w^*;\sB_{\la})^*]_{v,u}$. 
\end{proof}

\begin{thm}\label{thm:fullbounded}
The degree in $\sq$ of the entries of the matrices $\pi_{\la}(T_w;\sB_{\la})$ for $w\in \Wext$ is bounded by $\ell(\sw_{\la'})$. Moreover, if $\deg\pi_{\la}(T_w;\sB_{\la})=\ell(\sw_{\la'})$ then $w\in \Delta_{\la}$ (the two sided cell containing $\sw_{\la'}$). 
\end{thm}

\begin{proof}
Let $N=\max\{\deg[\pi_{\la}(T_w;\sB_{\la})]_{u,v}\mid w\in \Wext, \,u,v\in {^\la}W\}$ which is well defined by \cite[Theorem 6.13]{GLP:23}. Let $w$ be such that $\pi_{\la}(T_w;\sB_{\la})$ attains degree~$N$, and suppose that this degree is attained on the $u$th row of $\pi_{\la}(T_w;\sB_{\la})$. Let $v_1,\ldots,v_k\in {^\la}W$ be the columns with $\deg[\pi_{\la}(T_w;\sB_{\la})]_{u,v_j}=N$ for $1\leq j\leq k$. Write 
$$[\pi_{\la}(T_w;\sB_{\la})]_{u,v_j}=a_j(\zeta_{\la})\sq^N+(\textrm{terms of strictly lower degree}), \text{ with $a_j(\zeta_{\la})\in \sR[\zeta_{\la}]$}.$$
By Lemma~\ref{lem:conjugate} we have $[\pi_{\la}(T_{w^{-1}};\sB_{\la})]_{v_j,u}=a_j(\zeta_{\la}^{-1})\sq^N+(\textrm{terms of strictly lower degree})$. Using triangularity, we have also $[\pi_{\la}(C_w;\sB_{\la})]_{u,v_j}=a_j(\zeta)\sq^N+(\textrm{terms of strictly lower degree})$ and $[\pi_{\la}(C_{w^{-1}};\sB_{\la})]_{v_j,u}=a_j(\zeta_{\la}^{-1})\sq^N+(\textrm{terms of strictly lower degree})$. Thus
\begin{align*}
[\pi_{\la}(C_w;\sB_{\la})\pi_{\la}(C_{w^{-1}};\sB_{\la})]_{u,u}&=[a_1(\zeta_{\la})a_1(\zeta_{\la}^{-1})+\cdots+a_k(\zeta_{\la})a_k(\zeta_{\la}^{-1})]\sq^{2N}+\cdots,
\end{align*}
where the omitted terms in the sum are of strictly lower degree. We claim that the coefficient of $\sq^{2N}$ cannot vanish. To see this, note that if $a(\zeta_{\la})\in\sR[\zeta_{\la}]$ with $a(\zeta_{\la})=\sum_{\gamma\in P^{(\la)}}a_{\gamma}\zeta_{\la}^{\gamma}$ with $a_{\gamma}\in\ZZ$ then
$
a(\zeta_{\la})a(\zeta_{\la}^{-1})=\sum_{\gamma_1,\gamma_2}a_{\gamma_1}a_{\gamma_2}\zeta_{\la}^{\gamma_1-\gamma_2}=\sum_{\gamma}a_{\gamma}^2+\text{terms involving $\zeta_{\la}$}.
$
In particular, the constant term is strictly positive. Since $[a_1(\zeta_{\la})a_1(\zeta_{\la}^{-1})+\cdots+a_k(\zeta_{\la})a_k(\zeta_{\la}^{-1})]$ is a sum of terms of this form, it cannot vanish. In summary, $\pi_{\la}(C_w;\sB_{\la})\pi_{\la}(C_{w^{-1}};\sB_{\la})$ attains degree $2N$ in the $(u,u)$-entry.

On the other hand,
\begin{align}\label{eq:applypi}
\pi_{\la}(C_w;\sB_{\la})\pi_{\la}(C_{w^{-1}};\sB_{\la})=\sum_{z}h_{w,w^{-1},z}\pi_{\la}(C_z;\sB_{\la}).
\end{align}
By Theorem~\ref{thm:fullkilling} the sum is over $z$ in two-sided cells $\Delta_{\mu}$ with $\mu\geq \la$, and hence $\ba(z)\leq \ell(\sw_{\la'})$ for all such $z$. Since $N$ was the maximal degree of all $\pi_{\la}(C_z;\sB_{\la})$, we have $\deg\pi_{\la}(C_z;\sB_{\la})\leq N$. Since $\deg h_{w,w^{-1},z}\leq \ba(z)$ it follows that the maximum degree on the right hand side is $\ell(\sw_{\la'})+N$. Thus $2N\leq \ell(\sw_{\la'})+N$, hence the representation is bounded by $\ell(\sw_{\la'})$.

Suppose now that $\pi_{\la}(T_w;\sB_{\la})$ attains the optimal degree~$\ell(\sw_{\la'})$, in position $(u,v)$, say. As explained above, $\pi_{\la}(C_w;\sB_{\la})\pi_{\la}(C_{w^{-1}};\sB_{\la})$ obtains degree $2\ell(\sw_{\la'})$ in position $(u,u)$, and then~(\ref{eq:applypi}) implies that there exists $z\geq_{LR}\sw_{\la'}$ such that $\deg h_{w,w^{-1},z}=\ell(\sw_{\la'})$ and $\deg[\pi_{\la}(C_z;\sB_{\la})]_{u,u}=\ell(\sw_{\la'})$. Thus $\ba(z)=\ell(\sw_{\la'})$ (because $\ba(z)\leq \ell(\sw_{\la'})$) and hence $z\sim_{LR}\sw_{\la'}$ by P11 (see Section~\ref{sec:KLsec}). But then P8 gives $w\sim_R z$, and so $w\in \Delta_{\la}$ as required.
\end{proof}

The following corollary verifies \cite[Conjecture~6.16]{GLP:23} for type $\tilde{\sA}_n$. 

\begin{cor}\label{cor:recognise1}
The set of elements $w\in \Wext$ such that the matrix $\pi_{\la}(T_w;\sB_\la)$ attains the bound $\ell(\sw_{\la'})$ is a subset of the two-sided cell~$\Delta_{\la}$. 
\end{cor}

\begin{proof}
This is immediate from Theorem~\ref{thm:fullbounded}.
\end{proof}

Later we will be be able to improve on Corollary~\ref{cor:recognise1} to show that the set of elements recognised by $(\pi_{\la},M_{\la},\sB_{\la})$ is precisely $\Delta_{\la}$ (that is, $\pi_{\la}$ \textit{recognises} $\Delta_{\la}$, see Theorem~\ref{thm:recognise2a}). This improvement will require the asymptotic Plancherel Theorem.

\begin{remark}
The connection between the bound on matrix entries, and the bound on $\cQ_{\la}(p)$ for $\la$-folded alcove paths, is rather subtle (c.f. Theorem~\ref{thm:pathformula}). For example, consider $\tilde{\sA}_4$ with $\lambda=(2,1,1,1)$ (so $J_{\la}=\{1\}$), and let $\vec{w}=434234123$. We have $w\in\Delta_{\la}$ (one way to see this is to compute $[\pi_{\la}(T_w)]_{e,s_2s_3s_4}=\sq^6-3\sq^4+4\sq^2-4+4\sq^{-2}-3\sq^{-4}+\sq^{-6}$, and hence the bound $\ell(\sw_{\la'})=6$ is attained, and apply Theorem~\ref{thm:fullbounded}). However we note that individual paths can attain a higher degree. For example, the path $p_1=\hat{4}\hat{3}\hat{4}2\hat{3}\hat{4}\hat{1}2\hat{3}$  (where $\hat{i}$ indicates a fold) is a $\la$-folded alcove path of type $\vec{w}$ starting at $e$ and ending at $e$, and we have $\cQ_{\la}(p_1)=(\sq-\sq^{-1})^7$, which has degree~$7$. The leading term of $\cQ_{\la}(p_1)$ is cancelled by another path $p_2=\hat{4}\hat{3}\hat{4}\hat{2}\hat{3}\hat{4}\check{1}\hat{2}\hat{3}$ (where $\check{i}$ indicates a bounce) starting and ending at~$e$. Indeed we have $\cQ_{\la}(p_2)=-\sq^{-1}(\sq-\sq^{-1})^8$ and so $\cQ_{\la}(p_1)+\cQ_{\la}(p_2)=\sq^{-1}(\sq-\sq^{-1})^7$, which has degree $6$. In fact, further cancellations occur -- there are $14$ $\la$-folded alcove paths of type $\vec{w}$ starting and ending at $e$, and summing the associated $\cQ_{\la}(p)$ terms gives (by Theorem~\ref{thm:pathformula}) $[\pi_{\la}(T_w)]_{e,e}=-\sq+2\sq^{-1}-3\sq^{-5}+3\sq^{-7}-\sq^{-9}$. 
\end{remark}

\subsection{Leading matrices and the characters $\chi_{\la}^{\infty}$ of $\cJ$}\label{sec:leading}

The following definition is modelled by the work of Geck~\cite{Geck:02, Geck:11} in the finite dimensional case, and extends the concepts introduced by the second and fourth authors in~\cite{GP:19,GP:19b}. For $p(\zeta_{\la})\in(\ZZ[\sq^{-1}])[\zeta_{\la}]$ let $\sp(p(\zeta_{\la}))\in\ZZ[\zeta_{\la}]$ denote the specialisation of $p(\zeta_{\la})$ at $\sq^{-1}=0$. We extend this definition entrywise to matrices over $(\ZZ[\sq^{-1}])[\zeta_{\la}]$.

\begin{defn}\label{defn:leading}
The \textit{$\la$-leading matrix} of $w\in\Wext$ is 
$$
\fc_{\la}(w)=\sp\big(\sq^{-\ell(\sw_{\la'})}\pi_{\la}(C_w;\sB_{\la})\big)\in\mathrm{Mat}_{N_{\la}}(\ZZ[\zeta_{\la}]).
$$
This specialisation exists by Theorem~\ref{thm:fullbounded}, and if $\fc_{\la}(w)\neq 0$ then $w\in\Delta_{\la}$ by Corollary~\ref{cor:recognise1}. 
\end{defn}

In particular, by Proposition~\ref{prop:formofmatrix} we have
\begin{align}\label{eq:leadinglap}
\fc_{\la}(\sw_{\la'})=E_{\su_{\la},\su_{\la}},
\end{align}
where $E_{u,v}$ denotes the matrix with $1$ in the $(u,v)$-place, and $0$ elsewhere.

Let
$$
\fC_{\la}=\mathrm{span}_{\ZZ}\{\fc_{\la}(w)\mid w\in\Delta_{\la}\}\quad\text{and}\quad 
\fC=\bigoplus_{\la\vdash n+1}\fC_{\la}.
$$

The following lemma shows that we may replace $C_w$ by $T_w$ in the definition of $\fc_{\la}(w)$. We will use this result frequently.

\begin{lemma}\label{lem:replaceCbyT}
We have $\fc_{\la}(w)=\sp\big(\sq^{-\ell(\sw_{\la'})}\pi_{\la}(T_w;\sB_{\la})\big)$.
\end{lemma}

\begin{proof}
This follows from Theorem~\ref{thm:fullbounded} and triangularity between the $C_w$ and $T_w$ bases.
\end{proof}

Recall the definition of $\gamma_{x,y,z^{-1}}\in\ZZ$ from Definition~\ref{def:gamma}. In the following proposition we show that $\fC_{\la}$ is an associative $\ZZ$-algebra. 

\begin{prop}\label{prop:Zalgebra}
The $\ZZ$-module $\fC_{\la}$ is an associative $\ZZ$-algebra under matrix multiplication. Moreover, for $x,y\in\Delta_{\la}$ we have
$$
\fc_{\la}(x)\fc_{\la}(y)=\sum_{z\in\Delta_{\la}}\gamma_{x,y,z^{-1}}\fc_{\la}(z).
$$
Thus the linear map $\cJ_{\la}\to \fC_{\la}$, $\st_w\mapsto\fc_{\la}(w)$, is a surjective homomorphism of unital rings. 
\end{prop}

\begin{proof}
For $x,y\in\Delta_{\la}$ we have
$$
[\sq^{-\ell(\sw_{\la'})}\pi_{\la}(C_x;\sB_{\la})][\sq^{-\ell(\sw_{\la'})}\pi_{\la}(C_y;\sB_{\la})]=\sum_{z\in\Wext}[\sq^{-\ell(\sw_{\la'})}h_{x,y,z}][\sq^{-\ell(\sw_{\la'})}\pi_{\la}(C_z;\sB_{\la})].
$$
By Definition~\ref{defn:leading} the left hand side specialises to $\fc_{\la}(x)\fc_{\la}(y)$. By Theorem~\ref{thm:fullkilling} the sum on the right is over $z\in\Delta_{\mu}$ with $\mu\geq\la$ (for all other terms vanish). Thus the sum is over elements $z$ in two-sided cells higher than, or equal to, $\Delta_{\la}$, and hence $\deg(\sq^{-\ell(\sw_{\la'})}h_{x,y,z})\leq 0$ for these terms. Thus each term $[\sq^{-\ell(\sw_{\la'})}h_{x,y,z}][\sq^{-\ell(\sw_{\la'})}\pi_{\la}(C_z;\sB_{\la})]$ can be specialised at $\sq^{-1}=0$. If this specialisation is nonzero, then necessarily $\deg\pi_{\la}(C_z;\sB_{\la})=\ell(\sw_{\la'})$, and so $z\in\Delta_{\la}$ by Corollary~\ref{cor:recognise1}, and the result follows. 
\end{proof}

It follow from Lemmas~\ref{lem:conjugate} and~\ref{lem:replaceCbyT} that
\begin{align}\label{eq:conjugation}
\fc_{\la}(w^{-1})=\fc_{\la}(w)^*
\end{align} 
where, as in Section~\ref{sec:leading}, $\fc_{\la}(w)^*$ denotes the conjugate transpose matrix. 

For $\la\vdash n+1$ define $\pi_{\la}^{\infty}:\cJ\to \mathrm{Mat}_{N_{\la}}(\ZZ[\zeta_{\la}])$ by 
\begin{align}\label{eq:asymptoticrepresentation}
\pi_{\la}^{\infty}(A)=\sum_{w\in \Delta_{\la}}a_w\fc_{\la}(w)\quad\text{where}\quad A=\sum_{w\in \Wext}a_w\st_w.
\end{align}
In particular, note that $\pi_{\la}^{\infty}(\st_w)=0$ if $w\notin \Delta_{\la}$. Proposition~\ref{prop:Zalgebra} shows that $\pi_{\la}^{\infty}$ gives a matrix representation of $\cJ$. Let $\chi_{\la}^{\infty}$ be the character of $\pi_{\la}^{\infty}$, and so
$$
\chi_{\la}^{\infty}(A)=\tr(\pi_{\la}^{\infty}(A))\quad\text{for $A\in\cJ$}
$$
where $\tr$ denotes matrix trace. 

\section{The asymptotic Plancherel Theorem}\label{sec:asymptoticplancherel}

In this section we prove an \textit{asymptotic Plancherel Theorem} for type $\tilde{\sA}_n$. This notion was introduced in \cite{GP:19,GP:19b}, and Proposition~\ref{prop:degrees} confirms \cite[Conjecture~9.8]{GP:19} and \cite[Conjecture~6.15]{GLP:23} for affine type~$\sA$. The asymptotic Plancherel Theorem will play a crucial role, often leading to efficient proofs of certain statements (for example, see Theorems~\ref{thm:recognise2a}, \ref{thm:recognise2b}, and Claim 3 in the proof of Theorem~\ref{thm:leading0}). We note that other interesting applications of the Plancherel Theorem in Kazhdan-Lusztig theory have recently been made by Dawydiak~\cite{Daw:23}.

\subsection{The canonical trace on $\cJ$}

In this subsection (and only in this subsection) let $(W,S)$ be a Coxeter system of arbitrary type, let $\varphi:W\to \mathbb{Z}_{\geq 0}$ be a weight function with $\varphi(s)>0$, and let $H$ be the associated weighted Hecke algebra over $\sR=\ZZ[\sq,\sq^{-1}]$, as in \cite{Bon:17,Lus:03}. We assume that the weighted Coxeter system $(W,S,\varphi)$ is bounded as in \cite[\S13.2]{Lus:03}, and that properties P1--P15 hold. Thus the asymptotic algebra $\cJ$ of $H$ is an associative $\ZZ$-algebra (see \cite[\S18.3]{Lus:03}).

Recall the definition of the set $\cD$ of distinguished involutions, and the integers $n_d$ $(d\in \cD$) from Section~\ref{sec:KLsec}. For $w\in W$ we write $d(w)$ for the unique element of~$\cD$ in the right cell containing~$w$ (c.f. P13). By P5 we have 
$
n_{d(x)}=\gamma_{x,x^{-1},d(x)}=\pm 1.
$

Recall that the \textit{canonical trace} on $H$ is the $\sR$-linear function $\Tr:H\to \sR$ with $\Tr(T_w)=\delta_{w,e}$ (see \cite[\S8.1]{GP:00}). Induction on $\ell(v)$ shows that $\Tr(T_uT_{v^{-1}})=\delta_{u,v}$ for all $u,v\in W$. This implies that $\Tr(h_1h_2)=\Tr(h_2h_1)$ for $h_1,h_2\in H$, and that 
\begin{align*}
\langle h_1,h_2\rangle=\Tr(h_1h_2^*)
\end{align*}
defines a nondegenerate symmetric bilinear form on $H$, with $\langle T_u,T_v\rangle=\delta_{u,v}$ (recall here that $\ast$ is defined in (\ref{eq:star})). We now define a canonical trace on the asymptotic algebra~$\cJ$, following the finite case given in \cite[\S20.1(b)]{Lus:03}. 

\begin{defn}\label{defn:canonicalasymptotictrace}
Define a linear map $\Tr^{\infty}:\cJ\to \ZZ$ and a bilinear form $\langle \cdot,\cdot\rangle^{\infty}:\cJ\times\cJ\to\ZZ$ by 
$$
\Tr^{\infty}\left(\sum a_w\st_w\right)=\sum_{d\in \cD}n_da_d\quad\text{and}\quad \langle A,B\rangle^{\infty}=\Tr^{\infty}(AB^*)
$$
for $A,B\in\cJ$, where $\left(\sum a_w\st_w\right)^*=\sum a_w\st_{w^{-1}}$. 
\end{defn}

\begin{thm}\label{thm:innerproductonJ} With the above notation and assumptions, we have:
\begin{compactenum}[$(1)$]
\item The linear map $\Tr^{\infty}$ is a trace functional on $\cJ$. 
\item The bilinear form $\langle\cdot,\cdot\rangle^{\infty}$ is an inner product and $(\st_w)_{w\in W}$ is an orthonormal basis of $\cJ$.
\item For $A,B,C\in\cJ$ we have $\langle AB,C\rangle^{\infty}=\langle B,A^*C\rangle^{\infty}$. 
\end{compactenum}
\end{thm}

\begin{proof}
Since $
\st_x\st_y^*=\sum_{z\in W}\gamma_{x,y^{-1},z^{-1}}\st_z
$
we have
$
\Tr^{\infty}(\st_x\st_y^*)=\sum_{d\in \cD}n_d\gamma_{x,y^{-1},d}.
$
If $y\neq x$ then P2 gives $\Tr^{\infty}(\st_x\st_y^*)=0$, and if $y=x$ then by P3, P5, and P13 we have $\Tr^{\infty}(\st_x\st_y^*)=n_{d(x)}\gamma_{x,x^{-1},d(x)}=n_{d(x)}^2=1$. Thus 
$\Tr^{\infty}(\st_x\st_y^*)=\delta_{x,y}$ for all $x,y\in W$.
It follows that for $A,B\in \cJ$ with $A=\sum a_w\st_w$ and $B=\sum b_w\st_w$ we have 
\begin{align*}
\Tr^{\infty}(AB^*)&=\sum_{x,y\in W}n_{d(x)}n_{d(y)}a_xb_y\Tr^{\infty}(\st_x\st_y^*)=\sum_{x\in W}a_xb_{x}=\Tr^{\infty}(B^*A)
\end{align*}
and $\langle A,A\rangle^{\infty}=\sum_{w\in W} a_w^2$, and (1) and (2) follow.

Let $x,y,z\in W$. By P7 and the fact that $\gamma_{u,v,w}=\gamma_{v^{-1},u^{-1},w^{-1}}$ (for $u,v,w\in W$) we have
$
\gamma_{x^{-1},z,y^{-1}}=\gamma_{y^{-1},x^{-1},z}=\gamma_{x,y,z^{-1}}
$
and so by (2) we have
\begin{align*}
\langle \st_x\st_y,\st_z\rangle^{\infty}&=\gamma_{x,y,z^{-1}}=\gamma_{x^{-1},z,y^{-1}}=\langle \st_y,\st_x^*\st_z\rangle^{\infty}
\end{align*}
and (3) follows. 
\end{proof}

\begin{remark}
Theorem~\ref{thm:innerproductonJ} is expressed in the setup for general weighted Hecke algebras. It is clear that the statement of this theorem also applies to extended affine Hecke algebras (the proof applies verbatim). 
\end{remark}

\subsection{The Plancherel Theorem for $\tilde{\sA}_n$}\label{sec:plancherel}

We now return to the case that $\Hext$ is the extended affine Hecke algebra of type~$\tilde{\sA}_n$. The \textit{Plancherel Theorem} is a spectral decomposition of the canonical trace functional $\Tr:\Hext\to\ZZ[\sq,\sq^{-1}]$. We now explicitly describe this decomposition in type $\tilde{\sA}_n$, following \cite{AP:05,Opd:04}. Recall that $\sv=-\sq^{-1}$. Let
$$
C_{\la}(\sq)=\sq^{-n(n+1)}\prod_{i=1}^{r(\la)}\frac{\sq^{\la_i^2-\la_i}(1-\sq^{-2})^{\la_i}}{1-\sq^{-2\la_i}}\quad\text{and}\quad 
c^{\la}(\zeta_{\la})=\prod_{\substack{1\leq i<j\leq r(\la)\\ 1\leq k\leq \la_j}}\frac{1-\sv^{\la_i-\la_j+2k}z_i^{-1}z_j}{1-\sv^{-\la_i-\la_j+2k}z_i^{-1}z_j}.
$$
Note that each factor in the numerator of $c^{\la}(\zeta_{\la})$ has strictly negative degree, and each factor in the denominator has degree at least~$0$.

If $p(\zeta_{\la})\in\sR[\zeta_{\la}]$ we extend the notation from Section~\ref{sec:rings} and write 
$$
\bigg[\frac{p(\zeta_{\la})}{c^{\la}(\zeta_{\la})c^{\la}(\zeta_{\la}^{-1})}\bigg]_{\mathrm{ct}}
$$
for the coefficient of $\zeta_{\la}^0$ in the series expansion of $p(\zeta_{\la})/(c^{\la}(\zeta_{\la})c^{\la}(\zeta_{\la}^{-1}))$, where the rational function is expanded using 
\begin{align}\label{eq:expansion}
\frac{1}{1-\sv^{\la_i-\la_j+2k}z_i^{-1}z_j}=\sum_{r\geq 0}\sv^{(\la_i-\la_j+2k)r}z_i^{-r}z_j^r
\end{align}
(note that with this choice of expansion, the degree in $\sq$ remains bounded from above).

For $h_1,h_2\in \Hext$ define
$$
\langle h_1,h_2\rangle_{\la}=\frac{C_{\la}(\sq)}{|G_{\la}|}\bigg[\frac{\chi_{\la}(h_1h_2^*)}{c^{\la}(\zeta_{\la})c^{\la}(\zeta_{\la}^{-1})}\bigg]_{\mathrm{ct}}.
$$
The \textit{Plancherel Theorem} is the following ``spectral decomposition'' of the inner product $\langle\cdot,\cdot\rangle$ (hence the trace functional $\Tr$), proved in \cite{Opd:04} and \cite{AP:05}.

\begin{thm}[{\cite{Opd:04},\cite[Remark 5.6]{AP:05}}]\label{thm:plancherel}
For all $h_1,h_2\in \Hext$ we have 
$$
\langle h_1,h_2\rangle=\sum_{\la\vdash n+1}\langle h_1,h_2\rangle_{\la}.
$$
\end{thm}

\begin{remark}\label{rem:analyticexpression1}
The Plancherel Theorem is usually expressed as an analytic statement, and some comments are required for the translation between Theorem~\ref{thm:plancherel} and this analytic statement from~\cite{AP:05,Opd:04}. Specialise $\sq\to q$ with $q>1$ a real number, and extend scalars of $\Hext$ to~$\mathbb{C}$. Let $\mathbb{T}$ be the group of complex numbers of modulus~$1$ and for $\la\vdash n+1$ let $dz_{\la}$ be normalised Haar measure on the group 
$
\mathbb{T}_{\la}=\{(z_1,z_2,\ldots,z_{r(\la)})\in\mathbb{T}^{r(\la)}\mid z_1^{\la_1}z_2^{\la_2}\cdots z_{r(\la)}^{\la_{r(\la)}}=1\}.
$
On specialising the ``variables'' $z_1,\ldots,z_{r(\la)}$ in our representations to complex numbers of modulus~$1$ with $z_1^{\la_1}z_2^{\la_2}\cdots z_{r(\la)}^{\la_{r(\la)}}=1$, properties of the Haar measure gives
$$
\langle h_1,h_2\rangle_{\la}=\frac{C_{\la}(q)}{|G_{\la}|}\int_{\mathbb{T}_{\la}}\frac{\chi_{\la}(h_1h_2^*)}{c^{\la}(z_{\la})c^{\la}(z_{\la}^{-1})}dz_{\la},
$$
where $\langle h_1,h_2\rangle_{\la}$ is defined as above. To see this, expand the integrand into a series in the variables $z_1,\ldots,z_{r(\la)}$  using~(\ref{eq:expansion}), noting that this choice of expansion gives an absolutely convergent series since $q>1$. Then integrate term by term using $\int_{\mathbb{T}_{\la}}z_{\la}^{\gamma}dz_{\la}=\delta_{\gamma,0}$, and hence the integral on the right hand side of the above equation gives the constant term (in $z_{\la}$) of the expansion, which by definition is $\langle h_1,h_2\rangle_{\la}$. Thus, since $c^{\la}(z_{\la}^{-1})=\overline{c^{\la}(z_{\la})}$ for $z_{\la}\in\mathbb{T}_{\la}$ (here the bar indicates complex conjugation) we obtain the analytic expression for the Plancherel Theorem:
$$
\langle h_1,h_2\rangle=\sum_{\la\vdash n+1}\frac{C_{\la}(q)}{|G_{\la}|}\int_{\mathbb{T}_{\la}}\chi_{\la}(h_1h_2^*)d\mu_{\la}(z_{\la})\quad\text{where}\quad d\mu_{\la}(z_{\la})=\frac{dz_{\la}}{|c^{\la}(z_{\la})|^2}.
$$
The measure $d\mu_{\la}$ is the \textit{Plancherel measure}. 
\end{remark}

\begin{example}
Consider $n=2$ (type $\tilde{\sA}_2$). Then $\cP(3)=\{(1,1,1),(2,1),(3)\}$. We have  
\begin{align*}
c^{(1,1,1)}(\zeta_{(1,1,1)})&=\frac{(1-\sq^{-2}z_1^{-1}z_2)(1-\sq^{-2}z_2^{-1}z_3)(1-\sq^{-2}z_1^{-1}z_3)}{(1-z_1^{-1}z_2)(1-z_2^{-1}z_3)(1-z_1^{-1}z_3)}&&\text{where $z_1z_2z_3=1$}\\
c^{(2,1)}(\zeta_{(2,1)})&=\frac{1+\sq^{-3}z_1^{-1}z_2}{1+\sq z_1^{-1}z_2}&&\text{where $z_1^2z_2=1$},
\end{align*} 
and $c^{(3)}(\zeta_{(3)})=1$. We have $|G_{(1,1,1)}|=6$ and $|G_{(2,1)}|=|G_{(3)}|=1$. For the analytic formulation of the Plancherel Theorem, note that if $\la=(3)$ then $\mathbb{T}_{\la}=\{z\in\mathbb{C}\mid z^3=1\}=\{1,\omega,\omega^2\}$ where $\omega=e^{2\pi i/3}$, and so the normalised Haar measure on $\mathbb{T}_{\la}$ is the discrete measure assigning mass $1/3$ to each atom, and we recover the analytic formulation of the Plancherel Theorem for $\tilde{\sA}_2$ computed explicitly in \cite[Theorem~4.3]{Par:14}. 
\end{example}

\subsection{Asymptotic Plancherel Theorem}\label{subsec:apt}

The following \textit{asymptotic Plancherel Theorem} gives a spectral decomposition of the inner product $\langle\cdot,\cdot\rangle^{\infty}$ on~$\cJ$ (and hence the trace functional $\Tr^{\infty}$ on~$\cJ$) from Definition~\ref{defn:canonicalasymptotictrace}, mirroring the standard Plancherel Theorem at the asymptotic level.

\begin{thm}\label{thm:AsymptoticPlancherel}
For $A,B\in\cJ$ we have
$$
\langle A,B\rangle^{\infty}=\sum_{\la\vdash n+1}\langle A,B\rangle_{\la}^{\infty}\quad\text{where}\quad \langle A,B\rangle_{\la}^{\infty}=\frac{1}{|G_{\la}|}\bigg[\chi_{\la}^{\infty}(AB^*)\prod_{\alpha\in\Phi_{G_{\la}}}(1-\zeta_{\la}^{\alpha})\bigg]_{\mathrm{ct}}.
$$
Moreover, for each $\la\vdash n+1$ the bilinear form $\langle\cdot,\cdot\rangle_{\la}^{\infty}$ is an inner product on the $\ZZ$-module $\cJ_{\la}$, and the elements $\st_w$, $w\in\Delta_{\la}$, form an orthonormal basis. 
\end{thm}

The proof of Theorem~\ref{thm:AsymptoticPlancherel} is given after the following two preliminary results. Note that the notation $\langle A,B\rangle_{\la}^{\infty}$ for $A,B\in\cJ$ is a natural extension of the notation $\langle f(\zeta_{\la}),g(\zeta_{\la})\rangle_{\la}^{\infty}$ for $f(\zeta_{\la}),g(\zeta_{\la})\in \ZZ[\zeta_{\la}]$ from Section~\ref{sec:rings}.

Recall the definition of the specialisation $\sp$ from Section~\ref{sec:leading}.

\begin{prop}\label{prop:degrees}
For $\la\vdash n+1$ we have
$$
\deg \frac{C_{\la}(\sq)}{c^{\la}(\zeta_{\la})c^{\la}(\zeta_{\la}^{-1})}=-2\ell(\sw_{\la'})
\quad\text{and}\quad
\sp\bigg(\frac{\sq^{2\ell(\sw_{\la'})}C_{\la}(\sq)}{c^{\la}(\zeta_{\la})c^{\la}(\zeta_{\la}^{-1})}\bigg)=\prod_{\alpha\in\Phi_{G_{\la}}}(1-\zeta_{\la}^{\alpha}). 
$$
\end{prop}

\begin{proof}
We have $\deg C_{\la}(\sq)=-n(n+1)+\sum_{i=1}^{r(\la)}\la_i(\la_i-1)=-(n+1)^2+\sum_{i=1}^{r(\la)}\la_i^2$, and 
$$
\deg\frac{1}{c^{\la}(\zeta_{\la})c^{\la}(\zeta_{\la}^{-1})}=\sum_{\substack{1\leq i<j\leq r(\la)\\1\leq k\leq \la_j}}2(\la_i+\la_j-2k)=\sum_{1\leq i<j\leq r(\la)}2\la_j(\la_i-1).
$$
Note that
$
\sum_{1\leq i<j\leq r(\la)}\la_j=\sum_{k\geq 1}(k-1)\la_k'=\ell(\sw_{\la'}),
$
and so
\begin{align*}
\deg\frac{C_{\la}(\sq)}{c^{\la}(\zeta_{\la})c^{\la}(\zeta_{\la}^{-1})}&=-(n+1)^2+\sum_{i=1}^{r(\la)}\la_i^2+\sum_{1\leq i<j\leq r(\la)}2\la_i\la_j-2\ell(\sw_{\la'})=-2\ell(\sw_{\la'})
\end{align*}
and hence the first statement. 

The specialisation of $\sq^{2\ell(\sw_{\la'})}C_{\la}(\sq)/(c^{\la}(\zeta_{\la})c^{\la}(\zeta_{\la}^{-1}))$ at $\sq^{-1}=0$ exists and is a nonzero rational function in~$\zeta_{\la}$ (by the first statement). Consider a typical term from $1/c^{\la}(\zeta_{\la})$:
\begin{align*}
\frac{1-\sv^{-\la_i-\la_j+2k}z_i^{-1}z_j}{1-\sv^{\la_i-\la_j+2k}z_i^{-1}z_j}&=(-\sq)^{\la_i+\la_j-2k}\frac{(-\sq)^{-\la_i-\la_j+2k}-z_i^{-1}z_j}{1-(-\sq)^{-\la_i+\la_j-2k}z_i^{-1}z_j}.
\end{align*}
The factor in front will be absorbed in the overall degree. Thus, on specialising, this term will contribute either $-z_i^{-1}z_j$ (in the case $\la_i+\la_j-2k>0$), or 
$
1-z_i^{-1}z_j
$
in the case that $\la_i+\la_j-2k=0$. The $-z_i^{-1}z_j$ term will cancel with the corresponding term from $1/c^{\la}(\zeta_{\la}^{-1})$, and so only the terms of the second type will ultimately appear. These terms occur if and only if $\la_i+\la_j=2k$ (with $1\leq k\leq \la_j$), and this forces $\la_i=\la_j$ and $k=\la_j$. Thus 
$$
\sp\bigg(\frac{\sq^{2\ell(\sw_{\la'})}C_{\la}(\sq)}{c^{\la}(\zeta_{\la})c^{\la}(\zeta_{\la}^{-1})}\bigg)=\prod_{1\leq i,j\leq r(\la),\,i\neq j,\,\la_i=\la_j}(1-z_i^{-1}z_j)=\prod_{\alpha\in\Phi_{G_{\la}}}(1-\zeta_{\la}^{-\alpha}).
$$
as required. 
\end{proof}

\begin{cor}\label{cor:defineasymptotic}
For $\la\vdash n+1$ and $u,v\in \Wext$ we have $\deg\langle T_u,T_v\rangle_{\la}\leq 0$, and if equality holds then $\deg\pi_{\la}(T_u;\sB_{\la})=\deg\pi_{\la}(T_v;\sB_{\la})=\ell(\sw_{\la'})$, and hence $u,v\in \Delta_{\la}$.
\end{cor}

\begin{proof}
By definition, we have 
$$
\langle T_u,T_v\rangle_{\la}=\frac{C_{\la}(\sq)}{|G_{\la}|}\bigg[\frac{\chi_{\la}(T_uT_{v^{-1}})}{c^{\la}(\zeta_{\la})c^{\la}(\zeta_{\la}^{-1})}\bigg]_{\mathrm{ct}}.
$$
By Theorem~\ref{thm:fullbounded}, $\deg\pi_{\la}(T_u;\sB_{\la})\leq \ell(\sw_{\la'})$ and $\deg\pi_{\la}(T_v;\sB_{\la})\leq \ell(\sw_{\la'})$. Thus $\deg\chi_{\la}(T_uT_{v^{-1}})\leq 2\ell(\sw_{\la'})$, and if equality holds then necessarily $\deg\pi_{\la}(T_u;\sB_{\la})=\deg\pi_{\la}(T_v;\sB_{\la})=\ell(\sw_{\la'})$, and so $u,v\in\Delta_{\la}$ (again by Theorem~\ref{thm:fullbounded}). The result now follows from Proposition~\ref{prop:degrees}.
\end{proof}

We now give the proof of the asymptotic Plancherel Theorem.

\begin{proof}[Proof of Theorem~\ref{thm:AsymptoticPlancherel}]
By bilinearity it is sufficient to prove that
$$
\sum_{\la\vdash n+1}\langle \st_u,\st_v\rangle_{\la}^{\infty}=\delta_{u,v}\quad\text{for all $u,v\in\Wext$}.
$$
If $u\not\sim_{LR} v$ then $\st_u\st_v^*=0$ and so $\chi_{\la}^{\infty}(\st_u\st_v^*)=0$, hence $\langle\st_u,\st_v\rangle_{\la}^{\infty}=0$ for all $\la\vdash n+1$. Thus suppose that $u,v\in\Delta_{\mu}$ for some $\mu\vdash n+1$. Since $\langle T_u,T_v\rangle=\delta_{u,v}$ the Plancherel Theorem (Theorem~\ref{thm:plancherel}) gives
$$
\delta_{u,v}=\sum_{\la\vdash n+1}\langle T_u,T_v\rangle_{\la},
$$
and by Corollary~\ref{cor:defineasymptotic} each specialisation $\sp\big(\langle T_u,T_v\rangle_{\la}\big)$ exists, and is zero unless $\la=\mu$. Thus 
$
\delta_{u,v}=\sp\big(\langle T_u,T_v\rangle_{\mu}\big). 
$
Now, by Theorem~\ref{thm:fullbounded} we have
$$
\chi_{\mu}(T_uT_{v^{-1}})=\sq^{2\ell(\sw_{\mu'})}\tr(\fc_{\mu}(u)\fc_{\mu}(v)^*)+(\textrm{terms of strictly lower degree}), 
$$
and thus by Proposition~\ref{prop:degrees} we have
\begin{align*}
C_{\mu}(\sq)\frac{\chi_{\mu}(T_uT_{v^{-1}})}{c^{\mu}(\zeta_{\mu})c^{\mu}(\zeta_{\mu}^{-1})}&=\chi_{\mu}^{\infty}(\st_u\st_v^*)\prod_{\alpha\in\Phi_{G_{\mu}}}(1-\zeta_{\mu}^{\alpha})+(\textrm{terms of degree $<0$}).
\end{align*}
Thus
\begin{align}\label{eq:spec}
\delta_{u,v}=\sp\big(\langle T_u,T_v\rangle_{\mu}\big)=\frac{1}{|G_{\mu}|}\bigg[\chi_{\mu}^{\infty}(\st_u\st_v^*)\prod_{\alpha\in\Phi_{G_{\mu}}}(1-\zeta_{\mu}^{\alpha})\bigg]_{\mathrm{ct}}=\langle \st_u,\st_v\rangle_{\mu}^{\infty}
\end{align}
as required. 
\end{proof}

The asymptotic Plancherel Theorem has the following important consequences. Note that the first theorem shows that the set of elements $w\in \Wext$ such that the matrix $\pi_{\la}(T_w;\sB_\la)$ attains the bound $\ell(\sw_{\la'})$ is precisely the two-sided cell~$\Delta_{\la}$ (improving Corollary \ref{cor:recognise1}).

\begin{thm}\label{thm:recognise2a}
We have $\fc_{\la}(w)\neq 0$ if and only if $w\in\Delta_{\la}$. 
\end{thm}

\begin{proof}
By Corollary~\ref{cor:recognise1}, it only remains to show that if $w\in\Delta_{\la}$ then $\deg\pi_{\la}(T_w;\sB_{\la})=\ell(\sw_{\la'})$. If $w\in\Delta_{\la}$ then by the asymptotic Plancherel Theorem (Theorem~\ref{thm:AsymptoticPlancherel}) we have $1=\langle \st_w,\st_w\rangle_{\la}^{\infty}$ and thus $\deg\langle T_w,T_w\rangle_{\la}=0$ (see~(\ref{eq:spec})), and the result follows from Corollary~\ref{cor:defineasymptotic}. 
\end{proof}

\begin{thm}\label{thm:recognise2b}
We have $\cJ_{\la}\cong \fC_{\la}$ as $\ZZ$-algebras, with $\st_w\mapsto \fc_{\la}(w)$. Thus $\cJ\cong \fC$. 
\end{thm}

\begin{proof}
By Proposition~\ref{prop:Zalgebra} the map $\psi:\cJ_{\la}\to \fC_{\la},\st_w\mapsto\fc_{\la}(w)$ is a surjective ring homomorphism. To prove that $\psi$ is injective, if $a=\sum_{w\in\Delta_{\la}}a_w\st_w\in\cJ_{\la}$ with $\psi(a)=0$ then for each $v\in \Delta_{\la}$ we have
$
\sum_{w\in \Delta_{\la}}a_w\fc_{\la}(w)\fc_{\la}(v)^*=0.
$
Taking traces, multiplying by $\prod_{\alpha\in\Phi_{G_{\la}}}(1-\zeta_{\la}^{\alpha})$, and applying Theorem~\ref{thm:AsymptoticPlancherel} gives $a_v=0$. 
\end{proof}

\begin{remark}
Following Remark~\ref{rem:analyticexpression1}, the analytic expression for the asymptotic Plancherel Theorem is 
$$
\langle A,B\rangle^{\infty}=\sum_{\la\vdash n+1}\frac{1}{|G_{\la}|}\int_{\mathbb{T}_{\la}}\chi_{\la}^{\infty}(AB^*)d\mu_{\la}^{\infty}(z_{\la})\quad\text{where}\quad d\mu_{\la}^{\infty}(z_{\la})=\bigg|\prod_{\alpha\in\Phi_{G_{\la}}^+}(1-z_{\la}^{-\alpha})\bigg|^2\,dz_{\la}
$$
for $A,B\in\cJ$, where $d\mu_{\la}^{\infty}(z_{\la})$ is the \textit{asymptotic Plancherel measure}. 
\end{remark}

\begin{remark} The results of this subsection verify \cite[Conjecture~9.8]{GP:19} and \cite[Conjecture~6.15]{GLP:23} for affine type~$\sA$.
\end{remark}

The results of this section, together with the killing and boundedness properties of Section~\ref{sec:killbound}, imply that the matrix representations $\pi_{\la}$, $\la\vdash n+1$, form a  \textit{balanced system of cell representations} for $\tilde{\sA}_n$, in the sense of \cite[Definition~1.5]{GP:19b}. 

\begin{cor}\label{cor:balancedsystem} The matrix representations $\pi_{\la}(\,\cdot\, ;\sB_{\la})$, with $\la\vdash n+1$, form a balanced system of cell representations.
\end{cor}

\begin{proof}
We verify properties $\mathbf{B1}$--$\mathbf{B6}$ from \cite[Definition~1.5]{GP:19b}. Property $\mathbf{B1}$ is the killing property (Theorem~\ref{thm:fullkilling}), and property $\mathbf{B2}$ is boundedness (Theorem~\ref{thm:fullbounded}, with bounds $\ba_{\la}=\ell(\sw_{\la'})$). Property~$\mathbf{B6}$ (monotonicity of the bounds) is then immediate from Lemma~\ref{lem:basic}. Property $\mathbf{B3}$ is the statement that $\pi_{\la}$ recognises precisely $\Delta_{\la}$ (Theorem~\ref{thm:recognise2a}), and property $\mathbf{B4}$ (freeness of the leading matrices over $\ZZ$) follows from the isomorphism $\fC_{\la}\cong \cJ_{\la}$ in Theorem~\ref{thm:recognise2b}. Finally, if $z\in\Delta_{\la}$ let $d\in\cD$ be the unique distinguished involution with $z\sim_R d$. Then in $\cJ_{\la}$ we have $\st_d\st_z=\st_z$ (by P2, P7, P13), and hence by Theorem~\ref{thm:recognise2b} we have $\fc_{\la}(d)\fc_{\la}(z)=\fc_{\la}(z)$, verifying~$\mathbf{B5}$.
\end{proof}

\section{Maximal length double coset representatives}\label{sec:maximallengthelts}

By \cite[Proposition~2.4]{KL:79} it follows that if $w\in\Ga_{\la}$ (respectively $w\in \Ga_{\la}^{-1}$) then $\ell(\sw_{\la'}w)=\ell(w)-\ell(\sw_{\la'})$ (respectively $\ell(w\sw_{\la'})=\ell(w)-\ell(\sw_{\la'})$). In particular, if $w\in\Ga_{\la}\cap \Ga_{\la}^{-1}$ then $w$ is of maximal length in its double coset $W_{\la'}wW_{\la'}$. 

\begin{defn}\label{defn:maxlength}
For $\gamma\in P^{(\la)}$ let $\sm_{\gamma}$ be the  longest element of the double coset $W_{\la'}\su_{\la}^{-1}\tau_{\gamma}\su_{\la}W_{\la'}$. 
\end{defn}

Ultimately we will see in Theorem~\ref{thm:leading0} that $\Ga_{\la}\cap\Ga_{\la}^{-1}=\{\sm_{\ga}\mid \ga\in P^{(\la)}_+\}$. We first record the following result on $G_{\la}$-invariance.

\begin{prop}\label{prop:dominant} If $\gamma\in P^{(\la)}$ and $g\in G_{\la}$ then $\sm_{g\gamma}=\sm_{\gamma}$. 
\end{prop}

\begin{proof}
By \cite[Lemma~3.21]{GLP:23} we have
$
\tau_{g\gamma}=g\tau_{\gamma}g^{-1}$ for all $g\in G_{\la}$, and hence
$$
\su_{\la}^{-1}\tau_{g\gamma}\su_{\la}=(\su_{\la}^{-1}g\su_{\la})(\su_{\la}^{-1}\tau_{\gamma}\su_{\la})(\su_{\la}^{-1}g^{-1}\su_{\la}).
$$
By Proposition~\ref{prop:Gla1} we have $\su_{\la}^{-1}g\su_{\la}\in W_{\la'}$, so $\su_{\la}^{-1}\tau_{g\gamma}\su_{\la}\in W_{\la'}\su_{\la}^{-1}\tau_{\gamma}\su_{\la}W_{\la'}$, hence the result.
\end{proof}

\begin{remark}
In fact $\sm_{\gamma_1}=\sm_{\gamma_2}$ if and only if $\gamma_2\in G_{\la}\gamma_1$. The ``if'' direction is Proposition~\ref{prop:dominant}, and the ``only if'' direction follows from later results of this paper. Specifically, by the ``if'' direction we may assume that $\ga_1,\ga_2\in P^{(\la)}_+$, and then by Theorem~\ref{thm:leading0} if $\ga_1\neq\ga_2$ then $\pi_{\la}(C_{\sm_{\ga_1}})=\fs_{\ga_1}(\zeta_{\la})\pi_{\la}(\clap)\neq \fs_{\ga_2}(\zeta_{\la})\pi_{\la}(\clap)=\pi_{\la}(C_{\sm_{\ga_1}})$, and so $\sm_{\ga_1}\neq \sm_{\ga_2}$. 
\end{remark}

The following theorem plays an important role in computing $\Ga_{\la}\cap\Ga_{\la}^{-1}$. The proof of the theorem is technical, and so for readability we present it in Appendix~\ref{app:proofs}.

\begin{thm}\label{thm:mgamma}
Let $\gamma\in P^{(\la)}$. There exist $x,y\in W_{\la'}$ such that 
$\sm_{\gamma}=x\su_{\la}^{-1}\tau_{\gamma}\su_{\la}y$ with 
$$
\ell(\sm_{\gamma})=\ell(x)+\ell(\su_{\la}^{-1}\tau_{\gamma}\su_{\la})+\ell(y)\quad\text{and}\quad \ell(x)+\ell(y)=\ell(\sw_{\la'}).
$$
If $\gamma,\gamma'\in P^{(\la)}_+$ with $\gamma+Q_{\la}\peq_{\la}\gamma'+Q_{\la}$ then $\ell(\sm_{\gamma})\leq\ell(\sm_{\gamma'})$ with equality if and only if $\ga=\ga'$. 
\end{thm}

\begin{proof} See Appendix~\ref{app:proofs}. 
\end{proof}

\section{Lusztig's asymptotic algebra}\label{sec:asymptotic}

In this section we give our description of Lusztig's asymptotic algebra $\cJ=\bigoplus_{\la\vdash n+1}\cJ_{\la}$. Section~\ref{sec:pre} gives preliminary reductions, and outlines our strategy, and in Section~\ref{sec:62} we prove our main theorem. We make some concluding remarks, and give an example, in Section~\ref{sec:63}. 

\subsection{Preliminary reductions}\label{sec:pre}

Recall that $\Ga_{\la}$ denotes the right cell containing $\sw_{\la'}$, and thus $\Ga_{\la}^{-1}$ is the left cell containing~$\sw_{\la'}$. Let
$$
\cJ_{\Ga_{\la}\cap\Ga_{\la}^{-1}}=\mathrm{span}_{\ZZ}\{\st_w\mid w\in\Ga_{\la}\cap\Ga_{\la}^{-1}\}.
$$
It is clear that $\cJ_{\Ga_{\la}\cap\Ga_{\la}^{-1}}$ is a ring. The following fundamental reduction result, due to Xi, shows that the asymptotic algebra $\cJ_{\la}$ is completely determined by the structure of the ring $\cJ_{\Ga_{\la}\cap\Ga_{\la}^{-1}}$. Recall that $N_{\la}=(n+1)!/(\la_1!\la_2!\cdots\la_{r(\la)}!)$. 

\begin{thm}[{\cite[Theorems~2.3.2, 8.4.2]{Xi:02}}]\label{thm:xi}
We have $\cJ_{\la}\cong \mathrm{Mat}_{N_{\la}}(\cJ_{\Ga_{\la}\cap\Ga_{\la}^{-1}})$ as $\ZZ$-algebras. 
\end{thm}

Thus the main work in determining $\cJ_{\la}$ is reduced to computing the set $\Ga_{\la}\cap\Ga_{\la}^{-1}$ and understanding the ring $\cJ_{\Ga_{\la}\cap\Ga_{\la}^{-1}}$. Our strategy is as follows. By Theorem~\ref{thm:recognise2b} we have the isomorphism $\cJ_{\Ga_{\la}\cap\Ga_{\la}^{-1}}\cong \fC_{\Ga_{\la}\cap\Ga_{\la}^{-1}}$ where
$$
 \fC_{\Ga_{\la}\cap\Ga_{\la}^{-1}}=\mathrm{span}_{\ZZ}\{\fc_{\la}(w)\mid w\in\Ga_{\la}\cap\Ga_{\la}^{-1}\},
$$
and hence it suffices to consider our leading matrices. In Theorem~\ref{thm:indexing} we show that $w\in\Ga_{\la}\cap\Ga_{\la}^{-1}$ if and only if $\fc_{\la}(w)$ has a nonzero entry in the $(\su_{\la},\su_{\la})$-position, and using this we show in Proposition~\ref{prop:path} that each maximal length double coset representative $\sm_{\gamma}$, $\gamma\in P_+^{(\la)}$ lies in $\Ga_{\la}\cap\Ga_{\la}^{-1}$. Then, finally, in Theorem~\ref{thm:leading0} we show that $\Ga_{\la}\cap\Ga_{\la}^{-1}$ is equal to the set $\{\sm_{\gamma}\mid \gamma\in P_+^{(\la)}\}$, and explicitly compute the leading matrices $\fc_{\la}(\sm_{\gamma})$ in terms of Schur functions, thus showing that $\cJ_{\Ga_{\la}\cap\Ga_{\la}^{-1}}$ is isomorphic to the ring $\ZZ[\zeta_{\la}]^{G_{\la}}$. 

\subsection{The set $\Ga_{\la}\cap\Ga_{\la}^{-1}$ and the ring $\cJ_{\Ga_{\la}\cap\Ga_{\la}^{-1}}$}\label{sec:62}

The following theorem shows that our leading matrices can be used to detect the set $\Ga_{\la}\cap\Ga_{\la}^{-1}$.

\begin{thm}\label{thm:indexing}
We have $w\in\Ga_{\la}\cap\Ga_{\la}^{-1}$ if and only if the matrix $\fc_{\la}(w)$ has a non-zero entry in the $(\su_{\la},\su_{\la})$-position. Moreover, if $w\in\Ga_{\la}\cap\Ga_{\la}^{-1}$ then the matrix $\fc_{\la}(w)$ has a unique non-zero entry. 
\end{thm}

\begin{proof}
If $w\in\Ga_{\la}\cap\Ga_{\la}^{-1}$ then by P2, P7 and P13 we have $\st_{\sw_{\la'}}\st_w=\st_w$ and $\st_w\st_{\sw_{\la'}}=\st_w$. Thus by Theorem~\ref{thm:recognise2b} we have $\fc_{\la}(\sw_{\la'})\fc_{\la}(w)=\fc_{\la}(w)$ and $\fc_{\la}(w)\fc_{\la}(\sw_{\la'})=\fc_{\la}(w)$. Since $\fc_{\la}(\sw_{\la'})=E_{\su_{\la},\su_{\la}}$ (see~(\ref{eq:leadinglap})), the former forces all nonzero entries of $\fc_{\la}(w)$ to lie on the $\su_{\la}$-row of the matrix, and the latter forces all nonzero elements of $\fc_{\la}(w)$ to lie on the $\su_{\la}$-column of the matrix. Hence $\fc_{\la}(w)$ has a unique nonzero entry, and this entry is in position $(\su_{\la},\su_{\la})$. 

Conversely, if $\fc_{\la}(w)$ has a nonzero entry in the $(\su_{\la},\su_{\la})$-entry, then $\fc_{\la}(\sw_{\la'})\fc_{\la}(w)\neq 0$. Thus by Theorem~\ref{thm:recognise2b} we have $\st_{\sw_{\la'}}\st_w\neq 0$, and so there is $y\in W$ with $\gamma_{\sw_{\la'},w,y}\neq 0$. Therefore $w\sim_R \sw_{\la'}$ (using P8). Similarly since $\fc_{\la}(w)\fc_{\la}(\sw_{\la'})\neq 0$ we have $w\sim_L\sw_{\la'}$, and so $w\in\Ga_{\la}\cap\Ga_{\la}^{-1}$. 
\end{proof}

In the following important proposition we use the combinatorics of $\lambda$-folded alcove paths to show that the $(\su_{\la},\su_{\la})$-entry of the leading matrix $\fc_{\la}(\sm_{\gamma})$ is non-zero, which implies that $\sm_{\gamma}\in\Ga_{\la}\cap\Ga_{\la}^{-1}$ by Theorem~\ref{thm:indexing}.

\begin{prop}\label{prop:path}
If $\gamma\in P^{(\la)}_+$ then there is an integer $c>0$ such that
$$
[\fc_{\la}(\sm_{\gamma})]_{\su_{\la},\su_{\la}}=c\zeta_{\la}^{\gamma}+(\textrm{$\ZZ$-linear combination of terms $\zeta_{\la}^{\gamma'}$ with $\gamma\not\peq_{\la}\gamma'$}).
$$
\end{prop}

\begin{proof}
By Theorem~\ref{thm:mgamma} we have $\sm_{\gamma}=x\su_{\la}^{-1}\tau_{\gamma}\su_{\la} y$ with $x,y\in W_{\la'}$ and $\ell(\sm_{\gamma})=\ell(x)+\ell(\su_{\la}^{-1}\tau_{\gamma}\su_{\la})+\ell(y)$ and $\ell(x)+\ell(y)=\ell(\sw_{\la'})$. Let $p_0$ be the path starting at $\su_{\la}$ of type $\vec{\sm}_{\gamma}=x\cdot (\su_{\la}^{-1}\tau_{\gamma}\su_{\la})\cdot y$ (where we choose any reduced expressions for $x$, $\su_{\la}^{-1}\tau_{\gamma}\su_{\la}$, and $y$) such that the first $\ell(x)$ steps are folds, the next $\ell(\su_{\la}^{-1}\tau_{\gamma}\su_{\la})$ steps are crossings, and the final $\ell(y)$ steps are folds. The following observations show that $p_0$ is a $\la$-folded alcove path.
\begin{compactenum}[$(1)$]
\item Since $\ell(\su_{\la}s_j)=\ell(\su_{\la})+1$ and $\su_{\la}s_j\in {^\la}W$ for all $j\in J_{\la'}$ the first $\ell(x)$ steps of $p_0$ are positive folds (not bounces). 
\item To check that the next $\ell(\su_{\la}^{-1}\tau_{\gamma}\su_{\la})$ steps can be taken to be crossings we must show that this part of the path remains in~$\cA_{\la}$ (and hence there are no forced bounces). Note that the starting alcove of this part of the path is $\su_{\la}\cA_0$ (which lies in $\cA_{\la}$) and the end alcove is $\su_{\la}(\su_{\la}^{-1}\tau_{\gamma}\su_{\la})\cA_0=\tau_{\gamma}\su_{\la}\cA_0$. Thus by \cite[Theorem~3.11]{GLP:23} the end alcove lies in $\cA_{\la}$, and since the path is of reduced type, and $\cA_{\la}$ is convex (being an intersection of half-spaces) the entire path from $\su_{\la}\cA_0$ to $\tau_{\la}\su_{\la}\cA_0$ lies in $\cA_{\la}$ (see \cite[Proposition~3.94]{AB:08}). 
\item To show that the final $\ell(y)$ steps are positive folds, note first that if $s\in J_{\la'}$ then $\tau_{\la}\su_{\la}s\cA_0\subseteq\cA_{\la}$ (since $\su_{\la}s\in {^\la}W$, and apply \cite[Theorem~3.11]{GLP:23}). Moreover, since $\tau_{\gamma}\su_{\la}=t_{\gamma}\sy_{\gamma}\su_{\la}$ and $\ell(\sy_{\gamma}\su_{\la}s)=\ell(\sy_{\gamma}\su_{\la})+1$ we have that $\tau_{\gamma}\su_{\la}\cA_0$ lies on the positive side of the hyperplane separating $\tau_{\gamma}\su_{\la}\cA_0$ from $\tau_{\gamma}\su_{\la}s\cA_0$, and hence the result. 
\end{compactenum}
Now, since $\ell(x)+\ell(y)=\ell(\sw_{\la'})$ we have $\cQ_{\la}(p_0)=(\sq-\sq^{-1})^{\ell(\sw_{\la'})}$, and since $\mathrm{end}(p_0)=\tau_{\gamma}\su_{\la}$ we have $\wt(p_0)=\gamma$ and $\theta^{\la}(p_0)=\su_{\la}$. Thus by Theorem~\ref{thm:pathformula} we have
$$
[\pi_{\la}(T_{\sm_{\gamma}};\sB_{\la})]_{\su_{\la},\su_{\la}}=(\sq-\sq^{-1})^{\ell(\sw_{\la'})}\zeta_{\la}^{\gamma}+\sum_{p\in\cP_{\la}(\vec{\sm}_{\gamma},\su_{\la})_{\su_{\la}}\backslash\{p_0\}}\cQ_{\la}(p)\zeta_{\la}^{\wt(p)}. 
$$
It follows from Theorem~\ref{thm:fullbounded} and the above equation that
$$
\sq^{-\ell(\sw_{\la'})}[\pi_{\la}(T_{\sm_{\gamma}};\sB_{\la})]_{\su_{\la},\su_{\la}}=\zeta_{\la}^{\gamma}+(\textrm{$\ZZ[\sq^{-1}]$-linear combination of terms $\zeta_{\la}^{\gamma'}$ with $\gamma'\in P$}),
$$
and so to prove the proposition it is sufficient to show that if $p\in\cP_{\la}(\vec{\sm}_{\gamma},\su_{\la})_{\su_{\la}}$ with $\wt(p)=\gamma'$ and $\gamma\peq_{\la}\gamma'$ then:
\begin{compactenum}[(a)]
\item if $\gamma\prec_{\la}\gamma'$ then $\deg\cQ_{\la}(p)< \ell(\sw_{\la'})$;
\item if $\gamma=\gamma'$ then $\deg\cQ_{\la}(p)\leq \ell(\sw_{\la'})$, and if equality holds then $\cQ_{\la}(p)$ has positive leading coefficient.
\end{compactenum}

Consider a path $p\in\cP_{\la}(\vec{\sm}_{\gamma},\su_{\la})_{\su_{\la}}$ with $\wt(p)=\gamma'$. Since $\theta^{\la}(p)=\su_{\la}$ \cite[Corollary~3.13]{GLP:23} gives $\mathrm{end}(p)=\tau_{\gamma'}\su_{\la}$. Let $N=f(p)+b(p)$ be the total number of folds and bounces in $p$. If $\vec{\sm}_{\gamma}=s_1s_2\cdots s_{\ell}\pi$ then, since $p$ starts at $\su_{\la}$, we have
$
\mathrm{end}(p)=\su_{\la}s_1\cdots \hat{s}_{i_1}\cdots \hat{s}_{i_N}\cdots s_{\ell}\pi,
$
where the folds and bounces of $p$ occur at the indices $i_k$ with $1\leq i_1<i_2<\cdots<i_N\leq \ell$ (and $\hat{s}_{i_k}$ denotes omission of the generator in the expression). It follows that 
$
\ell(\su_{\la}^{-1}\tau_{\gamma'}\su_{\la})\leq \ell(\sm_{\gamma})-N,
$
and thus $N\leq \ell(\sm_{\gamma})-\ell(\su_{\la}^{-1}\tau_{\gamma'}\su_{\la})$. Since $\ell(\sm_{\gamma'})=\ell(\su_{\la}^{-1}\tau_{\gamma'}\su_{\la})+\ell(\sw_{\la'})$ we have
$$
N\leq \ell(\sw_{\la'})+\ell(\sm_{\gamma})-\ell(\sm_{\gamma'}).
$$
Since $\cQ_{\la}(p)=(-\sq)^{-b(p)}(\sq-\sq^{-1})^{f(p)}$ it follows that $\deg\cQ_{\la}(p)\leq \ell(\sw_{\la'})+\ell(\sm_{\gamma})-\ell(\sm_{\gamma'})$ with equality if and only if $f(p)=\ell(\sw_{\la'})+\ell(\sm_{\gamma})-\ell(\sm_{\gamma'})$ and $b(p)=0$. Both (a) and (b) now follow from the monotonicity statement in Theorem~\ref{thm:mgamma}.
\end{proof}

The following theorem is the main result of this paper. It gives a $\la$-relative version of \cite[Proposition~8.6]{Lus:83} (see also \cite[Theorem~2.22(b)]{NR:03}), and recovers the main result of Xi~\cite{Xi:02}. We note that our description of $\Ga_{\la}\cap\Ga_{\la}^{-1}$ and $\cJ_{\Ga_{\la}\cap\Ga_{\la}^{-1}}$ is rather different to the descriptions given by Xi~\cite{Xi:02} and Kim and Pylyavskyy~\cite{KP:23}. In particular, our description of $\Ga_{\la}\cap\Ga_{\la}^{-1}$ is in terms of the geometry of the fundamental $\la$-alcove~$\cA_{\la}$, and our description of the ring $\cJ_{\Ga_{\la}\cap\Ga_{\la}^{-1}}$ is in terms of symmetric functions and alcove paths, rather than representation rings of algebraic groups.

\begin{thm}\label{thm:leading0}
We have $\Ga_{\la}\cap\Ga_{\la}^{-1}=\{\sm_{\gamma}\mid \gamma\in P^{(\la)}_+\}$, and if $\gamma\in P^{(\la)}_+$ then 
$$
\pi_{\la}(C_{\sm_{\gamma}})=\fs_{\gamma}(\zeta_{\la})\pi_{\la}(C_{\sw_{\la'}}). 
$$
Moreover, $\fc_{\la}(\sm_{\gamma})=\fs_{\gamma}(\zeta_{\la})E_{\su_{\la},\su_{\la}}$, and the linear map
$$
\cJ_{\Ga_{\la}\cap\Ga_{\la}^{-1}}\to \ZZ[\zeta_{\la}]^{G_{\la}}\quad\text{with}\quad  \st_{\sm_{\gamma}}\mapsto \fs_{\gamma}(\zeta_{\la})
$$
is an isomorphism of unital rings. 
\end{thm}

\begin{proof}
We argue as follows.
\smallskip

\noindent\textit{Claim 1: If $w\in\Ga_{\la}\cap\Ga_{\la}^{-1}$ then $\pi_{\la}(C_w)=\mathfrak{f}_w(\zeta_{\la})\pi_{\la}(\clap)$ for some $\mathfrak{f}_w(\zeta_{\la})\in\ZZ[\zeta_{\la}]^{G_{\la}}$.} 
\smallskip

\noindent\textit{Proof of Claim 1.} If $w\in \Ga_{\la}\cap\Ga_{\la}^{-1}$ then by Theorem~\ref{thm:indexing} the matrix $\pi_{\la}(C_w;\sB_{\la})$ attains the bound $\ell(\sw_{\la'})$ in the $(\su_{\la},\su_{\la})$-position, and only in this position. Define $\mathfrak{f}_{w}(\zeta_{\la})=\sq^{-\ell(\sw_{\la'})}[\pi_{\la}(C_w;\sB_{\la})]_{\su_{\la},\su_{\la}}$, and so $\mathfrak{f}_w(\zeta_{\la})\in(\ZZ[\sq^{-1}])[\zeta_{\la}]$. Since $\ell(w\sw_{\la'})=\ell(\sw_{\la'}w)=\ell(w)-\ell(\sw_{\la'})$ Lemma~\ref{lem:CJformula}, combined with Theorem~\ref{thm:satake1}, gives
\begin{align}\label{eq:string}
\sq^{-2\ell(\sw_{\la'})}W_{\la'}(\sq^2)^2\pi_{\la}(C_w)=\pi_{\la}(\clap C_w\clap)=f_{\la}(C_w)\pi_{\la}(\clap).
\end{align}
Reading the $(\su_{\la},\su_{\la})$-entries (using Proposition~\ref{prop:formofmatrix}) gives
$
\sq^{-\ell(\sw_{\la'})}W_{\la'}(\sq^2)^2\,\mathfrak{f}_w(\zeta_{\la})=\sq^{\ell(\sw_{\la'})}f_{\la}(C_w).
$
In particular, $f_{\la}(C_w)$ is divisible (in $\sR[\zeta_{\la}]$) by $W_{\la'}(\sq^2)^2$, and we have
$$
\mathfrak{f}_w(\zeta_{\la})=\frac{\sq^{2\ell(\sw_{\la'})}}{W_{\la'}(\sq^2)^2}f_{\la}(C_w).
$$
It then follows from Corollary~\ref{cor:fbar} that $\overline{\mathfrak{f}_w(\zeta_{\la})}=\mathfrak{f}_w(\zeta_{\la})$, and since $\mathfrak{f}_w(\zeta_{\la})\in(\ZZ[\sq^{-1}])[\zeta_{\la}]$ this forces $\mathfrak{f}_w(\zeta_{\la})\in\ZZ[\zeta_{\la}]$. Then by Theorem~\ref{thm:symmetry} we have $\mathfrak{f}_w(\zeta_{\la})\in\ZZ[\zeta_{\la}]^{G_{\la}}$, and~(\ref{eq:string}) gives $\pi_{\la}(C_w)=\mathfrak{f}_w(\zeta_{\la})\pi_{\la}(\clap)$, completing the proof of Claim 1. 
\smallskip


\noindent\textit{Claim 2: We have $\{\sm_{\gamma}\mid \gamma\in P^{(\la)}_+\}\subseteq \Ga_{\la}\cap \Ga_{\la}^{-1}$, and if $\gamma\in P^{(\la)}$ then $\mathfrak{f}_{\sm_{\gamma}}(\zeta_{\la})=\fs_{\gamma}(\zeta_{\la})$.} 
\smallskip

\noindent\textit{Proof of Claim 2.} By Proposition~\ref{prop:path} we have $\deg[\pi_{\la}(T_{\sm_{\gamma}};\sB_{\la})]_{\su_{\la},\su_{\la}}=\ell(\sw_{\la'})$, and so by Theorem~\ref{thm:indexing} we have $\sm_{\gamma}\in\Ga_{\la}\cap\Ga_{\la}^{-1}$, hence the claimed containment. 

Let $\gamma\in P^{(\la)}$. Since $\sm_{g\gamma}=\sm_{\gamma}$ for all $g\in G_{\la}$ (see Proposition~\ref{prop:dominant}) we may assume that $\gamma\in P^{(\la)}_+$. By Proposition~\ref{prop:path} we have
\begin{align}\label{eq:positive}
\mathfrak{f}_{\sm_{\gamma}}(\zeta_{\la})=c\zeta_{\la}^{\gamma}+(\text{$\ZZ$-linear combination of terms $\zeta_{\la}^{\gamma'}$ with $\gamma\not\peq_{\la}\gamma'$}),
\end{align}
where $c>0$ is an integer. Since $\mathfrak{f}_{\sm_{\gamma}}(\zeta_{\la})\in\ZZ[\zeta_{\la}]^{G_{\la}}$, and since the Schur functions $\fs_{\gamma}(\zeta_{\la})$, $\gamma\in P^{(\la)}_+$, form a $\ZZ$-basis of this ring we have 
$$
\mathfrak{f}_{\sm_{\gamma}}(\zeta_{\la})=\sum_{\gamma'\in (P/Q_{\la})_+}a_{\gamma'}\fs_{\gamma'}(\zeta_{\la})
$$
for some integers $a_{\gamma'}$. By Theorem~\ref{thm:AsymptoticPlancherel} we have $\langle \mathfrak{f}_{\sm_{\gamma}}(\zeta_{\la}),\mathfrak{f}_{\sm_{\gamma}}(\zeta_{\la})\rangle_{\la}^{\infty}=1$, and then Lemma~\ref{lem:uniquebasis} gives $\sum_{\gamma'}a_{\gamma'}^2=1$. Thus $\mathfrak{f}_{\sm_{\gamma}}(\zeta_{\la})=\epsilon\fs_{\gamma_1}(\zeta_{\la})$ for some $\gamma_1\in P^{(\la)}_+$ with $\epsilon\in\{-1,1\}$. By~(\ref{eq:positive}) and the positivity in Lemma~\ref{lem:uniquebasis} we have $\epsilon=1$. Moreover, the triangularity in Lemma~\ref{lem:uniquebasis} gives $\gamma\peq_{\la}\gamma_1$, and hence $\gamma_1=\gamma$ by~(\ref{eq:positive}) completing the proof of Claim~2.
\smallskip

\noindent\textit{Claim 3: We have $\Ga_{\la}\cap\Ga_{\la}^{-1}=\{\sm_{\gamma}\mid \gamma\in P^{(\la)}_+\}$.}
\smallskip

\noindent\textit{Proof of Claim 3.} If $w\in\Ga_{\la}\cap\Ga_{\la}^{-1}$ then it follows from Claim 1 that $\fc_{\la}(w)=\mathfrak{f}_w(\zeta_{\la})E_{\su_{\la},\su_{\la}}$, with $\mathfrak{f}_w(\zeta_{\la})\in \ZZ[\zeta_{\la}]^{G_{\la}}$. If $w\notin\{\sm_{\gamma}\mid \gamma\in P^{(\la)}_+\}$ then by Claim 2 and the asymptotic Plancherel Theorem (Theorem~\ref{thm:AsymptoticPlancherel}) we have $\langle \mathfrak{f}_{w}(\zeta_{\la}),\fs_{\gamma}(\zeta_{\la})\rangle_{\la}^{\infty}=0$ for all $\gamma\in P^{(\la)}_+$, contradicting the fact that the Schur functions $\fs_{\gamma}(\zeta_{\la})$, $\gamma\in P^{(\la)}_+$, form a basis of $\ZZ[\zeta_{\la}]^{G_{\la}}$. Thus $w=\sm_{\gamma}$ for some $\gamma\in P^{(\la)}_+$. Claim 2 provides the reverse containment. 

It is now clear that the linear map $\fC_{\Ga_{\la}\cap\Ga_{\la}^{-1}}\to \ZZ[\zeta_{\la}]^{G_{\la}}$ with $\fc_{\la}(\sm_{\gamma})\mapsto \fs_{\gamma}(\zeta_{\la})$ is a ring isomorphism (see Theorem~\ref{thm:recognise2b}), and the proof of the theorem is complete. 
\end{proof}

\goodbreak

\begin{cor}\label{cor:aalg}
For each $\la\vdash n+1$ we have $\cJ_{\la}\cong \mathsf{Mat}_{N_{\la}}(\ZZ[\zeta_{\la}]^{G_{\la}})$. 
\end{cor}

\begin{proof}
This is immediate from Theorem~\ref{thm:leading0} and Xi's reduction result Theorem~\ref{thm:xi}. 
\end{proof}

We can now complete the proof of Lemma~\ref{lem:span}.

\begin{proof}[Proof of Lemma~\ref{lem:span}]
By Theorem~\ref{thm:leading0} we have 
$
f_{\la}(C_{\sm_{\gamma}})=\sq^{-2\ell(\sw_{\la'})}W_{\la'}(\sq^2)^2\fs_{\gamma}(\zeta_{\la}),
$
and the Schur functions form a basis of $\sR[\zeta_{\la}]^{G_{\la}}$. 
\end{proof}

\begin{remark}
The description of~$\cJ_{\Ga_{\la}\cap\Ga_{\la}^{-1}}$ in~\cite[Theorem~8.4.5]{Xi:02} is in terms of representation rings. To make the translation with our description given in Theorem~\ref{thm:leading0}, for each $\la\vdash n+1$ let $u_{\la}$ be a unipotent element of $\mathsf{SL}_{n+1}(\CC)$ with Jordan blocks given by the partition~$\la$. Let $F_{\la}$ be the maximal reductive subgroup of the centraliser in $\mathsf{SL}_{n+1}(\CC)$ of~$u_{\la}$. Then $\ZZ[\zeta_{\la}]^{G_{\la}}$ is isomorphic to the representation ring of~$F_{\la}$. 

We also note that Xi works with the ``affine symmetric group'' $\Waff\rtimes \ZZ$, and our extended affine Weyl group $\Wext$ is a quotient of this group. Thus Xi's results apply to $\mathsf{GL}_{n+1}(\CC)$. The setup of our paper can readily be adapted to work with the affine symmetric group and its associated extended affine Hecke algebra, however we have chosen to present our work directly in terms of the extended affine Weyl group $P\rtimes \Wfin$ (and hence $\mathsf{SL}_{n+1}(\CC)$) as this setting more naturally extends to other types (where an ``extended affine Weyl group'' is typically taken to be a group of the form $L\rtimes \Wfin$ with $L$ a lattice with $Q\leq L\leq P$).  
\end{remark}

\subsection{Concluding remarks and an example}\label{sec:63}

The distinguished involution $\sw_{\la'}$ has the property that $\fc_{\la}(\sw_{\la'})=E_{\su_{\la},\su_{\la}}$. We conjecture that all distinguished involutions behave similarly. Specifically, let $d_1,\ldots,d_{N_{\la}}$ denote the distinguished involutions of $\Delta_{\la}$, and let $\Ga_i$ be the right cell containing $d_i$. We make the following conjecture.

\begin{conjecture}\label{conj:Duflo}
There exists an ordering of the basis $\sB_{\la}$ of $M_{\la}$ such that the leading matrices of the distinguished involutions are given by $\fc_{\la}(d_i)=E_{ii}$
\end{conjecture}
 
 Example~\ref{ex:A34} below verifies this conjecture for $\tilde{\sA}_3$ with $\la=(2,2)$, and we have verified the conjecture for $\tilde{\sA}_3$ for all partitions $\la$. 

The following theorem shows that if Conjecture~\ref{conj:Duflo} holds, then our leading matrices explicitly realise the isomorphism from Xi's reduction theorem (Theorem~\ref{thm:xi}).

\begin{thm}\label{thm:fullalgebra}
Assume that Conjecture~\ref{conj:Duflo} holds. Then there exist weights $\ga_1,\ldots,\ga_{N_{\la}}\in P^{(\la)}$ and signs $\epsilon_1,\ldots,\epsilon_{N_{\la}}\in\{-1,1\}$ such that
\begin{align*}
D^{-1}\fC_{\la}D=\mathrm{Mat}_{N_{\la}}(\ZZ[\zeta_{\la}]^{G_{\la}})
\end{align*}
where $D=\mathrm{diag}(\epsilon_1\zeta_{\la}^{\ga_1},\ldots,\epsilon_{N_{\la}}\zeta_{\la}^{\ga_{N_{\la}}})$. Moreover, there is a function $\hh:\Delta_{\la}\to P^{(\la)}_+$ such that  
\begin{compactenum}[$(a)$]
\item for each $1\leq i,j\leq N_{\la}$ the map $\hh:\Ga_i\cap\Ga_j^{-1}\to P^{(\la)}_+$ is bijective, and 
\item if $w\in \Ga_i\cap\Ga_j^{-1}$ then $D^{-1}\fc_{\la}(w)D=\fs_{\hh(w)}(\zeta_{\la})E_{ij}$. 
\end{compactenum}
\end{thm}

\begin{proof} 
By assumption we have $\fc_{\la}(d_i)=E_{ii}$. A similar argument to Theorem~\ref{thm:indexing} now shows that if $w\in\Ga_i\cap\Ga_j^{-1}$ then the matrix $\fc_{\la}(w)$ has a unique nonzero entry, and this nonzero entry is in position $(i,j)$. Specifically, since $w\in \Ga_i$ we have $\st_{d_i}\st_w=\st_w$ and since $w\in \Ga_j^{-1}$ we have $\st_w\st_{d_j}=\st_w$, and hence by Theorem~\ref{thm:recognise2b} we have $\fc_{\la}(d_i)\fc_{\la}(w)=\fc_{\la}(w)=\fc_{\la}(w)\fc_{\la}(d_j)$, showing that all nonzero entries of $\fc_{\la}(w)$ are in both the $i$th row and the $j$th column.

Fix bijections $\phi_{ij}:\Gamma_i\cap\Gamma_j^{-1}\to\Ga_{\la}\cap\Ga_{\la}^{-1}$, as in \cite[\S2.3]{Xi:02}. Write $E_{ij}(a)$ for the matrix with~$a$ in the $(i,j)$-position and $0$ elsewhere. Xi's reduction theorem (\cite[Theorem~2.3.2]{Xi:02}) says that
\begin{compactenum}[$(1)$]
\item the map $\st_w\mapsto \st_{\phi_{ii}(w)}$ induces a ring isomorphism $\cJ_{\Ga_i\cap\Ga_i^{-1}}\to \cJ_{\Ga_{\la}\cap\Ga_{\la}^{-1}}$, and
\item the map $\st_w\mapsto E_{ij}(\st_{\phi_{ij}(w)})$, for $w\in \Ga_i\cap\Ga_j^{-1}$, defines a $\ZZ$-algebra  isomorphism from $\cJ_{\la}$ to the ring $\mathrm{Mat}_{N_{\la}}(\cJ_{\Ga_{\la}\cap\Ga_{\la}^{-1}})$ of $N_{\la}\times N_{\la}$ matrices with entries in $\cJ_{\Ga_{\la}\cap\Ga_{\la}^{-1}}$. 
\end{compactenum}

 Let $w_{ij}=\phi^{-1}_{ij}(\sw_{\la'})\in\Ga_i\cap\Ga_j^{-1}$. By \cite[Lemma~2.3.1]{Xi:02} we have $w_{ii}=d_i$ and $w_{ji}=w_{ij}^{-1}$. Since $E_{ij}(\st_{\sw_{\la'}})E_{ji}(\st_{\sw_{\la'}})=E_{ii}(\st_{\sw_{\la'}}^2)=E_{ii}(\st_{\sw_{\la'}})$ it follows (under the isomorphism in (2) above) that 
$
\st_{w_{ij}}\st_{w_{ij}^{-1}}=\st_{d_i}. 
$
By the first paragraph of the proof we have $\fc_{\la}(w_{ij})=a(\zeta_{\la})E_{ij}$ for some $a(\zeta_{\la})\in\ZZ[\zeta_{\la}]$. Then the above equation, and the isomorphism from Theorem~\ref{thm:recognise2b}, gives (using~(\ref{eq:conjugation}))
$
E_{ii}=\fc_{\la}(d_i)=\fc_{\la}(w_{ij})\fc_{\la}(w_{ij}^{-1})=a(\zeta_{\la})a(\zeta_{\la}^{-1})E_{ii},
$
and hence $a(\zeta_{\la})a(\zeta_{\la}^{-1})=1$. Writing $a(\zeta_{\la})=\sum_{\gamma}c_{\gamma}\zeta_{\la}^{\gamma}$ (with $c_{\gamma}\in\ZZ$ and the sum over $\gamma\in P^{(\la)}$) it follows that $\sum_{\gamma}c_{\gamma}^2=1$, and hence $a(\zeta_{\la})=\epsilon_{ij} \zeta_{\la}^{\gamma_{ij}}$ for some $\gamma_{ij}\in P^{(\la)}$ and $\epsilon_{ij}=\pm 1$. Thus the elements $w_{ij}$ satisfy
$
\fc_{\la}(w_{ij})=\epsilon_{ij}\zeta_{\la}^{\gamma_{ij}}E_{ij},
$
with $\epsilon_{ji}=\epsilon_{ij}$ and $\gamma_{ji}=-\gamma_{ij}$ (with $\epsilon_{ii}=1$). 

We claim that, for all $i,j,k$,
\begin{align}\label{eq:triangle}
\gamma_{ij}+\gamma_{jk}+\gamma_{ki}=0\quad \text{and}\quad \epsilon_{ij}\epsilon_{jk}\epsilon_{ki}=1.
\end{align}
To see this, it follows from (2) above that $\st_{w_{ij}}\st_{w_{jk}}\st_{w_{ki}}=\st_{d_i}$. On the other hand,
$$
\fc_{\la}(w_{ij})\fc_{\la}(w_{jk})\fc_{\la}(w_{ki})=\epsilon_{ij}\epsilon_{jk}\epsilon_{ki}\zeta_{\la}^{\gamma_{ij}+\gamma_{jk}+\gamma_{ki}}E_{ii},
$$
and since $\fc_{\la}(d_i)=E_{ii}$ the claim follows (using Theorem~\ref{thm:recognise2b}). 

It is now elementary, using~(\ref{eq:triangle}), that one may choose weights $\ga_1,\ldots,\ga_{N_{\la}}\in P^{(\la)}$ and signs $\epsilon_1,\ldots,\epsilon_{N_{\la}}\in\{-1,1\}$  such that $\epsilon_{ij}=\epsilon_i\epsilon_j$ and $\zeta_{\la}^{\ga_{ij}}=\zeta_{\la}^{\ga_i}\zeta_{\la}^{-\ga_j}$. This is a consequence of the following: if $z_{ij}$, $1\leq i,j\leq n$, are invertible satisfying $z_{ij}z_{jk}=z_{ik}$ for all $1\leq i,j,k\leq n$ then there exist invertible $x_1,\ldots,x_n$ such that $x_ix_j^{-1}=z_{ij}$ for all $1\leq i,j\leq n$. The proof is induction on $n$ with the induction step given by setting $x_{n+1}=x_1z_{1,n+1}^{-1}$. It follows that
\begin{align*}
D^{-1}\fc_{\la}(\sw_{ij})D=E_{ij}\quad\text{where}\quad D=\mathrm{diag}(\epsilon_1\zeta_{\la}^{\ga_1},\ldots,\epsilon_{N_{\la}}\zeta_{\la}^{\ga_{N_{\la}}}).
\end{align*}

Let $\cR_{ij}=\{[D^{-1}AD]_{ij}\mid A\in\fC_{\la}\}$. We claim that $\cR_{ij}=\ZZ[\zeta_{\la}]^{G_{\la}}$. To see this, suppose that $a(\zeta_{\la}),b(\zeta_{\la})\in \cR_{ij}$. Then there exist $A,B\in \fC_{\la}$ such that $[D^{-1}AD]_{ij}=a(\zeta_{\la})$ and $[D^{-1}BD]_{ij}=b(\zeta_{\la})$. Then
\begin{align*}
D^{-1}\fc_{\la}(d_i)A\fc_{\la}(\sw_{ji})B\fc_{\la}(d_j)D&=D^{-1}E_{ii}AD(D^{-1}\fc_{\la}(\sw_{ji})D)D^{-1}BE_{jj}D\\
&=E_{ii}D^{-1}ADE_{ji}D^{-1}BDE_{jj}\\
&=a(\zeta_{\la})b(\zeta_{\la})E_{ij},
\end{align*}
and so $\cR_{ij}$ is a ring. Moreover
$$
D^{-1}\fc_{\la}(\sw_{1i})A\fc_{\la}(\sw_{j1})D=D^{-1}\fc_{\la}(\sw_{1i})D(D^{-1}AD)D^{-1}\fc_{\la}(\sw_{j1})D=a(\zeta_{\la})E_{11},
$$ 
and so $\fc_{\la}(\sw_{1i})A\fc_{\la}(\sw_{j1})=a(\zeta_{\la})E_{11}$, and hence $a(\zeta_{\la})\in\ZZ[\zeta_{\la}]^{G_{\la}}$. So $\cR_{ij}\subseteq \ZZ[\zeta_{\la}]^{G_{\la}}$. On the other hand,
$
D^{-1}\fc_{\la}(w_{i1})\fc_{\la}(\sm_{\gamma})\fc_{\la}(w_{1j})D=\fs_{\gamma}(\zeta_{\la})E_{ij},
$
and so $\ZZ[\zeta_{\la}]^{G_{\la}}\subseteq\cR_{ij}$, completing the proof of the claim. Thus $D^{-1}\fC_{\la}D=\mathrm{Mat}_{N_{\la}}(\ZZ[\zeta_{\la}]^{G_{\la}})$.

Let $w\in \Ga_i\cap \Ga_j^{-1}$, and write $D^{-1}\fc_{\la}(w)D=a(\zeta_{\la})E_{ij}$ for some $a(\zeta_{\la})\in \ZZ[\zeta_{\la}]^{G_{\la}}$. An argument similar to the proof of Claim 2 in Theorem~\ref{thm:leading0} gives that $a(\zeta_{\la})=\epsilon_w\fs_{\hh(w)}(\zeta_{\la})$ for some $\epsilon_w\in\{-1,1\}$ and $\hh(w)\in P^{(\la)}_+$. The map $\hh:\Ga_i\cap \Ga_j^{-1}\to P^{(\la)}_+$ is injective (using Theorem~\ref{thm:AsymptoticPlancherel}) and surjective (as $\cR_{ij}=\ZZ[\zeta_{\la}]^{G_{\la}}$). We claim that $\epsilon_w=1$. We have
$$
D^{-1}\fc_{\la}(\sw_{1i})\fc_{\la}(w)\fc_{\la}(\sw_{j1})D=\epsilon_w\fs_{\hh(w)}(\zeta_{\la})E_{11}=\epsilon_wD^{-1}\fc_{\la}(\sm_{\hh(w)})D,
$$
and so $\fc_{\la}(\sw_{1i})\fc_{\la}(w)\fc_{\la}(\sw_{j1})=\epsilon_w\fc_{\la}(\sm_{\hh(w)})$. Thus
$
\st_{\sw_{1i}}\st_{w}\st_{\sw_{j1}}=\epsilon_w\st_{\sm_{\hh(w)}}. 
$
It is known that the structure constants of $\cJ_{\la}$ are nonnegative (see \cite[(3.1.1)]{Lus:85}), and hence $\epsilon_w=1$. Thus we have $D^{-1}\fc_{\la}(w)D=\fs_{\hh(w)}(\zeta_{\la})E_{ij}$ for all $w\in\Ga_i\cap\Ga_j^{-1}$.
\end{proof}

\begin{remark}
Conjugation by $D$ in the above theorem amounts to choosing a basis associated to a (signed) ``fundamental domain'' for the action of $\sT_{\la}$ on $\cA_{\la}$ different from the standard fundamental domain~${^\la}W$. See \cite[Section~5.5]{GLP:23} for details. Specifically, the basis chosen is $\{\xi_{\la}\otimes \epsilon_{i}X_{\tau_{\gamma_i}^{-1}u_{i}}\mid 1\leq i\leq N_{\la}\}$ where ${^\la}W=\{u_i\mid 1\leq i\leq N_{\la}\}$ is ordered as in Theorem~\ref{thm:indexing}. 
\end{remark}

We make the following conjecture. 

\begin{conjecture}\label{conj:signs}
In Theorem~\ref{thm:fullalgebra} we have $\epsilon_i=1$ for all $1\leq i\leq N_{\la}$. 
\end{conjecture}

\begin{example}\label{ex:A34}
Consider $\sA_3$ with $\la=(2,2)$ (see Examples~\ref{ex:A3}, \ref{ex:A32}, and~\ref{ex:A33}). We have $\su_{\la}=s_2$, and 
$
{^\la}W=\{s_2,e,s_2s_1,s_2s_3,s_2s_1s_3,s_2s_1s_3s_2\}
$ (ordered with $\su_{\la}$ being the first element). Writing $\mathsf{Q}=\sq-\sq^{-1}$, the path formula in Theorem~\ref{thm:pathformula} gives
\begin{align*}
\pi_{\la}(T_1)&=\scriptsize{\begin{bmatrix}
\mathsf{Q}&0&1&0&0&0\\
0&-\sq^{-1}&0&0&0&0\\
1&0&0&0&0&0\\
0&0&0&\mathsf{Q}&1&0\\
0&0&0&1&0&0\\
0&0&0&0&0&-\sq^{-1}
\end{bmatrix}}
&
\pi_{\la}(T_2)&=\scriptsize{\begin{bmatrix}
0&1&0&0&0&0\\
1&\mathsf{Q}&0&0&0&0\\
0&0&-\sq^{-1}&0&0&0\\
0&0&0&-\sq^{-1}&0&0\\
0&0&0&0&\mathsf{Q}&1\\
0&0&0&0&1&0
\end{bmatrix}}\\
\scriptsize{\pi_{\la}(T_3)}&=\scriptsize{\begin{bmatrix}
\mathsf{Q}&0&0&1&0&0\\
0&-\sq^{-1}&0&0&0&0\\
0&0&\mathsf{Q}&0&1&0\\
1&0&0&0&0&0\\
0&0&1&0&0&0\\
0&0&0&0&0&-\sq^{-1}
\end{bmatrix}}
&
\pi_{\la}(T_0)&=\scriptsize{\begin{bmatrix}
0&0&0&0&0&z_1/z_2\\
0&0&0&0&z_1/z_2&0\\
0&0&-\sq^{-1}&0&0&0\\
0&0&0&-\sq^{-1}&0&0\\
0&z_2/z_1&0&0&\mathsf{Q}&0\\
z_2/z_1&0&0&0&0&\mathsf{Q}
\end{bmatrix}}\\
\pi_{\la}(T_{\sigma})&=\scriptsize{\begin{bmatrix}
0&0&0&0&z_1&0\\
0&0&0&z_1&0&0\\
0&z_2&0&0&0&0\\
0&0&0&0&0&z_1\\
z_2&0&0&0&0&0\\
0&0&z_2&0&0&0
\end{bmatrix}}
\end{align*}
The distinguished involutions are
\begin{align*}
d_1&=s_1s_3&d_2&=s_2s_1s_3s_2&d_3&=s_3s_2s_0s_3&
d_4&=s_1s_2s_0s_1&d_5&=s_0s_2&d_6&=s_0s_1s_3s_0,
\end{align*}
and we compute $\fc_{\la}(d_i)=E_{ii}$ (verifying Conjecture~\ref{conj:Duflo} in this case). After making choices of the bijections $\phi_{ij}:\Ga_i\cap\Ga_j^{-1}\to\Ga_1\cap\Ga_1^{-1}$ the elements $w_{ij}=\phi_{ij}^{-1}(\sw_{\la'})$ (see the proof of Theorem~\ref{thm:fullalgebra}) are
\begin{align*}
w_{12}&=s_1s_3s_2&w_{13}&=s_1s_3s_0\sigma&w_{14}&=s_1s_3s_2\sigma&w_{15}&=s_1s_3\sigma&w_{16}&=s_1s_3s_2\sigma^2\\
w_{23}&=s_2s_1s_3s_0\sigma&w_{24}&=s_2s_1s_3s_2\sigma&w_{25}&=s_2s_1s_3\sigma&w_{26}&=s_2s_1s_3s_2\sigma^2&w_{34}&=s_3s_2s_0s_1\\
w_{35}&=s_3s_2s_0&w_{36}&=s_3s_2s_0s_1\sigma&w_{45}&=s_1s_2s_0&w_{46}&=s_1s_2s_0s_1\sigma&w_{56}&=s_2s_0s_1\sigma
\end{align*}
(with $w_{ii}=d_i$ and $w_{ji}=w_{ij}^{-1}$). We compute 
\begin{align*}
\fc_{\la}(w_{12})&=E_{12},&\fc_{\la}(w_{13})&=z_1E_{13},&\fc_{\la}(w_{14})&=z_1E_{14},&\fc_{\la}(w_{15})&=z_1E_{15},&\fc_{\la}(w_{16})&=z_1^2E_{16}.
\end{align*}
The conjugating matrix (from Theorem~\ref{thm:fullalgebra}) may thus be taken as
$$
D=\mathrm{diag}(1,1,z_1^{-1},z_1^{-1},z_1^{-1},z_1^{-2}).
$$
Then $D^{-1}\fc_{\la}(w_{ij})D=E_{ij}$ for all $1\leq i,j\leq 6$ (note that all signs $\epsilon_i$ are $+1$, see Conjecture~\ref{conj:signs}), and we have $D^{-1}\fC_{\la}D=\mathrm{Mat}_6(\ZZ[\zeta_{\la}]^{G_{\la}})$. For example, if $w=s_2s_0s_1s_3s_2s_0s_1s_3s_2s_0\sigma^2$ we compute
$$
D^{-1}\fc_{\la}(w)D=(z_1^5z_2+z_1z_2^5+z_1^2+z_2^2+z_1z_2)E_{56}=\fs_{5\tilde{e}_1+\tilde{e}_2}(\zeta_{\la})E_{56},
$$
showing that $w\in \Ga_5\cap\Ga_6^{-1}$ and $h_{\la}(w)=3e_1+2e_2+e_3\in P^{(\la)}_+$. 
\end{example}

\begin{appendix}

\section{Proof of Theorem~\ref{thm:mgamma}}\label{app:proofs}

In this appendix we prove Theorem~\ref{thm:mgamma} (see Theorems~\ref{thm:doublecoset} and~\ref{thm:domlength}). In Section~\ref{sec:double} we give general analysis of maximal length double coset representatives in affine Weyl groups, before turning attention to the specific elements $\sm_{\ga}$ in Section~\ref{sec:double2} where we show that $\ell(\sm_{\ga})=\ell(\su_{\la}^{-1}\tau_{\ga}\su_{\la})+\ell(\sw_{\la'})$. In Section~\ref{sec:length} we compute $\ell(\su_{\la}^{-1}\tau_{\ga}\su_{\la})$, and in Section~\ref{sec:monotone} we use this length formula to prove monotonicity of $\ell(\sm_{\ga})$ with respect to $\peq_{\la}$.

\subsection{Maximal length double coset representatives in affine Weyl groups}\label{sec:double}

This section gives some general facts about maximal length double coset representatives in affine Weyl groups. 

\begin{lemma}\label{lem:characterise} 
Let $\gamma\in P$ and $u\in \Wfin$. For $1\leq j\leq n$ we have:
\begin{compactenum}[$(1)$]
\item $\ell(s_jt_{\gamma}u)=\ell(t_{\gamma}u)-1$ if and only if $\langle \gamma,\alpha_j\rangle\leq 0$ and if $\langle\gamma,\alpha_j\rangle=0$ then $u^{-1}\alpha_j<0$.
\item $\ell(t_{\gamma}us_j)=\ell(t_{\gamma}u)-1$ if and only if $\langle u^{-1}\gamma,\alpha_j\rangle\geq 0$ and if $\langle u^{-1}\gamma,\alpha_j\rangle=0$ then $u\alpha_j<0$. 
\end{compactenum}
\end{lemma}

\begin{proof}
It is sufficient to prove (1), because (2) follows by applying (1) to $(t_{\gamma}u)^{-1}=t_{-u^{-1}\gamma}u^{-1}$. By counting hyperplanes separating $\cA_0$ from $t_{\gamma}u\cA_0$, we have (see \cite[(2.4.1)]{Mac:03})
$$
\ell(t_{\gamma}u)=\sum_{\alpha\in\Phi^+}|\langle\gamma,\alpha\rangle-\chi^-(u^{-1}\alpha)|,
$$
where $\chi^-(\alpha)=1$ if $\alpha\in-\Phi^+$, and $0$ otherwise. The simple reflection $s_j$ permutes $\Phi^+\backslash\{\alpha_j\}$, and since $s_jt_{\gamma}u=t_{s_j\gamma}s_ju$ it follows that
\begin{align*}
\ell(s_jt_{\gamma}u)&=|\langle \gamma,\alpha_j\rangle+\chi^-(-u^{-1}\alpha_j)|+\sum_{\alpha\in\Phi^+\backslash\{\alpha_j\}}|\langle\gamma,\alpha\rangle-\chi^-(u^{-1}\alpha)|,
\end{align*}
and so $\ell(t_{\gamma}u)-\ell(s_jt_{\gamma}u)=|\langle\gamma,\alpha_j\rangle-\chi^-(u^{-1}\alpha_j)|-|\langle\gamma,\alpha_j\rangle+\chi^-(-u^{-1}\alpha_j)|$. Thus $\ell(s_jt_{\gamma}u)=\ell(t_{\gamma}u)-1$ if and only if either $\langle\gamma,\alpha_j\rangle<0$, or $\langle\gamma,\alpha_j\rangle=0$ and $u^{-1}\alpha_j<0$. 
\end{proof}

\begin{lemma}\label{lem:maxlength}
Let $\gamma\in P$, $u\in \Wfin$, and $J\subseteq \{1,2,\ldots,n\}$. Let
\begin{align*}
L_J(\gamma,u)&=\{\alpha\in\Phi_{J}^+\mid \langle \gamma,\alpha\rangle> 0\text{ or $\langle \gamma,\alpha\rangle=0$ and $u^{-1}\alpha>0$}\}\\
R_J(\gamma,u)&=\{\alpha\in\Phi_{J}^+\mid \langle u^{-1}\gamma,\alpha\rangle<0\text{ or $\langle u^{-1}\gamma,\alpha\rangle=0$ and $u\alpha>0$}\}.
\end{align*}
Then there exist (unique) elements $x,x'\in W_J$ with $\Phi(x)=L_J(\gamma,u)$ and $\Phi(x')=R_J(\gamma,u)$. Moreover $x^{-1}t_{\gamma}u$ is of maximal length in $W_Jt_{\gamma}u$, and $t_{\gamma}ux'$ is of maximal length in $t_{\gamma}uW_J$, and $\ell(x^{-1} t_{\gamma}u)=\ell(x)+\ell(t_{\gamma}u)$ and $\ell(t_{\gamma}ux')=\ell(t_{\gamma}u)+\ell(x')$.
\end{lemma}

\begin{proof}
It is sufficient to prove the results concerning the coset $W_Jt_{\gamma}u$ (for the other statements follow by considering $(t_{\gamma}u)^{-1}=t_{-u^{-1}\gamma}u^{-1}$). We argue by induction on $|L_J(\gamma,u)|$. If $|L_J(\gamma,u)|=0$ then Lemma~\ref{lem:characterise} gives $\ell(s_jt_{\gamma}u)=\ell(t_{\gamma}u)-1$ for all $j\in J$. Thus $t_{\gamma}u$ is of maximal length in $W_Jt_{\gamma}u$, and the result follows with $x=e$. Suppose $|L_J(\gamma,u)|>0$. Then there is $j\in J$ with $\alpha_j\in L_J(\gamma,u)$ (otherwise $L_J(\gamma,u)=\emptyset$) and so $\ell(s_jt_{\gamma}u)=\ell(t_{\gamma}u)+1$ by Lemma~\ref{lem:characterise}. We have $s_jt_{\gamma}u=t_{s_j\gamma}(s_ju)$, and we claim that 
\begin{align}\label{eq:evolve}
L_J(s_j\gamma,s_ju)=s_j(L_J(\gamma,u)\backslash\{\alpha_j\}). 
\end{align}
Note first that $\alpha_j\notin L_J(s_j\gamma,s_ju)$ (for otherwise either $\langle s_j\gamma,\alpha_j\rangle>0$ or $\langle s_j\gamma,\alpha_j\rangle=0$ and $u^{-1}s_j\alpha_j>0$, and so either $\langle \gamma,\alpha_j\rangle<0$ or $\langle \gamma,\alpha_j\rangle=0$ and $u^{-1}\alpha_j<0$, contradicting the fact that $\alpha_j\in L_J(\gamma,u)$). Now, to prove~(\ref{eq:evolve}), suppose that $\alpha\in L_J(s_j\gamma,s_ju)$. Then either $\langle \gamma,s_j\alpha\rangle>0$ or $\langle \gamma,s_j\alpha\rangle=0$ and $u^{-1}s_j\alpha>0$. Since $\alpha\in\Phi^+\backslash \{\alpha_j\}$, and since $s_j$ permutes this set of positive roots, it follows that $s_j\alpha\in L_J(\gamma,u)$ and so $\alpha\in s_j(L_J(\gamma,u)\backslash\{\alpha_j\})$. The converse is similar. 

It follows by induction using~(\ref{eq:evolve}) that there exist $j_1,\ldots,j_k\in J$ such that, writing $x^{-1}=s_{j_1}\cdots s_{j_k}$, we have
$
\ell(x^{-1}t_{\gamma}u)=k+\ell(t_{\gamma}u),
$
and such that $L_J(x\gamma,xu)=\emptyset$. By Lemma~\ref{lem:characterise} we have that $x^{-1}t_{\gamma}u$ is of maximal length in $W_Jt_{\gamma}u$, and it follows from~(\ref{eq:evolve}) that $\Phi(x)=L_J(\gamma,u)$. 
\end{proof}

\begin{cor}\label{cor:doublecoset1}
Let $\gamma\in P$, $u\in \Wfin$, and $J\subseteq \{1,2,\ldots,n\}$. Let $m$ be the longest element of $W_Jt_{\gamma}uW_J$. Then 
$$
m=x^{-1}t_{\gamma}uy\quad\text{with}\quad \ell(m)=\ell(x)+\ell(t_{\gamma}u)+\ell(y),
$$
where $\Phi(x)=L_J(\gamma,u)$ and $\Phi(y)=R_J(x^{-1}\gamma,x^{-1}u)$. In particular, 
$$
\ell(m)-\ell(t_{\gamma}u)=|L_J(\gamma,u)|+|R_J(x^{-1}\gamma,x^{-1}u)|.
$$
\end{cor}

\begin{proof}
By Lemma~\ref{lem:maxlength} we have that $x^{-1}t_{\gamma}u=t_{x^{-1}\gamma}(x^{-1}u)$ is of maximal length in $W_Jt_{\gamma}u$, where $\Phi(x)=L_J(\gamma,u)$, and that $\ell(x^{-1}t_{\gamma}u)=\ell(x)+\ell(t_{\gamma}u)$. Then, again using Lemma~\ref{lem:maxlength}, we have that $(x^{-1}t_{\gamma}u)y$ is of maximal length in $x^{-1}t_{\gamma}uW_J$, where $\Phi(y)=R_J(x^{-1}\gamma,x^{-1}u)$, and that 
$$
\ell((x^{-1}t_{\gamma}u)y)=\ell(x^{-1}t_{\gamma}u)+\ell(y)=\ell(x)+\ell(t_{\gamma}u)+\ell(y).
$$
Thus $x^{-1}t_{\gamma}uy$ is of maximal length in $W_Jt_{\gamma}uW_J$, and the result follows. 
\end{proof}

\subsection{The elements $\sm_{\gamma}$}\label{sec:double2}

In this section we prove that $\ell(\sm_{\ga})=\ell(\su_{\la}^{-1}\tau_{\ga}\su_{\la})+\ell(\sw_{\la'})$ (see Theorem~\ref{thm:doublecoset}). The length $\ell(\su_{\la}^{-1}\tau_{\ga}\su_{\la})$ is computed in the next section (see Theorem~\ref{thm:lengthconj}).

\begin{prop}\label{prop:doublecoset2}
Let $\gamma\in P^{(\la)}$. We have 
$
\ell(\sm_{\gamma})-\ell(\su_{\la}^{-1}\tau_{\gamma}\su_{\la})=|L(\gamma)|+|R(\gamma)|,
$
where
\begin{align*}
L(\gamma)&=\{\beta\in\su_{\la}\Phi_{\la'}^+\mid \langle \gamma,\beta\rangle>0\text{ or $\langle \gamma,\beta\rangle=0$ and $\su_{\la}^{-1}\sy_{\gamma}^{-1}\beta>0$}\},\\
R(\gamma)&=\{\beta\in\su_{\la}\Phi_{\la'}^+\mid \langle\sw_{\la}\gamma,\beta\rangle<0\text{ or } \langle\sw_{\la}\gamma,\beta\rangle=0\text{ and }\su_{\la}^{-1}\sy_{\gamma}\beta>0\text{ and }\su_{\la}^{-1}\sy_{\gamma}\beta\notin\Phi_{\la'}^+\}.
\end{align*} 
\end{prop}

\begin{proof}
Since $\su_{\la}^{-1}\tau_{\ga}\su_{\la}=t_{\su_{\la}^{-1}\ga}\su_{\la}^{-1}\sy_{\ga}\su$ Corollary~\ref{cor:doublecoset1} gives $\ell(\sm_{\gamma})-\ell(\su_{\la}^{-1}\tau_{\gamma}\su_{\la})=|L(\gamma)|+|R'(\gamma)|$ with $L(\ga)$ as given in the statement of the proposition, and $R'(\ga)$ given by 
$$
R'(\ga)=\{\beta\in\su_{\la}\Phi_{\la'}^+\mid \langle \sy_{\gamma}^{-1}\gamma,\beta\rangle<0\text{ or $\langle \sy_{\gamma}^{-1}\gamma,\beta\rangle=0$ and $x^{-1}\su_{\la}^{-1}\sy_{\gamma}\beta>0$}\}
$$
where $x\in\Wfin$ is defined by $\Phi(x)=\su_{\la}^{-1}L(\gamma)$. It remains to show that $R'(\ga)=R(\ga)$. By~(\ref{eq:defny}) we have $\sy_{\gamma}^{-1}=\sw_{\la}\sw_{J_{\la}\backslash J_{\la}(\gamma)}$. Moreover, since $\ga\in P^{(\la)}$ and $J_{\la}(\ga)=\{j\in J_{\la}\mid \langle\ga,\alpha_j\rangle=1\}$ we have $\langle\ga,\alpha_j\rangle=0$ for all $j\in J_{\la}\backslash J_{\la}(\gamma)$, and hence $\sy_{\ga}^{-1}\ga=\sw_{\la}\ga$. 

To conclude the proof we must show that if $\langle\sw_{\la}\gamma,\beta\rangle=0$, then $x^{-1}\su_{\la}^{-1}\sy_{\ga}\beta>0$ if and only if $\su_{\la}^{-1}\sy_{\ga}\beta>0$ and $\su_{\la}^{-1}\sy_{\ga}\beta\notin\Phi_{\la'}^+$. Suppose first that $x^{-1}\su_{\la}^{-1}\sy_{\ga}\beta>0$. If $\su_{\la}^{-1}\sy_{\ga}\beta<0$ then $-\su_{\la}^{-1}\sy_{\ga}\beta\in\Phi(x)=\su_{\la}^{-1}L(\ga)$, and so $\beta'=-\sy_{\ga}\beta\in L(\ga)$. Since $\langle\ga,\beta'\rangle=-\langle\sw_{\la}\ga,\beta\rangle=0$ the definition of $L(\ga)$ forces $\su_{\la}^{-1}\sy_{\ga}^{-1}\beta'>0$, however $\su_{\la}^{-1}\sy_{\ga}^{-1}\beta'=-\su_{\la}^{-1}\beta\in -\Phi_{\la'}^+$, a contradiction. Thus $\su_{\la}^{-1}\sy_{\ga}\beta>0$, and so $\sy_{\ga}\beta\notin L(\ga)$ (because $\su_{\la}^{-1}\sy_{\ga}\beta\notin\Phi(x)$). If $\su_{\la}^{-1}\sy_{\ga}\beta\in\Phi_{\la'}^+$ then $\beta'=\sy_{\ga}\beta\in\su_{\la}\Phi_{\la'}^+$ and since $\langle\ga,\beta'\rangle=\langle\sw_{\la}\ga,\beta\rangle=0$ and $\su_{\la}^{-1}\sy_{\ga}^{-1}\beta'=\su_{\la}^{-1}\beta>0$ we have $\beta'=\sy_{\ga}\beta\in L(\ga)$, a contradiction. Conversely, suppose that $\langle\sw_{\la}\ga,\beta\rangle=0$ with $\su_{\la}^{-1}\sy_{\ga}\beta>0$ and $\su_{\la}^{-1}\sy_{\ga}\beta\notin\Phi_{\la'}^+$. If $x^{-1}\su_{\la}^{-1}\sy_{\ga}\beta<0$ then $\su_{\la}^{-1}\sy_{\ga}\beta\in \Phi(x)=\su_{\la}^{-1}L(\ga)$ and so $\sy_{\ga}\beta\in L(x)\subseteq \su_{\la}\Phi_{\la'}^+$, contradicting $\su_{\la}^{-1}\sy_{\ga}\beta\notin\Phi_{\la'}^+$ and concluding the proof. 
\end{proof}

Recall that $\la[k,i]$ denotes the entry in the $k$th row and $i$th column of $\bt_{r}(\la)$. 

\begin{lemma}\label{lem:imageroots}
We have
$$
\su_{\la}\Phi^+=\{e_{\la[k,i]}-e_{\la[l,j]}\mid 0\leq k,l\leq r(\la),\,\,\,1\leq i\leq j\leq \la_l,\,\text{ and if $i=j$ then $k<l$}\}.
$$
\end{lemma}

\begin{proof}
It follows from the definition of $\su_{\la}$ that $\su_{\la}\Phi^+$ consists of the roots $e_{i_1}-e_{i_2}$ such that $i_1$ occurs before $i_2$ in the column reading of $\bt_r(\la)$, and hence the result. 
\end{proof}

\begin{lemma}\label{lem:poscriteria}
Let $\beta=e_{\la[k,i]}-e_{\la[l,j]}$ with $1\leq k<l\leq r(\la)$, $1\leq i\leq \la_k$, and $1\leq j\leq \la_l$. Then
\begin{compactenum}[$(1)$]
\item $\su_{\la}^{-1}\beta>0$ if and only if $i\leq j$.
\item $\su_{\la}^{-1}\beta\in \Phi_{\la'}^+$ if and only if $i=j$. 
\end{compactenum}
\end{lemma}

\begin{proof}
Again this follows directly from the definition of $\su_{\la}$.
\end{proof}

In particular, we have 
$$
\su_{\la}\Phi_{\la'}^+=\{e_{\la[k,i]}-e_{\la[l,i]}\mid 1\leq k<l\leq r(\la),\,1\leq i\leq \la_{l}\}.
$$
It is convenient to decompose $\su_{\la}\Phi_{\la'}^+$ as follows:
\begin{align}\label{eq:decomposeu}
\su_{\la}\Phi_{\la'}^+=\bigsqcup_{1\leq k<l\leq r(\la)}A_{k,l}\quad\text{where}\quad A_{k,l}=\{e_{\la(k-1)+i}-e_{\la(l-1)+i}\mid 1\leq i\leq \la_{l}\}.
\end{align}

For the next three results we will fix $1\leq k<l\leq r(\la)$ and $1\leq i\leq \la_l$, and set 
\begin{compactenum}[$(1)$]
\item $\beta=e_{\la[k,i]}-e_{[l,i]}\in A_{k,l}$,
\item $\gamma\in P^{(\la)}$ with $\gamma+Q_{\la}=\sum_{j=1}^{r(\la)}a_j\tilde{e}_j$, and $a_j=\la_jb_j+c_j$ with $0\leq c_j<\la_j$ for $1\leq j\leq r(\la)$,
\item $c_j^*=\la_j-c_j$ for $1\leq j\leq r(\la)$. 
\end{compactenum}

The following lemmas are helpful in determining $L(\ga)$ and $R(\gamma)$. 

\begin{lemma}\label{lem:conditions1}
Let $\beta$ and $\gamma$ be as above. We have $\su_{\la}^{-1}\sy_{\gamma}^{-1}\beta>0$ if and only if either:
\begin{compactenum}[$(a)$]
\item $i\leq c_k$ and $i\leq c_l$ with $c_k^*\leq c_l^*$, or
\item $i>c_k$ and $i\leq c_l$, or 
\item $i>c_k$ and $i>c_l$ with $c_l\leq c_k$. 
\end{compactenum}
\end{lemma}

\begin{proof}
Recall that $\sy_{\ga}=\sw_{J_{\la}\backslash J_{\la}(\ga)}\sw_{\la}$. A direct calculation gives $\sy_{\ga}^{-1}\beta=e_{\la[k,i']}-e_{\la[l,j']}$ where
\begin{align*}
i'&=\begin{cases}
c_k^*+i&\text{if $i\leq c_k$}\\
i-c_k&\text{if $i>c_k$}
\end{cases}&
j'&=\begin{cases}
c_l^*+i&\text{if $i\leq c_l$}\\
i-c_l&\text{if $i>c_l$}.
\end{cases}
\end{align*}
Now use that fact that $\su_{\la}^{-1}\sy_{\ga}^{-1}\beta>0$ if and only if $i'\leq j'$ (see Lemma~\ref{lem:poscriteria}). 
\end{proof}

\begin{lemma}\label{lem:conditions2}
Let $\beta$ and $\gamma$ be as above. We have $\su_{\la}^{-1}\sy_{\gamma}\beta>0$ with $\su_{\la}^{-1}\sy_{\gamma}\beta\notin \Phi_{\la'}^+$ if and only if either:
\begin{compactenum}[$(a)$]
\item $i\leq c_k^*$ and $i\leq c_l^*$ with $c_k< c_l$, or 
\item $i>c_k^*$ and $i\leq c_l^*$, or
\item $i>\la_k^*$ and $i>\la_l^*$ with $c_l^*< c_k^*$.
\end{compactenum}
\end{lemma}

\begin{proof}
Similarly to in Lemma~\ref{lem:conditions1}, a direct calculation gives $\sy_{\ga}\beta=e_{\la[k,i']}-e_{\la[l,j']}$ where
\begin{align*}
i'&=\begin{cases}
i-c_k^*&\text{if $i>c_k^*$}\\
c_k+i&\text{if $i\leq c_k^*$}
\end{cases}&
j'&=\begin{cases}
i-c_l^*&\text{if $i>c_l^*$}\\
c_l+i&\text{if $i\leq c_l^*$.}
\end{cases}
\end{align*}
The result follows from the fact that $\su_{\la}^{-1}\sy_{\ga}\beta>0$ if and only if $i'\leq j'$, and if $i'=j'$ then $\su_{\la}^{-1}\sy_{\ga}\beta\in \Phi_{\la'}^+$ (see Lemma~\ref{lem:poscriteria}). 
\end{proof}

In the following proposition we determine $L(\ga)$ and $R(\ga)$. 

\begin{prop}\label{prop:classifyLandR}
Let $\beta$ and $\ga$ be as above. We have $\beta\in L(\ga)$ if and only if either $b_k-b_l\geq 1$, or $b_k-b_l=0$ with $c_k\geq c_l$ and either:
\begin{compactenum}[$(a)$]
\item $c_k^*>c_l^*$ with $c_l<i\leq \la_l$, or
\item $c_k^*\leq c_l^*$ with $1\leq i\leq \la_l$.
\end{compactenum}
We have $\beta\in R(\ga)$ if and only if either $b_k-b_l\leq -1$, or $b_k-b_l=0$ with $c_k^*>c_l^*$ and either:
\begin{compactenum}[$(a)$]
\item $c_k\geq c_l$ with $c_l^*<i\leq \la_l$, or
\item $c_k<c_l$ with $1\leq i\leq \la_l$.
\end{compactenum}
\end{prop}

\begin{proof}
By Proposition~\ref{prop:bijectionP} we have
\begin{align*}
\langle\gamma,\beta\rangle&=\begin{cases}
b_k-b_l&\text{if $i\leq c_k$ and $i\leq c_l$, or $i>c_k$ and $i>c_l$}\\
b_k-b_l+1&\text{if $i\leq c_k$ and $i>c_l$}\\
b_k-b_l-1&\text{if $i>c_k$ and $i\leq c_l$,}
\end{cases}
\end{align*}
and it follows that:
\begin{compactenum}
\item[\raisebox{0.35ex}{\tiny$\bullet$}] If $b_k-b_l\leq -2$ then $\langle\ga,\beta\rangle<0$ and so $\beta\notin L(\ga)$. 
\item[\raisebox{0.35ex}{\tiny$\bullet$}] If $b_k-b_l=-1$ then $\beta\notin L(\ga)$ unless $i\leq c_k$ and $i>c_l$ and $\su_{\la}^{-1}\sy_{\ga}^{-1}\beta>0$. However by Lemma~\ref{lem:conditions1} if $i\leq c_k$ and $i>c_l$ then $\su_{\la}^{-1}\sy_{\ga}^{-1}\beta<0$, and so $\beta\notin L(\ga)$ whenever $b_k-b_l=-1$. 
\item[\raisebox{0.35ex}{\tiny$\bullet$}] If $b_k-b_l=0$ then we have the following:
\begin{compactenum}[$(a)$]
\item If $i\leq c_k$ and $i>c_l$ then $\beta\in L(\ga)$. 
\item If $i>c_k$ and $i\leq c_l$ then $\beta\notin L(\ga)$. 
\item If $i\leq c_k$ and $i\leq c_l$ we have $\langle\ga,\beta\rangle=0$, and using Lemma~\ref{lem:conditions1}, we have $\beta\in L(\ga)$ if and only if $c_k^*\leq c_l^*$.  
\item Similarly, if $i>c_k$ and $i>c_l$ we have $\beta\in L(\ga)$ if and only if $c_l\leq c_k$. 
\end{compactenum} 
\item[\raisebox{0.35ex}{\tiny$\bullet$}] If $b_k-b_l=1$ then $\beta\in L(\ga)$ unless $i>c_k$ and $i\leq c_l$ with $\su_{\la}^{-1}\sy_{\ga}^{-1}\beta<0$, but by Lemma~\ref{lem:conditions1} if $i>c_k$ and $i\leq c_l$ then $\su_{\la}^{-1}\sy_{\ga}^{-1}\beta>0$. Thus $\beta\in L(\ga)$ whenever $b_k-b_l=1$. 
\item[\raisebox{0.35ex}{\tiny$\bullet$}] If $b_k-b_l\geq 2$ then $\langle\ga,\beta\rangle>0$ and so $\beta\in L(\ga)$. 
\end{compactenum}
It follows from the above observations that if $b_k=b_l$ then $\beta\in L(\ga)$ if and only if either $c_l<i\leq c_k$, or
$i\leq c_k$ and $i\leq c_l$ with $c_k^*\leq c_l^*$, or
$i>c_k$ and $i>c_l$ with $c_l\leq c_k$, and this in turn is equivalent to the inequalities stated in the lemma (noting that the situation $c_k<c_l$ and $c_k^*\leq c_l^*$ is impossible, since it implies that $\la_k\leq \la_l-(c_l-c_k)<\la_l$). 

For $R(\ga)$, by Proposition~\ref{prop:bijectionP} and the fact that $\sw_{\la}\beta=e_{\la[k,\la_k-i+1]}-e_{\la[l,\la_l-i+1]}$ we have
\begin{align*}
\langle\sw_{\la}\gamma,\beta\rangle&=\begin{cases}
b_k-b_l&\text{if $i\leq c_k^*$ and $i\leq  c_l^*$, or $i>c_k^*$ and $i> c_l^*$}\\
b_k-b_l+1&\text{if $i> c_k^*$ and $i\leq c_l^*$}\\
b_k-b_l-1&\text{if $i\leq c_k^*$ and $i> c_l^*$}
\end{cases}
\end{align*}
and the analysis is then completely analogous to the $L(\ga)$ case, making use of Lemma~\ref{lem:conditions2}.
\end{proof}

\begin{cor}\label{cor:de}
Let $\ga\in P^{(\la)}$. For $1\leq k<l\leq r(\la)$ we have
$$
|L(\ga)\cap A_{k,l}|+|R(\ga)\cap A_{k,l}|=\la_l. 
$$
\end{cor}

\begin{proof}
Write $\gamma+Q_{\la}=\sum a_j\tilde{e}_j$ with $a_j=\la_jb_j+c_j$ as before. By Proposition~\ref{prop:classifyLandR} if $|b_k-b_l|\geq 1$ then $|L(\ga)\cap A_{k,l}|+|R(\ga)\cap A_{k,l}|=\la_l$. Moreover, if $b_k-b_l=0$ then 
$$
|L(\ga)\cap A_{k,l}|=\begin{cases}
\la_l-c_l&\text{if $b_k-b_l=0$ with $c_k\geq c_l$ and $c_k^*>c_l^*$}\\
\la_l&\text{if $b_k-b_l=0$ with $c_k\geq c_l$ and $c_k^*\leq c_l^*$}\\
0&\text{otherwise},
\end{cases}
$$
and 
$$
|R(\ga)\cap A_{k,l}|=\begin{cases}
c_l&\text{if $b_k-b_l=0$ with $c_k\geq  c_l$ and $c_k^*>c_l^*$}\\
\la_l&\text{if $b_k-b_l=0$ with $c_k< c_l$ and $c_k^*> c_l^*$}\\
0&\text{otherwise},
\end{cases}
$$
and hence the result (again, noting that the situation $c_k<c_l$ and $c_k^*\leq c_l^*$ is impossible). 
\end{proof}

\begin{thm}\label{thm:doublecoset}
Let $\gamma\in P^{(\la)}$. There exist $x,y\in W_{\la'}$ such that 
$\sm_{\gamma}=x\su_{\la}^{-1}\tau_{\gamma}\su_{\la}y$ with 
$$
\ell(\sm_{\gamma})=\ell(x)+\ell(\su_{\la}^{-1}\tau_{\gamma}\su_{\la})+\ell(y)\quad\text{and}\quad \ell(x)+\ell(y)=\ell(\sw_{\la'}).
$$
\end{thm}

\begin{proof}
By (\ref{eq:decomposeu}) and Corollary~\ref{cor:de} we have 
$$
|L(\ga)|+|R(\ga)|=\sum_{1\leq k<l\leq r(\la)}\big(|L(\ga)\cap A_{k,l}|+|R(\ga)\cap A_{k,l}|\big)=\sum_{1\leq k<l\leq r(\la)}\la_l=\sum_{l\geq 1}(l-1)\la_l,
$$
and the latter equals~$\ell(\sw_{\la'})$. The result now follows from Proposition~\ref{prop:doublecoset2}. 
\end{proof}

\subsection{The length of $\su_{\la}^{-1}\tau_{\ga}\su_{\la}$}\label{sec:length}

In this section we compute the length of $\su_{\la}^{-1}\tau_{\ga}\su_{\la}$ (see Theorem~\ref{thm:lengthconj}). We begin with some preliminary observations. For $1\leq k<l\leq r(\la)$ define
\begin{align*}
\beta^+(k,l;i,j)&=e_{\la[k,i]}-e_{\la[l,j]}&&\text{for $1\leq i\leq j\leq  \la_l$}\\
\beta^-(k,l;i,j)&=e_{\la[l,j]}-e_{\la[k,i]}&&\text{for $1\leq j\leq \la_l$ and $j<i\leq \la_k$}.
\end{align*}
Let $B_{k,l}=B_{k,l}^+\cup B_{k,l}^-$, where
\begin{align*}
B_{k,l}^+=\{\beta^+(k,l;i,j)\mid 1\leq i\leq j\leq \la_l\}\quad\text{and}\quad
B_{k,l}^-=\{\beta^-(k,l;i,j)\mid 1\leq j\leq \la_l\text{ and }j<i\leq \la_k\}.
\end{align*}
By Lemma~\ref{lem:imageroots} the set $\su_{\la}\Phi^+\backslash\Phi_{\la}$ is the disjoint union of the sets $B_{k,l}$ over $1\leq k<l\leq r(\la)$. Let
$$
\Phi_{k,l}^+=B_{k,l}^+\cup(-B_{k,l}^-). 
$$
Note that $\Phi_{k,l}^+$ consists precisely of the roots $e_p-e_q\in\Phi^+$ with $p$ in row $k$ of $\bt_r(\la)$ and $q$ in row $l$ of $\bt_r(\la)$. Thus $\Phi^+\backslash\Phi_{\la}$ is the disjoint union of the sets $\Phi_{k,l}^+$ over $1\leq k<l\leq r(\la)$. 

As in the previous section, we write $\ga=\sum_{m=1}^{r(\la)} a_m\tilde{e}_m$, where $a_m=\la_mb_m+c_m$ with $0\leq c_m<\la_m$ for $1\leq m\leq r(\la)$. Let $c_m^*=\la_m-c_m$.

\begin{lemma}\label{lem:conditions3} Let $\ga\in P^{(\la)}$ and let $\beta^+=\beta^+(k,l;i,j)$ {$($}respectively $\beta^-=\beta^-(k,l;i,j)${$)$}. We have $\su_{\la}^{-1}\sy_{\gamma}^{-1}\beta^+>0$ {$($}respectively $\su_{\la}^{-1}\sy_{\gamma}^{-1}\beta^->0${$)$} if and only if either:
\begin{compactenum}[$(a)$]
\item $i\leq c_k$, $j\leq c_l$ with $i+c_k^*\leq j+c_l^*$ {$($}respectively $i\leq c_k$, $j\leq c_l$ with $j+c_l^*<i+c_k^*${$)$}, or
\item $i\leq c_k$, $j>c_l$ with $i+c_k^*\leq j-c_l$ {$($}respectively $i\leq c_k$, $j>c_l${$)$}, or
\item $i>c_k$, $j\leq c_l$ {$($}respectively $i>c_k$, $j\leq c_l$ with $j+c_l^*<i-c_k${$)$}, or
\item $i>c_k$, $j>c_l$ with $i-c_k\leq j-c_l$ {$($}respectively $i>c_k$, $j>c_l$ with $j-c_l<i-c_k${$)$}. 
\end{compactenum}
\end{lemma}

\begin{proof}
As in Lemma~\ref{lem:conditions1}, a direct calculation gives $\sy_{\ga}^{-1}\beta^+=e_{\la[k,i']}-e_{\la[l,j']}$ and $\sy_{\ga}^{-1}\beta^-=e_{\la[l,j']}-e_{\la[k,i']}$, where
\begin{align*}
i'&=\begin{cases}
c_k^*+i&\text{if $i\leq c_k$}\\
i-c_k&\text{if $i>c_k$}
\end{cases}&
j'&=\begin{cases}
c_l^*+j&\text{if $j\leq c_l$}\\
j-c_l&\text{if $j>c_l$}.
\end{cases}
\end{align*}
We have $\su_{\la}^{-1}\sy_{\ga}^{-1}\beta^+>0$ if and only if $i'\leq j'$, and $\su_{\la}^{-1}\sy_{\ga}^{-1}\beta^->0$ if and only if $j'<i'$.
\end{proof}

\begin{lemma}\label{lem:samerow}
If $\ga\in P^{(\la)}$ and $\alpha\in\Phi^+$ with $\su_{\la}\alpha\in\Phi_{\la}$ then
$
\langle\gamma,\su_{\la}\alpha\rangle=\chi^-(\su_{\la}^{-1}\sy_{\gamma}^{-1}\su_{\la}\alpha).
$
\end{lemma}

\begin{proof}
Let $\beta=\su_{\la}\alpha$. Since $\beta\in\Phi_{\la}$ we have $\beta=e_{\la[k,i]}-e_{\la[k,j]}$ for some $1\leq k\leq r(\la)$, with $1\leq i<j\leq \la_k$ (see Lemma~\ref{lem:imageroots}). Let $\gamma+Q_{\la}=\sum_{j=1}^{r(\la)}a_j\tilde{e}_j$. Write $a_k=\la_k b_k+c_k$ with $0\leq c_k<\la_k$. By Proposition~\ref{prop:bijectionP} we have
$$
\langle\ga,\beta\rangle=\begin{cases}
0&\text{if either $i,j\leq c_k$ or $i,j>c_k$}\\
1&\text{if $i\leq c_k<j$}.
\end{cases}
$$
We claim that in the first case $\su_{\la}^{-1}\sy_{\ga}^{-1}\beta>0$ while in the second case $\su_{\la}^{-1}\sy_{\ga}^{-1}\beta<0$. For if $i,j\leq c_k$ then as in Lemma~\ref{lem:conditions3} we have $\sy_{\ga}^{-1}\beta=e_{\la[k,i+c_k^*]}-e_{\la[k,j+c_k^*]}$ and so $\su_{\la}^{-1}\sy_{\ga}^{-1}\beta>0$ (as $i+c_k^*<j+c_k^*$, see Lemma~\ref{lem:poscriteria} for a similar argument). Similarly if $i,j>c_k$. On the other hand, if $i\leq c_k<j$ then we have $\sy_{\ga}^{-1}\su_{\la}\beta=e_{\la[k,i+c_k^*]}-e_{\la[k,j-c_k]}$, and since $i+c_k^*>j-c_k$ (because $i+\la_k>j$) we have $\su_{\la}^{-1}\sy_{\ga}^{-1}\beta<0$. Hence the result.
\end{proof}

\begin{thm}\label{thm:lengthconj}
Let $\gamma\in P^{(\la)}$ with $\gamma+Q_{\la}=\sum_{k=1}^{r(\la)}a_k\tilde{e}_k$ where $a_k=\la_kb_k+c_k$ with $0\leq c_k<\la_k$, and let $c_k^*=\la_k-c_k$. Then 
$$
\ell(\su_{\la}^{-1}\tau_{\gamma}\su_{\la})=\sum_{1\leq k<l\leq r(\la)}z(\ga,k,l) \quad\text{where}\quad z(\ga,k,l)=\begin{cases}
|\la_la_k-\la_ka_l|&\text{if $b_k\neq b_l$}\\
c_l^*|c_k-c_l|+c_l|c_k^*-c_l^*|&\text{if $b_k=b_l$.}\\
\end{cases}
$$
\end{thm}

\begin{proof}
For $\beta\in\Phi$ let 
$
h(\ga,\beta)=\langle\gamma,\beta\rangle-\chi^-(\su_{\la}^{-1}\sy_{\gamma}^{-1}\beta),
$
and so since $\su_{\la}^{-1}\tau_{\ga}\su_{\la}=t_{\su_{\la}^{-1}\ga}\su_{\la}^{-1}\sy_{\ga}\su_{\la}$ we have
$$
\ell(\su_{\la}^{-1}\tau_{\gamma}\su_{\la})=\sum_{\beta\in\su_{\la}\Phi^+}|h(\ga,\beta)|.
$$
By Lemma~\ref{lem:samerow} we have $h(\ga,\beta)=0$ if $\beta\in\Phi_{\la}$, so we can omit these roots from the analysis. Then $\su_{\la}\Phi^+\backslash \Phi_{\la}$ is a disjoint union of the sets $B_{k,l}$ with $1\leq k<l\leq r(\la)$, and hence
$$
\ell(\su_{\la}^{-1}\tau_{\ga}\su_{\la})=\sum_{1\leq k<l\leq r(\la)}\sum_{\beta\in B_{k,l}}|h(\ga,\beta)|.
$$ 
Thus to prove the theorem it is sufficient to prove that
\begin{align}\label{eq:firstclaim}
\sum_{\beta\in B_{k,l}}|h(\ga,\beta)|=z(\ga,k,l).
\end{align}

Let $\beta^+=\beta^+(k,l;i,j)\in B_{k,l}^+$, where $1\leq i\leq j\leq \la_l$, and let $\beta^-=\beta^-(k,l;i,j)\in B_{k,l}^-$, where $1\leq j\leq \la_l$ and $j<i\leq \la_k$. Since
\begin{align}\label{eq:pluscase}
\langle\ga,\beta^+\rangle=\begin{cases}
b_k-b_l&\text{if $i\leq c_k$, $j\leq c_l$ or $i>c_k$, $j>c_l$}\\
b_k-b_l+1&\text{if $i\leq c_k$, $j>c_l$}\\
b_k-b_l-1&\text{if $i>c_k$, $j\leq c_l$}.
\end{cases}
\end{align}
Thus it follows that if $b_k-b_l\leq -1$ then $h(\ga,\beta^+)\leq 0$, and if $b_k-b_l\geq 1$ then $h(\ga,\beta^+)\geq 0$ (note that if $b_k-b_l=1$ with $i>c_k$ and $j\leq c_l$ then by Lemma~\ref{lem:conditions3} we have $\su_{\la}^{-1}\sy_{\ga}^{-1}\beta^+>0$ and so $h(\ga,\beta^+)=-\chi^-(\su_{\la}^{-1}\sy_{\ga}^{-1}\beta^+)=0$).

Similarly, since 
\begin{align}\label{eq:negcase}
\langle\ga,\beta^-\rangle=\begin{cases}
b_l-b_k&\text{if $i\leq c_k$, $j\leq c_l$ or $i>c_k$, $j>c_l$}\\
b_l-b_k+1&\text{if $i>c_k$, $j\leq c_l$}\\
b_l-b_k-1&\text{if $i\leq c_k$, $j> c_l$}
\end{cases}
\end{align}
it follows that if $b_k-b_l\geq 1$ then $h(\ga,\beta^-)\leq 0$, and if $b_k-b_l\leq -1$ then $h(\ga,\beta^-)\geq 0$ (note that if $b_k-b_l=-1$ with $i\leq c_k$ and $j>c_l$ then by Lemma~\ref{lem:conditions3} we have $h(\ga,\beta^-)=0$).

To prove~(\ref{eq:firstclaim}) we consider the following cases. 
\smallskip

\noindent \textit{Case 1: Suppose that $b_k-b_l\geq 1$.} The above observations give
\begin{align*}
\sum_{\beta\in B_{k,l}}|h(\ga,\beta)|&=\sum_{\beta\in B_{k,l}^+}h(\ga,\beta)-\sum_{\beta\in B_{k,l}^-}h(\ga,\beta)\\
&=\sum_{\beta\in B_{k,l}^+}\langle\ga,\beta\rangle-\sum_{\beta\in B_{k,l}^-}\langle\ga,\beta\rangle-\sum_{\beta\in B_{k,l}^+}\chi^-(\su_{\la}^{-1}\sy_{\ga}^{-1}\beta)+\sum_{\beta\in B_{k,l}^-}\chi^-(\su_{\la}^{-1}\sy_{\ga}^{-1}\beta). 
\end{align*}
The first two terms combine to give $\langle\ga,2\rho_{k,l}\rangle$, where
$$
\rho_{k,l}=\frac{1}{2}\sum_{\beta\in B_{k,l}^+\cup(-B_{k,l}^-)}\beta=\frac{1}{2}\sum_{\beta\in\Phi_{k,l}^+}\beta,
$$
and so $\sum_{\beta\in B_{k,l}}|h(\ga,\beta)|=\langle\ga,2\rho_{k,l}\rangle-S$ where $S=\sum_{\beta\in B_{k,l}^+}\chi^-(\su_{\la}^{-1}\sy_{\ga}^{-1}\beta)-\sum_{\beta\in B_{k,l}^-}\chi^-(\su_{\la}^{-1}\sy_{\ga}^{-1}\beta)$. We claim that $S=0$. To see this, note that since $B_{k,l}^+=\Phi_{k,l}^+\backslash\Phi(\su_{\la})$ and $-B_{k,l}^-=\Phi(\su_{\la})\cap \Phi_{k,l}$ we have
\begin{align*}
S
&=\sum_{\beta\in \Phi_{k,l}^+\backslash \Phi(\su_{\la})}\chi^-(\su_{\la}^{-1}\sy_{\ga}^{-1}\beta)-\sum_{\beta\in \Phi(\su_{\la})\cap\Phi_{k,l}}\chi^-(-\su_{\la}^{-1}\sy_{\ga}^{-1}\beta)\\
&=\sum_{\beta\in\Phi_{k,l}^+}\chi^-(\su_{\la}^{-1}\sy_{\ga}^{-1}\beta)-\sum_{\beta\in\Phi(\su_{\la})\cap\Phi_{k,l}}\big[\chi^-(\su_{\la}^{-1}\sy_{\ga}^{-1}\beta)+\chi^-(-\su_{\la}^{-1}\sy_{\ga}^{-1}\beta)\big].
\end{align*}
Each term in the second sum is $1$, and hence 
$
S=|\Phi(\sy_{\ga}\su_{\la})\cap \Phi_{k,l}|-|\Phi(\su_{\la})\cap \Phi_{k,l}|. 
$
Since $\ell(\sy_{\ga}\su_{\la})=\ell(\su_{\la})+\ell(\sy_{\ga})$ we have $\Phi(\sy_{\ga}\su_{\la})=\Phi(\sy_{\ga})\sqcup\sy_{\ga}\Phi(\su_{\la})$ and so $|\Phi(\sy_{\ga}\su_{\la})\cap \Phi_{k,l}|=|\Phi(\sy_{\ga})\cap \Phi_{k,l}|+|\Phi(\su_{\la})\cap \sy_{\ga}^{-1}\Phi_{k,l}|$. But $\Phi(\sy_{\ga})\cap \Phi_{k,l}=\emptyset$ (as $\Phi(\sy_{\ga})\subseteq\Phi_{\la}^+$) and $\sy_{\ga}^{-1}\Phi_{k,l}=\Phi_{k,l}$, and so $S=0$. Thus
$$
\sum_{\beta\in B_{k,l}}|h(\ga,\beta)|=\langle\ga,2\rho_{k,l}\rangle,
$$
but $2\rho_{k,l}=\la_l\sum_{i=1}^{\la_k}e_{\la[k,i]}-\la_k\sum_{j=1}^{\la_l}e_{\la[l,j]}$ and so by Proposition~\ref{prop:bijectionP} $\langle\ga,2\rho_{k,l}\rangle=\la_la_k-\la_ka_l$, and hence~(\ref{eq:firstclaim}) holds. 
\smallskip

\noindent\textit{Case 2: Suppose that $b_k-b_l\leq -1$.} Then the above observations, along with the calculations in Case 1, give
\begin{align*}
\sum_{\beta\in B_{k,l}}|h(\ga,\beta)|&=-\sum_{\beta\in B_{k,l}^+}h(\ga,\beta)+\sum_{\beta\in B_{k,l}^-}h(\ga,\beta)=-\langle\ga,2\rho_{k,l}\rangle=\la_ka_l-\la_la_k
\end{align*}
and so again~(\ref{eq:firstclaim}) holds. 
\smallskip

\noindent\textit{Case 3: Suppose that $b_k-b_l=0$.} We first note that in this case 
\begin{align}\label{eq:zcases}
z(\ga,k,l)=\begin{cases}
\la_kc_l-\la_lc_k&\text{if $c_k\leq c_l$ (and hence $c_k^*\geq c_l^*$)}\\
\la_lc_k-\la_kc_l&\text{if $c_k>c_l$ and $c_k^*\leq c_l^*$}\\
c_l^*(c_k-c_l)+c_l(c_k^*-c_l^*)&\text{if $c_k>c_l$ and $c_k^*>c_l^*$}.
\end{cases}
\end{align}
From~(\ref{eq:pluscase}) and~(\ref{eq:negcase}) we have the following (noting that $1-\chi^-(\cdot)=\chi^+(\cdot)$ where $\chi^+(\cdot)$ is the characteristic function of $\Phi^+$). 
\begin{compactenum}
\item[\raisebox{0.35ex}{\tiny$\bullet$}] If $i\leq c_k$ and $j\leq c_l$ then $|h(\ga,\beta^+)|=\chi^-(\su_{\la}^{-1}\sy_{\ga}^{-1}\beta^+)$ and $|h(\ga,\beta^-)|=\chi^-(\su_{\la}^{-1}\sy_{\ga}^{-1}\beta^-)$.
\item[\raisebox{0.35ex}{\tiny$\bullet$}] If $i>c_k$ and $j>c_l$ then $|h(\ga,\beta^+)|=\chi^-(\su_{\la}^{-1}\sy_{\ga}^{-1}\beta^+)$ and $|h(\ga,\beta^-)|=\chi^-(\su_{\la}^{-1}\sy_{\ga}^{-1}\beta^-)$.
\item[\raisebox{0.35ex}{\tiny$\bullet$}] If $i\leq c_k$, $j>c_l$ then $|h(\ga,\beta^+)|=\chi^+(\su_{\la}^{-1}\sy_{\ga}^{-1}\beta^+)$ and $|h(\ga,\beta^-)|=1$ (see Lemma~\ref{lem:conditions3}).
\item[\raisebox{0.35ex}{\tiny$\bullet$}] If $i>c_k$ and $j\leq c_l$ then $|h(\ga,\beta^+)|=1$ (see Lemma~\ref{lem:conditions3}) and $|h(\ga,\beta^-)|=\chi^+(\su_{\la}^{-1}\sy_{\ga}^{-1}\beta^-)$. 
\end{compactenum}
Thus, by Lemma~\ref{lem:conditions3}, we have
$
\sum_{\beta\in B_{k,l}}|h(\ga,\beta)|=|X_1|+\cdots+|X_8|, 
$
where
\begin{align*}
X_1&=\{(i,j)\mid 1\leq i\leq j\leq \la_l,\,i\leq c_k,\,j\leq c_l,\,i+c_k^*>j+c_l^*\}\\
X_2&=\{(i,j)\mid 1\leq i\leq j\leq \la_l,\,i> c_k,\,j> c_l,\,i-c_k>j-c_l\}\\
X_3&=\{(i,j)\mid 1\leq i\leq j\leq \la_l,\,i\leq c_k,\,j> c_l,\,i+c_k^*\leq j-c_l\}\\
X_4&=\{(i,j)\mid 1\leq i\leq j\leq \la_l,\,i> c_k,\,j\leq c_l\}\\
X_5&=\{(i,j)\mid 1\leq j\leq \la_l,\,j<i\leq \la_k,\,i\leq c_k,\,j\leq c_l,\,i+c_k^*\leq j+c_l^*\}\\
X_6&=\{(i,j)\mid 1\leq j\leq \la_l,\,j<i\leq \la_k,\,i> c_k,\,j> c_l,\,i-c_k\leq  j-c_l\}\\
X_7&=\{(i,j)\mid 1\leq j\leq \la_l,\,j<i\leq \la_k,\,i\leq c_k,\,j> c_l\}\\
X_8&=\{(i,j)\mid 1\leq j\leq \la_l,\,j<i\leq \la_k,\,i> c_k,\,j\leq c_l,\,i-c_k> j+c_l^*\}.
\end{align*}
We compute the cardinalities of these sets (to do so it helps to draw a sketch of the corresponding regions in the $(i,j)$-plane). The analysis divides into the following sub-cases. 
\begin{compactenum}[$(a)$]
\item Suppose that $c_k\leq c_l$ (then necessarily $c_k^*=\la_k-c_k\geq \la_l-c_l=c_l^*$). Then we compute
\begin{align*}
|X_1|&=\begin{cases}
c_k(c_k^*-c_l^*)-\frac{1}{2}(\la_k-\la_l)(\la_k-\la_l-1)&\text{if $c_l>c_k^*-c_l^*$}\\
c_k(c_l+1)-\frac{1}{2}c_k(c_k+1)&\text{if $c_l\leq c_k^*-c_l^*$}
\end{cases}\\
|X_8|&=
\begin{cases}
\frac{1}{2}(c_k^*-c_l^*-1)(c_k^*-c_l^*)&\text{if $c_l>c_k^*-c_l^*$}\\
c_l(c_k^*-c_l^*)-\frac{1}{2}c_l(c_l+1)&\text{if $c_l\leq c_k^*-c_l^*$,}
\end{cases}
\end{align*}
and $|X_2|=(c_l-c_k)(\la_l-c_l)$, $|X_4|=\frac{1}{2}(c_l-c_k)(c_l-c_k+1)$, and $|X_3|=|X_5|=|X_6|=|X_7|=0$. Summing gives $\sum_{\beta\in B_{k,l}}|h(\ga,\beta)|=\la_kc_l-\la_lc_k$ as required (see~(\ref{eq:zcases})).
\item Suppose that $c_k>c_l$ and $c_k^*\leq c_l^*$. Then $|X_1|=|X_2|=|X_4|=|X_8|=0$, $|X_3|=\frac{1}{2}(c_l^*-c_k^*)(c_l^*-c_k^*+1)$, $|X_5|=c_l(c_l^*-c_k^*)$, and
\begin{align*}
|X_6|&=\begin{cases}
(\la_l-c_k)(c_k-c_l)+\frac{1}{2}(c_k-c_l)(c_k-c_l+1)-\frac{1}{2}(c_l^*-c_k^*)(c_l^*-c_k^*+1)&\text{if $c_k\leq\la_l$}\\
\frac{1}{2}c_l^*(c_l^*+1)-\frac{1}{2}(c_l^*-c_k^*)(c_l^*-c_k^*+1)&\text{if $c_k>\la_l$}
\end{cases}\\
|X_7|&=\begin{cases}
\frac{1}{2}(c_k-c_l-1)(c_k-c_l)&\text{if $c_k\leq \la_l$}\\
\frac{1}{2}(c_k-c_l-1)(c_k-c_l)-\frac{1}{2}(c_k-\la_l-1)(c_k-\la_l)&\text{if $c_k>\la_l$.}
\end{cases}
\end{align*}
Thus $\sum_{\beta\in B_{k,l}}|h(\ga,\beta)|=\la_lc_k-\la_kc_l$ as required.
\item Suppose that $c_k>c_l$ and $c_k^*>c_l^*$. Then $|X_2|=|X_3|=|X_4|=|X_5|=0$, and 
\begin{align*}
|X_1|&=\begin{cases}
\frac{1}{2}c_l(c_l+1)&\text{if $c_l\leq c_k^*-c_l^*$}\\
(c_l-(c_k^*-c_l^*))(c_k^*-c_l^*)+\frac{1}{2}(c_k^*-c_l^*)(c_k^*-c_l^*+1)&\text{if $c_l>c_k^*-c_l^*$}
\end{cases}\\
|X_6|&=\begin{cases}
(\la_l-c_k)(c_k-c_l)+\frac{1}{2}(c_k-c_l)(c_k-c_l+1)&\text{if $c_k\leq\la_l$}\\
\frac{1}{2}c_l^*(c_l^*+1)&\text{if $c_k>\la_l$}
\end{cases}\\
|X_7|&=\begin{cases}
\frac{1}{2}(c_k-c_l-1)(c_k-c_l)&\text{if $c_k\leq \la_l$}\\
\frac{1}{2}(c_k-c_l-1)(c_k-c_l)-\frac{1}{2}(c_k-\la_l-1)(c_k-\la_l)&\text{if $c_k>\la_l$}
\end{cases}\\
|X_8|&=\begin{cases}
\frac{1}{2}c_l(c_l-1)+c_l(c_k^*-\la_l)&\text{if $c_l\leq c_k^*-c_l^*$}\\
\frac{1}{2}(c_k^*-c_l^*-1)(c_k^*-c_l^*)&\text{if $c_l>c_k^*-c_l^*$.}
\end{cases}
\end{align*}
Thus $|X_1|+|X_8|=c_l(c_k^*-c_l^*)$ (in both cases) and $|X_6|+|X_7|=c_l^*(c_k-c_l)$ (in both cases), and hence $\sum_{\beta\in B_{k,l}}|h(\ga,\beta)|=c_l^*(c_k-c_l)+c_l(c_k^*-c_l^*)$ as required, completing the proof.\qedhere
\end{compactenum}
\end{proof}

\begin{remark}\label{rem:simplifications}
Note that by~(\ref{eq:zcases}) we have $z(\ga,k,l)=|\la_la_k-\la_ka_l|$ in all cases except when $b_k=b_l$ with $c_k>c_l$ and $c_k^*>c_l^*$. Moreover, if $\la=(d^r)$ (that is one block of row length $d$ with $r$ rows) then the case $c_k>c_l$ and $c_k^*>c_l^*$ cannot occur, and thus $z(\ga,k,l)=d|a_k-a_l|$ in all cases. Thus in the $1$-block case we have the simplified formula
$$
\ell(\su_{\la}^{-1}\tau_{\ga}\su_{\la})=d\sum_{1\leq k<l\leq r}|a_k-a_l|.
$$
\end{remark}

\begin{example}
Consider the $\la=[4,2]$ case with $\ga=2\tilde{e}_1+\tilde{e}_2$. The $b_1=b_2=0$, $c_1=2$, $c_2=1$, $c_1^*=2$, and $c_2^*=1$, and so $c_1>c_2$ and $c_1^*>c_2^*$. The formula from Theorem~\ref{thm:lengthconj} gives $\ell(\su_{\la}^{-1}\tau_{\ga}\su_{\la})=c_2^*(c_1-c_2)+c_2(c_1^*-c_2^*)=2$. Indeed a direct calculation gives $\su_{\la}^{-1}\tau_{\ga}\su_{\la}=s_1s_4\sigma^3$. 
\end{example}

\subsection{Monotonicity of $\ell(\sm_{\ga})$ with respect to $\peq_{\la}$}\label{sec:monotone}

The following theorem will be used in the important Proposition~\ref{prop:path}. The broken symmetry in the definition of $z(\ga,k,l)$ (that is, the existence of two cases in the definition) complicates the proof.

\begin{thm}\label{thm:domlength}
If $\gamma,\gamma'\in P^{(\la)}_+$ with $\gamma+Q_{\la}\peq_{\la}\gamma'+Q_{\la}$ then $\ell(\sm_{\gamma})\leq\ell(\sm_{\gamma'})$ with equality if and only if $\ga=\ga'$. 
\end{thm}

\begin{proof}
By Theorem~\ref{thm:doublecoset} it suffices to prove that 
$
\ell(\su_{\la}^{-1}\tau_{\ga}\su_{\la})<\ell(\su_{\la}^{-1}\tau_{\ga'}\su_{\la})
$ with equality if and only if $\ga=\ga'$. Let $M(\ga)=\ell(\su_{\la}^{-1}\tau_{\ga}\su_{\la})$ and let $B_1,\ldots,B_t\subseteq\{1,2,\ldots,r(\la)\}$ be the blocks of $\bt_r(\la)$ (for example, if $\la=(6,6,4,4,4,2,1,1)$ as in Example~\ref{ex:running2} then $B_1=\{1,2\}$, $B_2=\{3,4,5\}$, $B_3=\{6\}$, and $B_7=\{7,8\}$). By Theorem~\ref{thm:lengthconj} we have
\begin{align}\label{eq:blocklength}
M(\ga)=\sum_{p=1}^t\sum_{k,l\in B_p,\,k<l}z(\ga,k,l)+\sum_{1\leq p<q\leq t}\sum_{k\in B_p,\,l\in B_q}z(\ga,k,l),
\end{align}
and similarly for $M(\ga')$. Write $\ga'+Q_{\la}=\ga+\epsilon+Q_{\la}$ with $\epsilon\in Q^{\la}_+$. From Lemma~\ref{lem:membershipQ^la} we have $\epsilon=\sum_{i=1}^{r(\la)}d_i\tilde{e}_i$ where $\sum_{i\in B_p}d_i=0$ for all $1\leq p\leq t$ and $d_1+\cdots+d_i\geq 0$ for all $1\leq i\leq r(\la)$. Thus $\epsilon=\epsilon_1+\cdots+\epsilon_t$ where $\epsilon_p=\sum_{i\in B_p}d_i\tilde{e}_i\in Q^{\la}_+$, and for any $1\leq p\leq t$ we have $\ga+\epsilon_p+Q_{\la}\in (P/Q_{\la})_+$ with $\ga+Q_{\la}\peq_{\la}\ga+\epsilon_p+Q_{\la}\peq_{\la}\ga'+Q_{\la}$. Thus it suffices to prove the result in the case $\ga'+Q_{\la}=\ga+\epsilon_p+Q_{\la}$ for some $1\leq p\leq t$ (that is, $\ga+Q_{\la}$ and $\ga'+Q_{\la}$ only differ in one block), and then by well known properties of the dominance order on the symmetric group (see, for example, \cite[Corollary~2.7]{Ste:98}) it suffices to prove the result when $\epsilon_p=\tilde{e}_i-\tilde{e}_j$ for some $i,j\in B_p$ with $i<j$. 

Thus we henceforth assume that $\ga+Q_{\la},\ga'+Q_{\la}\in (P/Q_{\la})_+$ with $\ga'+Q_{\la}=\ga+\tilde{e}_i-\tilde{e}_j+Q_{\la}$ for some $i,j\in B_p$ with $i<j$. Write $\ga+Q_{\la}=\sum_{m=1}^{r(\la)}a_m\tilde{e}_m$ and $\ga'+Q_{\la}=\sum_{m=1}^{r(\la)}a_m'\tilde{e}_m$. Thus $a_m'=a_m$ unless $m\in\{i,j\}$, and $a_i'=a_i+1$ and $a_j'=a_j-1$. Note that if $b_k> b_l$ then we have 
\begin{align}\label{eq:aineq}
\la_la_k-\la_ka_l=\la_k\la_l(b_k-b_l)+(\la_lc_k-\la_kc_l)>\la_k\la_l-\la_k\la_l=0,
\end{align}
and similarly if $b_k<b_l$ then $\la_la_k-\la_ka_l<0$.

Let $B_{<p}=B_1\cup\cdots\cup B_{p-1}$ and $B_{>p}=B_{p+1}\cup\cdots\cup B_t$. Since $z(\ga',k,l)=z(\ga,k,l)$ whenever $k,l\notin\{i,j\}$ the formula~(\ref{eq:blocklength}) gives
\begin{align*}
M(\ga')-M(\ga)&=\sum_{k,l\in B_p,\,k<l}\left(z(\ga',k,l)-z(\ga,k,l)\right)\\
&\quad+\sum_{k\in B_{<p}}\left(z(\ga',k,i)-z(\ga,k,i)+z(\ga',k,j)-z(\ga,k,j)\right)\\
&\quad+\sum_{l\in B_{>p}}\left(z(\ga',i,l)-z(\ga,i,l)+z(\ga',j,l)-z(\ga,j,l)\right).
\end{align*}
We must show that $M(\ga')-M(\ga)>0$. We analyse each sum separately. Let $\la_0=\la_k$ for any $k\in B_p$ (so $\la_0$ is the length of the rows in block $B_p$). 

Consider the first sum. For $k,l\in B_p$ with $k<l$ we have (by dominance within the block) $z(\ga,k,l)=\la_0(a_k-a_l)$ and $z(\ga',k,l)=\la_0(a_k'-a_l')$. Then, since $a_k'=a_k$ and $a_l=a_l'$ whenever $k,l\notin\{i,j\}$ and $a_i'=a_i+1$ and $a_j'=a_j-1$ we see that
$$
\sum_{k,l\in B_p,\,k<l}\left(z(\ga',k,l)-z(\ga,k,l)\right)=2\la_0(j-i)>0. 
$$

Now consider the second and third sums in the above formula for $M(\ga')-M(\ga)$. To prove the theorem it is sufficient, given the strict inequality for the first sum, to show that each summand is nonnegative. We have $a_i'=a_i+1$ and $a_j'=a_j-1$, and 
\begin{align*}
b_{i}'&=\begin{cases}
b_{i}&\text{if $c_{i}<\la_{i}-1$}\\
b_{i}+1&\text{if $c_{i}=\la_{i}-1$}
\end{cases}&
c_{i}'&=\begin{cases}
c_{i}+1&\text{if $c_{i}<\la_{i}-1$}\\
0&\text{if $c_{i}=\la_{i}-1$}
\end{cases}\\
b_{j}'&=\begin{cases}
b_{j}&\text{if $c_{j}>0$}\\
b_{j}-1&\text{if $c_{j}=0$}
\end{cases}&
c_{j}'&=\begin{cases}
c_{j}-1&\text{if $c_{j}>0$}\\
\la_{j}-1&\text{if $c_{j}=0$.}
\end{cases}
\end{align*}

We first analyse the second sum. For $k\in B_{<p}$ let $$x_1(k)=z(\ga',k,i)-z(\ga,k,i)\quad\text{and}\quad x_2(k)=z(\ga',k,j)-z(\ga,k,j),$$ and 
set $x(k)=x_1(k)+x_2(k)$. We must show that $x(k)\geq 0$ for all $k\in B_{<p}$.

First consider $x_1(k)$. As before, let $c_m^*=\la_m-c_m$ for $1\leq m\leq r(\la)$. We claim that
\begin{align}\label{eq:multicase1}
x_1(k)&=\begin{cases}
-\la_k&\text{if $b_k>b_i$}\\
-\la_k&\text{if $b_k=b_i=b_i'$ with $c_k>c_i$ and $c_k^*<c_i^*$}\\
\la_k-2(c_i^*-c_i-1+c_k)&\text{if $b_k=b_i=b_i'$ with $c_k>c_i$ and $c_k^*\geq c_i^*$}\\
\la_k&\text{if $b_k=b_i$ with $c_k\leq c_i$}\\
\la_k-2(c_k-c_i)&\text{if $b_k=b_i$ with $b_i'=b_i+1$ and $c_k> c_i$}\\
\la_k&\text{if $b_k<b_i$}. 
\end{cases}
\end{align}
To prove~(\ref{eq:multicase1}) we consider the following cases. 
\begin{compactenum}
\item[\raisebox{0.35ex}{\tiny$\bullet$}] Suppose that $b_k>b_i$. If either $b_k\geq b_i+2$, or $b_k=b_i+1$ with $b_i'=b_i$, then $b_k>b_i,b_i'$ and so by~(\ref{eq:aineq}) 
$
x_1(k)=\left(\la_0a_k-\la_k(a_i+1)\right)-\left(\la_0a_k-\la_ka_i\right)=-\la_k.
$
If $b_k=b_i+1$ with $b_i'=b_i+1$ then $c_i=\la_0-1$, and $c_i'=0$ (so ${c_i'}^*=\la_0$). Since $b_k>b_i$ we have
$z(\ga',k,i)=|{c'_i}^*(c_k-c_i')|+|c_i'(c_k^*-{c_i'}^*)|=\la_0c_k$ and $z(\ga,k,i)=\la_0a_k-\la_ka_i$, and so 
\begin{align*}
x_1(k)&=\la_0c_k-(\la_0a_k-\la_ka_i)=\la_0\la_k(b_i-b_k)+\la_kc_i=-\la_0\la_k+\la_k(\la_0-1)=-\la_k.
\end{align*}
\item[\raisebox{0.35ex}{\tiny$\bullet$}] Suppose that $b_k=b_i$ with $b_i'=b_i$. So $c_i<\la_i-1$ and $c_i'=c_i+1$, so ${c_i'}^*=\la_0-c_i-1=c_i^*-1$. Then we have $z(\ga,k,i)=c_i^*|c_k-c_i|+c_i|c_k^*-c_i^*|$ and 
\begin{align*}
z(\ga',k,i)&=(c_i^*-1)|c_k-c_i-1|+(c_i+1)|c_k^*-c_i^*+1|.
\end{align*} 
We have the following sub-cases. Suppose first that $c_k>c_i$ and $c_k^*<c_i^*$. Then 
\begin{align*}
x_1(k)&=(c_i^*-1)(c_k-c_i-1)+(c_i+1)(c_i^*-c_k^*-1)-c_i^*(c_k-c_l)-c_i(c_i^*-c_k^*)=-\la_k.
\end{align*}
Now suppose that $c_k>c_i$ and $c_k^*\geq c_i^*$. Then
\begin{align*}
x_1(k)&=(c_i^*-1)(c_k-c_i-1)+(c_i+1)(c_k^*-c_i^*+1)-c_i^*(c_k-c_i)-c_i(c_k^*-c_i^*)\\
&=\la_k-2(c_i^*-c_i-1+c_k).
\end{align*}
Finally suppose that $c_k\leq  c_i$. Then $c_k^*=\la_k-c_k\geq \la_i-c_i=c_i^*$, and 
\begin{align*}
x_1(k)&=(c_i^*-1)(c_i+1-c_k)+(c_i+1)(c_k^*-c_i^*+1)-c_i^*(c_i-c_k)-c_i(c_k^*-c_i^*)=\la_k,
\end{align*}
completing the analysis for this case. 
\item[\raisebox{0.35ex}{\tiny$\bullet$}] Suppose that $b_k=b_i$ with $b_i'=b_i+1$. So $c_i=\la_0-1$ and $c_i^*=1$. Since $b_k=b_i<b_i'$ we have (c.f.~(\ref{eq:aineq})) 
$
z(\ga',k,i)=-\la_0a_k+\la_k(a_i+1)=\la_0c_k^*.
$ 
Since $b_k=b_i$ we have 
$$
z(\gamma,k,i)=c_i^*|c_k-c_i|+c_i|c_k^*-c_i^*|=|c_k-c_i|+(\la_0-1)(c_k^*-1),
$$
and so if $c_k> c_i$ we have $x_1(k)=\la_k-2(c_k-c_i)$ while if $c_k\leq c_i$ we have $x_1(k)=\la_k$. 
\item[\raisebox{0.35ex}{\tiny$\bullet$}] Finally suppose that $b_k\leq b_i-1$. Thus $b_k<b_i$ and $b_k<b_i'$, and so 
$$
x_1(k)=-\left(\la_0a_k-\la_k(a_i+1)\right)+\left(\la_0a_k-\la_ka_i\right)=\la_k,
$$
completing the proof of the claim. 
\end{compactenum}

We now consider $x_2(k)$. We claim that
\begin{align}\label{eq:multicase2}
x_2(k)&=\begin{cases}
\la_k&\text{if $b_k>b_j'$}\\
\la_k&\text{if $b_k=b_j'=b_j$ with $c_k>c_j'$ and $c_k^*< c_j'^*$}\\
-\la_k+2(c_j'^*-c_j'-1+c_k)&\text{if $b_k=b_j'=b_j$ with $c_k> c_j'$ and $c_k^*\geq c_j'^*$}\\
-\la_k&\text{if $b_k=b_j'$ with $c_k\leq c_j'$}\\
-\la_k+2(c_k-c_j')&\text{if $b_k=b_j'$ with $b_j=b_j'+1$ and $c_k> c_j'$}\\
-\la_k&\text{if $b_k<b_j'$}.
\end{cases}
\end{align}

The proof of the claim is very similar to the case of $x_1(k)$. The details are as follows. 
\begin{compactenum}
\item[\raisebox{0.35ex}{\tiny$\bullet$}] Suppose that $b_k>b_j'$. If $b_k>b_j$ then by~(\ref{eq:aineq}) we have 
$
x_2(k)=(\la_0a_k-\la_ka_j')-(\la_0a_k-\la_ka_j)=\la_k.
$
If $b_k\leq b_j$ then since $b_j'\in \{b_j,b_j-1\}$ we necessarily have $b_k=b_j$ with $b_j'=b_j-1$. Thus $c_j=0$ and $c_j'=\la_0-1$, and we have
$z(\gamma',k,j)=\la_ja_k-\la_k(a_j-1)=\la_0c_k+\la_k$ and 
$z(\gamma,k,j)=c_j^*|c_k-c_j|+c_j|c_k^*-c_j^*|=\la_0c_k$, and hence $x_2(k)=\la_k$. 
\item[\raisebox{0.35ex}{\tiny$\bullet$}] Suppose that $b_k=b_j'=b_j$. Then $c_j=c_j'+1$ and $c_j^*=c_j'^*-1$, and so 
\begin{align*}
x_2(k)&=c_j'^*|c_k-c_j'|+c_j'|c_k^*-c_j'^*|-(c_j'^*-1)|c_k-c_j'-1|-(c_j'+1)|c_k^*-c_j'^*+1|
\end{align*}
If $c_k>c_j'$ and $c_k^*<c_j'^*$ we compute $x_2(k)=\la_k$. If $c_k>c_j'$ and $c_k^*\geq c_j'^*$ then $x_2(k)=-\la_k+2(c_j'^*-c_j'-1+c_k)$. If $c_k\leq c_j'$ then $c_k^*=\la_k-c_k\geq \la_0-c_j'=c_j'^*$, and it follows that $x_2(k)=-\la_k$. 
\item[\raisebox{0.35ex}{\tiny$\bullet$}] Suppose that $b_k=b_j'$ and $b_j=b_j'+1$. In this case we have $c_j=0$ and $c_j'=\la_0-1$. Since $b_k<b_j$ we compute $z(\gamma,k,j)=\la_ka_j-\la_0a_k=\la_0c_k^*$, giving  
$$
x_2(k)=c_j'^*|c_k-c_j'|+c_j'|c_k^*-c_j'^*|-\la_0c_k^*=|c_k-c_j'|-c_k^*-c_j'.
$$
If $c_k> c_j'$ then $x_2(k)=-\la_k+2(c_k-c_j')$, and if $c_k\leq c_j'$ then $x_2(k)=-\la_k$. 
\item[\raisebox{0.35ex}{\tiny$\bullet$}] Finally, suppose that $b_k<b_j'$. Then $b_k<b_j$ and so $x_2(k)=-\la_k$. 
\end{compactenum}

\medskip

We now use~(\ref{eq:multicase1}) and~(\ref{eq:multicase2}) to show that $x(k)=x_1(k)+x_2(k)\geq 0$. Recall that $a_i\geq a_j$, and so $b_i\geq b_j$, and moreover if $b_i=b_j$ then $c_i\geq c_j$ (by dominance in a block). 
\begin{compactenum}[$(1)$]
\item Suppose that $x_1(k)=-\la_k$. By~(\ref{eq:multicase1}) there are two subcases.
\begin{compactenum}[(a)]
\item Suppose that $b_k>b_i$. Then $b_k>b_j'$ (for otherwise $b_i<b_k\leq b_j'\leq b_j$, contradicting dominance), and hence $x_2(k)=\la_k$ and so $x(k)=0$. 
\item Suppose that $b_k=b_i=b_i'$ with $c_k>c_i$ and $c_k^*<c_i^*$. If $b_k\leq b_j'$ then $b_j\geq b_j'\geq b_k=b_i$ which forces $b_i=b_j$ and $b_j=b_j'=b_k$. Thus $c_i\geq c_j$, and hence $c_i^*\leq c_j^*$, and since $c_j'=c_j-1$ and $c_j'^*=c_j^*+1$ we have 
$c_j'+1\leq c_i<c_k$ and $c_k^*<c_i^*\leq c_j'^*-1$. Thus $c_j'<c_k$ and $c_k^*< c_j'^*$. Thus again $x_2(k)=\la_k$, and so $x(k)=0$. 
\end{compactenum}
\item Suppose that $x_1(k)=\la_k$. Then by~(\ref{eq:multicase2}) we have $x(k)\geq 0$ in all cases (for example, if $x_2(k)=-\la_k+2(c_j'^*-c_j'-1+c_k)$ then $c_k>c_j'$ and hence $x(k)=2(c_k-c_j')+2(c_j'^*-1)\geq 0$). 
\item Suppose that $x_1(k)=\la_k-2(c_i^*-c_i-1+c_k)$. Then $b_k=b_i=b_i'$ with $c_k>c_i$ and $c_k^*\geq c_i^*$. If $b_k>b_j'$ then $x_2(k)=\la_k$ and so
$
x(k)=2(\la_k-c_i^*+c_i+1-c_k)=2(c_k^*-c_i^*+c_i+1)\geq 0.
$ 
If $b_k\leq b_j'$ then as in (1)(b) we have $b_i=b_j$, $b_j=b_j'=b_k$, $c_i\geq c_j$, and $c_i^*\leq c_j^*$. We have $c_j'=c_j-1\leq c_i-1<c_k$. If $c_k^*<c_j'^*$ we have $x_2(k)=\la_k$ and $x(k)\geq 0$ (as above), and if $c_k^*\geq c_j'^*$ then $x_2(k)=-\la_k+2(c_j'^*-c_j'-1+c_k)$ and so 
$
x(k)=2(c_j'^*-c_i^*)+2(c_i-c_j')\geq 0. 
$
\item Suppose that $x_1(k)=\la_k-2(c_k-c_i)$. Then $b_k=b_i$ with $b_i'=b_i+1$ and $c_k> c_i$. If $x_2(k)=\la_k$ then $x(k)=2(c_k^*+c_i)\geq 0$. If $x_2(k)=-\la_k+2(c_k-c_j')$ then we have $b_j=b_j'+1$ and so $c_j'=\la_0-1$, and since $b_i'=b_i+1$ we also have $c_i=\la_0-1$, and hence $x(k)=2(c_i-c_j')=0$. If $x_2(k)=-\la_k+2(c_j'^*-c_j'-1+c_k)$ then in particular $b_k=b_j=b_j'$ and so $b_i=b_j$. Thus $c_i\geq c_j=c_j'+1>c_j'$, and we have $x(k)=2c_j'^*+2(c_i-c_j')\geq 0$. We claim that $x_2(k)=-\la_k$ is impossible. To see this, we either have (a) $b_k=b_j'$ with $c_k\leq c_j'$, in which case necessarily $b_j'=b_j$ (otherwise $b_j=b_j'+1=b_k+1=b_i+1$ contradicting dominance) and hence $b_i=b_j$, but then $c_i<c_k\leq c_j'=c_j-1$ contradicting $c_i\geq c_j$, or (b) $b_k<b_j'$, but then $b_i=b_k<b_j'\leq b_j$ again contradicting dominance. 
 \end{compactenum}
\medskip

Thus we have shown that $x(k)\geq 0$ for all $k\in B_{<p}$. It remains to consider the third sum in the formula for $M(\ga')-M(\ga)$. For $l\in B_{>p}$ let $y_1(l)=z(\ga',i,l)-z(\ga,i,l)$, $y_2(l)=z(\ga',j,l)-z(\ga,j,l)$, and 
set $y(l)=y_1(l)+y_2(l)$. A similar analysis to the above shows that $y(l)\geq 0$ for all $l\in B_{>p}$, completing the proof.
\end{proof}
\end{appendix}

\bibliographystyle{plain}


\end{document}